\documentclass{amsart}
\usepackage{amssymb, amsbsy, amsmath, amstext, amsopn, verbatim}
\usepackage[all]{xy}
\usepackage{amsfonts}
\usepackage{amscd}
\usepackage{mathrsfs}
\usepackage{verbatim}
\newtheorem{thm}{Theorem} [section]
\newtheorem{lemma}[thm]{Lemma}

\newtheorem{corollary}[thm]{Corollary}
\newtheorem{prop}[thm]{Proposition}
\newtheorem{notation}[thm]{Notation}

\theoremstyle{definition}

\newtheorem{defn}[thm]{Definition}

\theoremstyle{remark}

\newtheorem{remark}[thm]{Remark}
                         
\newtheorem{claim}[thm]{Claim}

\newtheorem{fact}[thm]{Fact}

\begin{document}


\newcommand{\hs}{\mbox{\hspace{.4em}}}
\newcommand{\ds}{\displaystyle}
\newcommand{\bd}{\begin{displaymath}}
\newcommand{\ed}{\end{displaymath}}
\newcommand{\bcd}{\begin{CD}}
\newcommand{\ecd}{\end{CD}}

\newcommand{\proj}{\operatorname{Proj}}
\newcommand{\bproj}{\underline{\operatorname{Proj}}}
\newcommand{\spec}{\operatorname{Spec}}
\newcommand{\bspec}{\underline{\operatorname{Spec}}}
\newcommand{\pline}{{\mathbf P} ^1}
\newcommand{\pplane}{{\mathbf P}^2}
\newcommand{\coker}{{\operatorname{coker}}}
\newcommand{\ldb}{[[}
\newcommand{\rdb}{]]}

\newcommand{\Sym}{\operatorname{Sym}^{\bullet}}
\newcommand{\Symp}{\operatorname{Sym}}
\newcommand{\Pic}{\operatorname{Pic}}
\newcommand{\AAut}{\operatorname{Aut}}
\newcommand{\PAut}{\operatorname{PAut}}

\newcommand{\too}{\twoheadrightarrow}
\newcommand{\C}{{\mathbf C}}
\newcommand{\cA}{{\mathcal A}}
\newcommand{\cS}{{\mathcal S}}
\newcommand{\cV}{{\mathcal V}}
\newcommand{\cM}{{\mathcal M}}
\newcommand{\bA}{{\mathbf A}}
\newcommand{\cB}{{\mathcal B}}
\newcommand{\cC}{{\mathcal C}}
\newcommand{\cD}{{\mathcal D}}
\newcommand{\D}{{\mathcal D}}
\newcommand{\cs}{{\mathbf C} ^*}
\newcommand{\boldc}{{\mathbf C}}
\newcommand{\cE}{{\mathcal E}}
\newcommand{\cF}{{\mathcal F}}
\newcommand{\cG}{{\mathcal G}}
\newcommand{\G}{{\mathbf G}}
\newcommand{\fg}{{\mathfrak g}}
\newcommand{\bH}{{\mathbf H}}
\newcommand{\cH}{{\mathcal H}}
\newcommand{\cI}{{\mathcal I}}
\newcommand{\cJ}{{\mathcal J}}
\newcommand{\cK}{{\mathcal K}}
\newcommand{\cL}{{\mathcal L}}
\newcommand{\baL}{{\overline{\mathcal L}}}
\newcommand{\M}{{\mathcal M}}
\newcommand{\bM}{{\mathbf M}}
\newcommand{\bm}{{\mathbf m}}
\newcommand{\cN}{{\mathcal N}}
\newcommand{\theo}{\mathcal{O}}
\newcommand{\cP}{{\mathcal P}}
\newcommand{\cR}{{\mathcal R}}
\newcommand{\boldp}{{\mathbf P}}
\newcommand{\boldq}{{\mathbf Q}}
\newcommand{\bbL}{{\mathbf L}}
\newcommand{\cQ}{{\mathcal Q}}
\newcommand{\cO}{{\mathcal O}}
\newcommand{\Oo}{{\mathcal O}}
\newcommand{\OX}{{\Oo_X}}
\newcommand{\OY}{{\Oo_Y}}
\newcommand{\otY}{{\underset{\OY}{\ot}}}
\newcommand{\otX}{{\underset{\OX}{\ot}}}
\newcommand{\cU}{{\mathcal U}}
\newcommand{\cX}{{\mathcal X}}
\newcommand{\cW}{{\mathcal W}}
\newcommand{\boldz}{{\mathbf Z}}
\newcommand{\cZ}{{\mathcal Z}}
\newcommand{\qgr}{\operatorname{qgr}}
\newcommand{\gr}{\operatorname{gr}}
\newcommand{\coh}{\operatorname{coh}}
\newcommand{\End}{\operatorname{End}}
\newcommand{\Hom}{\operatorname{Hom}}
\newcommand{\uHom}{\underline{\operatorname{Hom}}}
\newcommand{\uHomY}{\uHom_{\OY}}
\newcommand{\uHomX}{\uHom_{\OX}}
\newcommand{\Ext}{\operatorname{Ext}}
\newcommand{\bExt}{\operatorname{\bf{Ext}}}
\newcommand{\Tor}{\operatorname{Tor}}

\newcommand{\inv}{^{-1}}
\newcommand{\airtilde}{\widetilde{\hspace{.5em}}}
\newcommand{\airhat}{\widehat{\hspace{.5em}}}
\newcommand{\nt}{^{\circ}}
\newcommand{\del}{\partial}

\newcommand{\supp}{\operatorname{supp}}
\newcommand{\GK}{\operatorname{GK-dim}}
\newcommand{\hd}{\operatorname{hd}}
\newcommand{\id}{\operatorname{id}}
\newcommand{\res}{\operatorname{res}}
\newcommand{\lrar}{\leadsto}
\newcommand{\im}{\operatorname{Im}}
\newcommand{\HH}{\operatorname{H}}
\newcommand{\TF}{\operatorname{TF}}
\newcommand{\Bun}{\operatorname{Bun}}
\newcommand{\Hilb}{\operatorname{Hilb}}
\newcommand{\Fact}{\operatorname{Fact}}
\newcommand{\F}{\mathcal{F}}
\newcommand{\nthord}{^{(n)}}
\newcommand{\Aut}{\underline{\operatorname{Aut}}}
\newcommand{\Gr}{\operatorname{\bf Gr}}
\newcommand{\Fr}{\operatorname{Fr}}
\newcommand{\GL}{\operatorname{GL}}
\newcommand{\gl}{\mathfrak{gl}}
\newcommand{\SL}{\operatorname{SL}}
\newcommand{\ff}{\footnote}
\newcommand{\ot}{\otimes}
\def\Ext{\operatorname {Ext}}
\def\Hom{\operatorname {Hom}}
\def\Ind{\operatorname {Ind}}
\def\bbZ{{\mathbb Z}}

\newcommand{\nc}{\newcommand}
\newcommand{\on}{\operatorname}
\nc{\cont}{\on{cont}}
\nc{\rmod}{\on{mod}}
\nc{\Mtil}{\widetilde{M}}
\nc{\wb}{\overline}
\nc{\wt}{\widetilde}
\nc{\wh}{\widehat}
\nc{\sm}{\setminus}
\nc{\mc}{\mathcal}
\nc{\mbb}{\mathbb}
\nc{\Mbar}{\wb{M}}
\nc{\Nbar}{\wb{N}}
\nc{\Mhat}{\wh{M}}
\nc{\pihat}{\wh{\pi}}
\nc{\JYX}{\cJ_{Y\leftarrow X}}
\nc{\phitil}{\wt{\phi}}
\nc{\Qbar}{\wb{Q}}
\nc{\DYX}{\D_{Y\leftarrow X}}
\nc{\DXY}{\D_{X\to Y}}
\nc{\dR}{\stackrel{\bbL}{\underset{\D_X}{\ot}}}
\nc{\Winfi}{\cW_{1+\infty}}
\nc{\K}{{\mc K}}
\nc{\unit}{{\bf \on{unit}}}
\nc{\boxt}{\boxtimes}
\nc{\xarr}{\stackrel{\rightarrow}{x}}
\nc{\Cnatbar}{\overline{C}^{\natural}}
\nc{\oJac}{\overline{\on{Jac}}}
\nc{\gm}{{\mathbf G}_m}
\nc{\Loc}{\on{Loc}}
\nc{\Bm}{\operatorname{Bimod}}

\nc{\Gm}{{\mathbb G}_m}
\nc{\Gabar}{\wb{\G}_a}
\nc{\Gmbar}{\wb{\G}_m}
\nc{\PD}{{\mathbb P}_{\D}}
\nc{\Pbul}{P_{\bullet}}
\nc{\PDl}{{\mathbb P}_{\D(\lambda)}}
\nc{\PLoc}{\mathsf{MLoc}}
\nc{\Tors}{\on{Tors}}
\nc{\PS}{{\mathsf{PS}}}
\nc{\PB}{{\mathsf{MB}}}
\nc{\Pb}{{\underline{\operatorname{MBun}}}}
\nc{\Ht}{\mathsf{H}}
\nc{\bbH}{\mathbb H}
\nc{\gen}{^\circ}
\nc{\Jac}{\operatorname{Jac}}
\nc{\sP}{\mathsf{P}}
\nc{\otc}{^{\otimes c}}
\nc{\Det}{\mathsf{det}}
\nc{\PL}{\on{ML}}

\title[Mirabolic Langlands Duality and Quantum CM]{Mirabolic Langlands Duality and the Quantum Calogero-Moser System}
\author{Thomas Nevins}
\address{Department of Mathematics\\University of Illinois at Urbana-Champaign\\Urbana, IL 61801 USA}
\email{nevins@illinois.edu}

\maketitle

\begin{abstract}
We give a generic spectral decomposition of the derived category of twisted $\D$-modules on a moduli stack
of mirabolic vector bundles on a curve $X$ in characteristic $p$: that is, we construct an equivalence with the
derived category of quasi-coherent sheaves on a moduli stack of mirabolic local systems on $X$.  This  equivalence may
be understood as a tamely ramified form of the geometric Langlands equivalence.  When $X$ has genus $1$,
 this equivalence generically solves (in the sense of noncommutative geometry) the quantum Calogero-Moser system.
\end{abstract}

\section{Introduction}
The geometric Langlands program aims at harmonic analysis of derived categories.  We initiate the application of these methods to
the case of mirabolic vector bundles on a curve $X$: that is, we construct an equivalence between the derived category of
twisted $\D$-modules on the moduli stack of vector bundles with mirabolic structure and the
derived category of quasi-coherent sheaves on a moduli stack of mirabolic local systems on $X$.
When $X$ has genus one, this equivalence gives a generic spectral decomposition of the quantum Calogero-Moser system, which is also closely
related to the category of representations of the spherical Cherednik algebra of $X$.

\subsection{Mirabolic Langlands Duality}
    As explained by Beilinson-Drinfeld \cite{BD Hitchin} (see also \cite{Frenkel}), the geometric Langlands program aims
     to develop harmonic analysis for certain derived categories.  
The main result of this paper is a geometric-Langlands--type equivalence in the tamely ramified {\em mirabolic} case,
extending the results (and methods) of Bezrukavnikov-Braverman \cite{BB} in the unramified case.
 
Let $k$ be an algebraically closed field of characteristic $p$, and suppose $p>n$.  Fix a smooth,
projective curve $X$ of genus $g\geq 1$ and a point $b\in X$.
Let $\lambda \in k\smallsetminus {\mathbb F}_p$.  Let $\PB= \PB_n(X)$ denote the moduli stack of rank $n$ vector bundles $\cE$
on $X$ equipped with a reduction of structure group at $b$ to the subgroup of $GL_n$, i.e. matrices fixing a line;
in other words, $\cE$ is equipped with a choice of a line in the fiber $\cE_b$.  We will also use the closely related stack $\Pb_n(X)$ of
bundles equipped with a reduction at $b$ to the {\em mirabolic subgroup} of matrices fixing a nonzero vector; we will refer to both as moduli stacks of {\em mirabolic vector bundles} on $X$.
  The moduli stack $\PB$ has a preferred ``determinant'' line bundle $\Det$ (which is not quite the usual determinant-of-cohomology line bundle on the moduli stack of bundles): see Section \ref{parabolic bundles} for a definition and properties.  We write
$\D_\PB(\lambda)$ for the sheaf of PD (for ``puissances divis\'ees,'' see \cite{BO})
 differential operators twisted by the $\lambda$th power of $\Det$.  Finally,
let $D(\D_\PB(\lambda))$ denote the quasicoherent derived category of left $\D_\PB(\lambda)$-modules on $\PB_n(X)$.  We will define
(see Section \ref{main equiv section}) a certain localization of this derived category,
denoted by $D(\D_\PB(\lambda))^\circ$,
 which (as we will explain below) corresponds to restricting to an open set of a quantum space.

Similarly, to each choice of $\overline{\ell}\in {\mathbb Z}/p{\mathbb Z}$, we associate
a moduli stack $\PLoc_n^\lambda(X, \overline{\ell})$  parametrizing rank $n$ mirabolic vector
bundles on $X$ with $\lambda$-twisted meromorphic connection compatible with the mirabolic structure
 (see Section \ref{twisted local systems} for definitions).\footnote{The parameter $\overline{\ell}$ simply imposes
 a restriction on which components of the full stack we allow, i.e. which degrees of vector bundles underlying local systems
are allowed.}
 We will define an open subset
$\PLoc_n^\lambda(X, \overline{\ell})^\circ$ of {\em generic local systems}, which corresponds to the localization of the $\D$-module category mentioned above.  We then have:

\begin{thm}\label{thm 1}
For each $\overline{\ell}$, there is a Fourier-Mukai--type equivalence of derived categories:
\bd
D\big(\D_\PB(\lambda)\big)^\circ \xrightarrow{\Phi^{\cP^\vee}} D_{\on{qcoh}}\big(\PLoc_n^\lambda(X, \overline{\ell})^\circ\big).
\ed
\end{thm}

There are natural functors on the two categories appearing in Theorem \ref{thm 1}: the Hecke functors
\bd
{\mathbf H}_r: D\big(\D_{\PB}(\lambda)\big)^\circ \rightarrow D\big(\D_{\PB\times X}(\lambda)\big)^\circ, \hspace{2em}
1\leq r\leq n,
\ed
 and the functors
\bd
{\mathbf T}_r: D_{\on{qcoh}}\big(\PLoc_n^\lambda(X, \overline{\ell})^\circ\big) \rightarrow
D_{\on{qcoh}}\big(\theo_{\PLoc_n^\lambda(X, \overline{\ell})^\circ}\boxtimes \D_X(\lambda)\big), \hspace{2em} 1\leq r\leq n
\ed
of tensoring with wedge powers of the universal twisted mirabolic local system ${\mathcal L}_{\on{univ}}$; see Section \ref{final Hecke section} for precise
definitions.  The equivalence $\Phi^{\cP^\vee}$ can be extended to an equivalence
\bd
D\big(\D_{\PB\times X}(\lambda)\big)^\circ
\longrightarrow
D_{\on{qcoh}}\big(\theo_{\PLoc_n^\lambda(X, \overline{\ell})^\circ}\boxtimes \D_X(\lambda)\big)
\ed
that is constant in the $X$ direction; we will also denote it by $\Phi^{\cP^\vee}$.
 We then have the ``Hecke eigenvalue property:"
\begin{thm}\label{thm 2}
For $1\leq r\leq n$,
$\Phi^{\cP^\vee}\circ {\mathbf H}_r \simeq {\mathbf T}_r \circ \Phi^{\cP^\vee}$.  In particular, for each choice of $x\in X$
($x\neq b$) and $1\leq r \leq n$,
the action of the $r$th Hecke operator ${\mathbf H}_r(x)$ at $x$ is identified by $\Phi$ with the action of tensoring with the vector space
$\wedge^r({\mathcal L}_{\on{univ}})_x$.
\end{thm}
In other words, the equivalence of Theorem \ref{thm 1} transforms the action of the Hecke functors into multiplication operators, i.e.
it ``diagonalizes'' the Hecke functors on Fourier transforms of skyscraper sheaves.

Passing from bundles to mirabolic bundles corresponds to allowing simple poles in Higgs fields or flat connections, i.e. tamely
ramified representations of fundamental groups or Galois groups.  Consequently, Theorems \ref{thm 1} and \ref{thm 2} 
 should be seen as a (twisted, and somewhat localized) form of tamely ramified geometric Langlands correspondence.

\subsection{Mirabolic Bundles, Quantum CM System, and Cherednik Algebras}
When $X$ is an elliptic curve, Theorems \ref{thm 1} and \ref{thm 2} generically solve the $n$-particle quantum elliptic Calogero-Moser system.  

Noncommutative algebraic geometry studies categories.  To an ``ordinary'' algebraic variety $V$, noncommutative geometry
associates its derived category of quasicoherent sheaves, $D_{\on{qcoh}}(V)$.\footnote{The derived category is generally not enough
to reconstruct the variety without a choice of $t$-structure or monoidal structure, but it provides a flexible and interesting ``shadow''
of the variety.}
There are other natural geometric 
sources of stable $\infty$-categories and we define a noncommutative space just to be a stable $\infty$-category.
When $X$ is an elliptic curve, 
the twisted cotangent bundle $T_\PB^*(\lambda)$ of $\PB_n(X)$ is the phase space of the classical $n$-particle elliptic CM system on $X$:
more precisely, here we are taking $\PB = \PB_n(X,0)^{ss}$, the semistable part of the degree $0$ component of the moduli stack of
rank $n$ mirabolic bundles on $X$.  This is the usual ``spectral'' description of the classical CM system (see \cite{TV, TV2, Nekrasov, Nekrasov survey, Do}, or a survey with references in \cite{flows}).
Moreover, just as the sheaf of algebras $\D_\PB(\lambda)$ is the canonical quantization of
the sheaf of functions on $T_\PB^*(\lambda)$,
  the category $D\big(\D_\PB(\lambda)\big)$ is the quantization of the category $D_{\on{qcoh}}\big(T_\PB^*(\lambda)\big)$: that is, 
  $D\big(\D_\PB(\lambda)\big)$ is
  the quantum CM phase space. 

From this point of view, Theorem \ref{thm 1} should be understood as describing and generically solving the quantum CM system.
Solving
a quantum system amounts to writing down
a basis of eigenstates for the quantum Hamiltonian.  To do this, one usually introduces a large collection of commuting operators that also commute with the Hamiltonian, and then one tries to solve the more highly structured problem of finding simultaneous
eigenstates of this collection of operators.  Theorems \ref{thm 1} and \ref{thm 2} realize this process for the localized category $D\big(\D_\PB(\lambda)\big)^\circ$. That is, Theorem \ref{thm 1}  gives a spectral
decomposition of $D\big(\D_\PB(\lambda)\big)^\circ$, the space of quantum states, via a Fourier transform, thereby diagonalizing the family of Hecke operators which are the analogs of the commuting family of quantum Hamiltonians.  This decomposition
allows us to write arbitrary states as direct integrals (Fourier-Mukai transforms) of basic states, which correspond to the skyscraper sheaves on $\PLoc$.   
Furthermore, the description in Theorem \ref{thm 2} of the Fourier transform of
Hecke functors ${\mathbf H}_r$  shows that this direct integral decomposition also simultaneously diagonalizes the entire ring of Hecke operators.  
The Hecke operators are quantum analogs of flows along the fibers of the Hitchin fibration---in other words, of Poisson brackets with Hitchin Hamiltonians. 

As we explain in Section \ref{cherednik algebras}, the category $D\big(\D_\PB(\lambda)\big)$ is also closely connected to
Cherednik algebras when $X$ has genus $1$ (see also \cite{Et2} for further explanation and \cite{FG2, FGT} for related recent work).  Indeed, suppose $X$ is a Weierstrass cubic (smooth or singular) and   restrict attention to
the semistable locus
$\PB_n(X,0)^{ss}$ of the degree $0$ component of $\PB$.\footnote{We also impose the additional technical condition, when $X$ is singular, that the pullback to the normalization be trivial:
see Definition \ref{cherednik ss}.}  The moduli {\em space} of degree $0$ semistable bundles is $(X^{sm})^{(n)}$, the 
 $n$th symmetric product (where $X^{sm}$ means the smooth locus if $X$ is singular).  Then a suitable interpretation of \cite{GG, FG} shows that the direct image of $\D_\PB(\lambda)$ to the moduli space $(X^{sm})^{(n)}$ is the spherical Cherednik algebra associated to $X$ and the symmetric group $S_n$ 
 (and the parameter value $\lambda$).  Moreover, one expects (and knows in the case that $X$ is a genus one curve with a cusp by
 \cite{KR}) that the category of
 representations of the spherical Cherednik algebra is exactly the microlocal analog of $D\big(\D_\PB(\lambda)\big)$ (see also the closely
 related 
 \cite{GS1, GS2}).  
 
This point of view on Cherednik algebras leads to a rather simple picture of the part of their representation
theory that should be captured by Theorem \ref{thm 1} when $X$ is an elliptic curve (we also expect analogs for singular curves $X$, cf. Section \ref{future directions}).  By way of motivation, let us consider the {\em rational} case, when $X$ is a genus one curve with a cusp and
$X^{sm} = {\mathbb A}^1$: the Cherednik algebra of $X^{sm}$ is then the rational Cherednik algebra.  Let ${\mathfrak h}^{reg}$ denote
$({\mathbb A}^1)^n \smallsetminus \Delta$, the $n$-fold product of ${\mathbb A}^1$ minus the big diagonal (i.e. the locus where at least
two points coincide).  
Recall that the {\em KZ functor} takes a module $M$ for the rational
 Cherednik algebra associated to $S_n$
 and localizes it to a $\D$-module on  ${\mathfrak h}^{reg}/S_n$.
If one starts with a module in category $\theo$ of the Cherednik algebra, this $\D$-module is actually a local system on
${\mathfrak h}^{\on{reg}}/S_n$, hence gives a representation of the braid group $\pi_1({\mathfrak h}^{\on{reg}}/S_n)$ which actually
factors through the finite Hecke algebra.  This representation of the Hecke algebra is the output $\operatorname{KZ}(M)$ of the
KZ functor applied to $M$ \cite{GGOR}.

We now observe that the localization in the definition of the KZ functor may be interpreted as restriction to an open set of
the moduli space of degree $0$ semistable bundles: more precisely, ${\mathfrak h}^{\on{reg}}/S_n$ is the moduli space of ``regular semistable'' bundles of degree $0$
on the projective curve $X$, and there is a gerbe over it---the classifying stack of the bundle of regular centralizers---that is actually
an open set of the moduli {\em stack} of semistable degree $0$ bundles on $X$.\footnote{On the regular locus, the difference between
the stack and the space is exactly captured by the bundle of regular centralizers.}
  Roughly speaking, then,  the KZ functor may be
understood as ``localization to the regular semistable locus of $\PB$.''  A similar functor exists for (smooth) elliptic curves $X$.  

By contrast, Theorems \ref{thm 1} and \ref{thm 2} describe a category of modules localized ``in a transverse direction'' to the localization appearing
in the KZ
functor.  That is, the phase space $T^*_\PB(\lambda)$ of the classical CM system admits two Lagrangian fibrations
$T^*_\PB(\lambda)\rightarrow \PB$ and $T^*_\PB(\lambda)\rightarrow \Ht$ (where $\Ht$ denotes the base of the Hitchin system)
which, roughly speaking, are transverse.  The KZ functor takes a module and restricts to an open set in the base $\PB$, whereas our
results apply over an open set in $\Ht$; in this sense, the two descriptions in the quantum case (which in characteristic $p$ is rather close
to the classical case, see below) are transverse to each other.\footnote{In the rational
case, i.e. when $X$ is a genus $1$ curve with a cusp, there is actually a duality interchanging the two directions: the quantum version of the duality of classical integrable systems from \cite{FGNR}.}

\subsection{Methods}\label{methods}
As we already remarked, the main theorem requires a field $k$ of characteristic $p>0$.  As the reader may have guessed,
this is because we use the same powerful methods as have been already exploited, to fantastic effect, in \cite{BMR, BMR2, BB, OV}
among others.  Namely, the crucial point is the so-called {\em Azumaya property} of PD (or ``crystalline'') differential operators
in characteristic $p$: that is, rings of differential operators have large centers and indeed are Azumaya algebras over their centers
in good (smooth) cases \cite{BMR}.  Usual Morita theory tells us that the module theory of Azumaya algebras is essentially a gerbe-twisted
form of the module theory of their centers.  Consequently, the derived category of $\D_Z$-modules in characteristic $p$ is just a twisted form
of the quasicoherent derived category of (the Frobenius twist of) $T^*_Z$.  The twisting itself is a subtle and deep structure,
but Morita theory at least tells us that we should hope to understand it geometrically.  In our setting of twisted differential operators,
the center of $\D(\lambda)$ is $\theo_{T^*_{Z^{(1)}}(\lambda^p-\lambda)}$, so we will describe the derived category of a gerbe
over $T^*_{Z^{(1)}}(\lambda^p-\lambda)$.  

A more precise description of our methods looks as follows.   The category of modules over an Azumaya algebra is equivalently encoded in the category
of quasicoherent sheaves on a gerbe, the gerbe of splittings of the Azumaya algebra.  So, for our purposes, in light of the Azumaya
property of (twisted) differential operators, we want to identify the Fourier-Mukai dual of a gerbe $\cG_\lambda$ of splittings.
 On the other hand, a remarkable fact, exploited in \cite{DP1, DP2, BB} is that gerbes over
abelian fibrations that arise ``in nature'' are often identified, under fiberwise Fourier-Mukai duality, with torsors over the abelian
fibration (this is essentially a simple consequence of the fact that FM transforms identify translation by a group element with
tensoring by a line bundle).

This Fourier-Mukai duality for gerbes was given a new
 interpretation in Arinkin's appendix \cite{Ar} to \cite{DP1}: it is an instance of Cartier
duality for group stacks.  In order to exploit this interpretation, we need a group structure on the gerbe $\cG_\lambda$.  Indeed,
one already {\em expects} from the Langlands philosophy that there should be a convolution-type monoidal structure on the
``automorphic side,'' i.e. on the $\D$-module side of our equivalence: this should be identified under the equivalence of
Theorem \ref{thm 1} with the tensor product on quasicoherent sheaves.  So we need to construct that monoidal structure, realized
via a group structure on $\cG_\lambda$, and then apply Cartier duality.

In the unramified case,
the existence of such a structure follows from a beautiful insight from \cite{BB}.  Once adapted to our setting, this insight is that
 the category of $\D(\lambda)$-modules
(i.e. twisted $\D$-modules) on $\PB$ can be identified as the subcategory of the category of $\D(\lambda)$-modules on the (much larger) space $T^*_\PB(\lambda)$ that are supported along a particular subvariety, the image of the canonical section.  On the other hand, the space $T^*_\PB(\lambda)$ has the structure of a Hitchin system---it is
a group stack---and so, at least over an open set, there is a convolution-type monoidal structure for $\D(\lambda)$-modules on
$T^*_\PB(\lambda)$.  Proving that this gives a monoidal structure on $\D_\PB(\lambda)$-modules then amounts to showing that this
subcategory is preserved by convolution, which follows from a character property (Proposition \ref{equality of thetas}) for a certain (twisted)
$1$-form on $\PB$.  This is explained and proven in Section \ref{hecke section}.  

Given this group stack structure on $\cG_\lambda$, then, we may apply Cartier duality as in \cite{Ar, BB}.   We use an explicit
construction to identify the Cartier dual of $\cG_\lambda$ with $\PLoc_n^\lambda(X,\overline{\ell})$.  This construction, although
fairly direct, is slightly unnatural from the spectral point of view on $\PLoc$, since it relies on an extension property
(Proposition \ref{extension lemma}) to control behavior near the pole of a connection; this is surely an artifact of our approach that
we expect to improve in the future.

At this point, we should explain how our methods differ from those of
 \cite{BB}.  There are at least two principal differences.  First, we work systematically with twisted $\D$-modules here, and the twistings we use are
{\em not} the same as the ones by $K^{1/2}$ that appear in the unramified geometric Langlands program; indeed, our twistings
require the presence of the mirabolic structure for their existence (in this regard, see
\cite[Sections~9.4-9.5]{Frenkel2} for an overview of tamely ramified geometric Langlands).
 As a result, we have a new parameter $\lambda$ appearing that does
not appear in the unramified setting of \cite{BB}.
This parameter $\lambda$ exactly matches the parameter of Cherednik algebras.  Furthermore, the use of the twisting actually helps us considerably: the conditition that the twisting parameter $\lambda$ lies in $k\smallsetminus {\mathbb F}_p$ is necessary for good behavior of the underlying spectral geometry.  Our condition should be compared with \cite{BFG}, where the twisting lies exactly
in ${\mathbb F}_p$, which leads to localization on the Hilbert scheme (i.e. an open subset of the ordinary cotangent bundle, cf. \cite{perverse}) rather than an honestly twisted cotangent bundle.

The second main difference is that we deal with mirabolic bundles here.  This is not merely a formal difference from \cite{BB}: indeed, one important consequence, already alluded to, is that dealing with mirabolic connections involves passing from Azumaya algebras to their more degenerate cousins, namely orders.
 We have tried to explain to the
reader just how this framework naturally arises in Section \ref{TDOs}.
 In the
present paper, the precise geometry of the orders that arise does not play a very large role (exactly because of
our genericity conditions) but it should be expected to play an important role in future work in this direction.  Moreover, we expect that some of the same phenomena that arise here should have (microlocal) counterparts for tamely ramified
geometric Langlands in characteristic zero as well.

There is also some extra
work involved in extending standard facts from bundles to mirabolic bundles: as just one example, a description of the behavior
of the Hitchin section (Section \ref{unit section section})
 doesn't seem to exist in the literature
in a form we can directly use here.   We have also included explanations of a few crucial ideas 
from \cite{BB} that we use, 
in the hope that the reader will be able to read the present paper without having already fully digested \cite{BB}.

Here is a more detailed summary of the contents of the sections.  In Section \ref{TDOs}, we summarize basic properties of twisted 
differential operators in characteristic $p$.  Section \ref{TDOs} begins with a summary of twisted differential operators and quantum
Hamiltonian reduction, followed by a review of the relationship between Azumaya algebras and their gerbes of splittings; we also
review the notion of tensor structure on an Azumaya algebra.  In Sections \ref{azumaya property subsection} and \ref{stacks subsection} we review the Azumaya property of PD differential operators in finite characteristic and its extension to stacks.
In Section \ref{canonical sections} we explain the canonical section of a twisted cotangent bundle over a twisted cotangent bundle
and explain how differential operators, seen as Azumaya algebras on the two, are related.  In Section \ref{compactified TCBs and orders}, we explain how differential operators extend as an order (an algebra that is generically Azumaya) over a compactification of
the twisted cotangent bundle.  In Section \ref{parabolic bundles}, we begin by introducing the notion of mirabolic bundle and explain
the determinant-twisted cotangent bundle of the moduli stack $\PB$ of mirabolic bundles.  Section \ref{cherednik algebras} explains the 
relationship between twisted differential operators on $\PB$ and Cherednik algebras.  We also describe a Hitchin-type 
map $T^*_\PB(\lambda) \rightarrow \Ht$ and prove the existence of smooth spectral curves for these Hitchin systems.  
Section \ref{twisted local systems} introduces the ``spectral'' side of the mirabolic Langlands duality, the moduli stack
$\PLoc_n^\lambda(X, \overline{\ell})$ of twisted
mirabolic
local systems.  In Section \ref{torsor section}, we prove that the moduli stack of generic twisted local systems forms a torsor over the 
generic locus of the  Hitchin
system  $\PB\rightarrow \Ht$ corresponding to smooth spectral curves. 

 Section \ref{hecke section} introduces the Hecke operators
and uses their action to prove a character property (Proposition \ref{equality of thetas}) for the canonical section of the twisted 
cotangent bundle of $T^*_\PB(\lambda)$.  Over the generic locus, $T^*_\PB(\lambda)$ is a group over $\Ht$, from which we 
get a convolution product on twisted $\D$-modules over $T^*_\PB(\lambda)$; the character property allows us to transport that
convolution to twisted $\D$-modules on $\PB$, providing (Theorem \ref{multiplicative structure}) a commutative group stack 
structure on the gerbe of splittings of 
$\D_{\PB}(\lambda)$.  Theorem \ref{thm 1} is proven in Section \ref{FM duality} using these ingredients.   More precisely,
the derived categories
$D\big(\D_\PB(\lambda)\big)^\circ$ and $D_{\on{qcoh}}\big(\PLoc_n^\lambda(X, \overline{\ell})^\circ\big)$ are both described in
terms of commutative group stacks containing the generic locus of the Hitchin system 
$T^*_\PB(\lambda)\rightarrow\Ht$.  We obtain the equivalence of Theorem \ref{thm 1} by showing that these commutative group
stacks are Cartier dual; this is Theorem \ref{main equiv thm}.  Section \ref{FM duality} concludes with the proof of Theorem 
\ref{thm 2}.

\subsection{Future Directions}\label{future directions}
As we have mentioned, this paper is a first step in the application of ideas and methods of the geometric Langlands program to
``quantum CM geometry.''  It leaves open a number of problems and questions to which we hope to return.  Most immediately,
one would like to extend the equivalence of Theorem \ref{thm 1} from the ``generic locus'' to, for example, the entire locus of
reduced and irreducible Hitchin spectral curves.  This seems to be a difficult problem in complete generality, but seems far more
tractable in genus one, and we hope to address this in future work.

In a related direction, the story told here can be generalized to {\em singular} base curves $X$ of genus one: this requires the
philosophy of log geometry and will be the
subject of a future paper.  Our emphasis on genus one here is no accident: as we have already indicated, in genus one
the spaces studied here lie at the heart of the representation theory of (rational, trigonometric, and elliptic) Cherednik algebras,
and we hope to convert the use of categorical harmonic analysis initiated here into information about representations.  All of this
forms a joint program, in progress, with D. Ben-Zvi.

We expect also that some of the methods used here can
be extended to work in
characteristic zero (and hence will give results related to Cherednik algebras in characteristic zero).  Second, it appears that the story told here has a $q$-analog related to
representations of DAHAs and related algebras.  Both of these are the subject of work in progress.

\subsection{Acknowledgments}
The author is deeply indebted to D. Ben-Zvi for generously sharing his ideas.  This paper springs from our previous joint work and is closely related to a larger, ongoing
joint project to understand applications of Langlands duality in genus one.

This research was partially supported by NSF grant DMS-0500221 and by a Beckman Fellowship from the Center for Advanced Study at 
the University of Illinois at Urbana-Champaign.

\subsection{Notation and Conventions}
Throughout, we let
$k$ be an algebraically closed field of finite characteristic $p$.  We let $X$ denote
a smooth projective curve over $k$ of genus $g(X)\geq 1$, and
$b\in X$ a fixed base point.  We will choose a twisting parameter
$\lambda \in k\smallsetminus {\mathbb F}_p$ as in Theorems \ref{thm 1} and \ref{thm 2}.  Given such a $\lambda$ we write
$c=\lambda^p -\lambda$ and $a=\lambda-\lambda^{1/p}$.  When we wish to work with a twisting parameter unrelated to this set-up, we 
sometimes use $\alpha\in k$ to denote such a parameter.
Finally, $n$ will denote a positive integer and $p>n$.

Throughout the paper, we will work without comment with the canonical enhancements of derived categories to stable $\infty$-categories (see, for example, Section 2 of \cite{BZFN})
and refer (abusively) to these as the ``derived categories.''  This requires no changes
in the proofs.

\section{Twisted Differential Operators and the Azumaya Property}\label{TDOs}
In this section, we review some basics of PD or crystalline twisted differential operators.  In particular, we explain the
Azumaya property \cite{BMR} and an extension of it (to an order on a compactified cotangent bundle).
We omit proofs that are either straightforward or follow closely analogous facts in \cite{BB}.
\subsection{Basics of TDOs}\label{basics of TDOs}
Let $Z$ be a smooth, separated algebraic space over $k$.  A {\em twisted differential operator} ring (or sheaf of twisted differential operators, or {\em TDO}) is a filtered
$\theo_Z$-algebra that, in a sense, is ``locally modelled on (the sheaf of PD differential operators) $\D_Z$.''  See \cite{BeiBer} for a precise definition and a general
discussion in characteristic zero; it is important to note that not all aspects of the discussion in \cite{BeiBer} apply equally well in
finite characteristic.  We will be interested here in only a special class of TDOs, namely the sheaves of differential operators
$\D(L^{\otimes \lambda})$ on
``fractional'' powers of a line bundle $L$: see Section \ref{tdos} below.  These really are sheaves of filtered rings, with $\theo_Z$ as
a noncentral subring, that are locally isomorphic to $\D_Z$.

\subsubsection{Hamiltonian Reduction}
Let $G$ be
an affine algebraic group and $P\xrightarrow{\pi} Z$ a principal $G$-bundle; we let $X\mapsto \wt{X}$ denote the infinitesimal
${\mathfrak g}$-action (here
${\mathfrak g} = \on{Lie}(G)$).  Let $\chi\in {\mathfrak g}^*$ be a $G$-invariant element.
We then get a Lie algebra homomorphism $I_\chi: {\mathfrak g}\rightarrow \D_P$ by $X\mapsto \wt{X} - \chi(X)$ (where $\chi(X)$ is interpreted
as a constant function, i.e. zeroth-order differential operator).  The {\em quantum Hamiltonian reduction}
(see \cite{BFG, GG}) of $\D_P$ at $\chi$ is
then defined to be
\bd
\D(\chi) = \D_Z(P,\chi) = \pi_*\left(\D_P/\D_P\cdot I_\chi({\mathfrak g})\right)^G.
\ed
  This is easily seen to be a TDO on $Z$.
  \subsubsection{TCBs}\label{tdos}
Now, replace the universal enveloping algebroid $\D_P$ of the tangent sheaf $T_P$ by the enveloping algebroid $\Sym(T_P)$ of the tangent sheaf thought of as an abelian Lie algebroid (with its symmetric $\theo_P$-bimodule
  structure).  Then the analogous procedure defines a commutative $\theo_Z$-algebra
  \bd
  \cA(\chi) = \pi_*\left(\Sym(T_P)/\Sym(T_P)\cdot I_\chi({\mathfrak g})\right)^G
  \ed
   whose associated affine $Z$-scheme
  $T^*_Z(P, \chi) : = \spec(\cA(\chi))$ is a {\em twisted cotangent bundle} of $Z$ in the sense of \cite{BeiBer}---and, indeed, it is the
  twisted cotangent bundle corresponding to the TDO $\D(\chi)$ under the procedure described in \cite{BeiBer}.
  It is easy to check that this construction exactly corresponds to
  ordinary Hamiltonian reduction of $T^*_P$ at the character $\chi$.  More precisely, if $\mu: T^*_P\rightarrow {\mathfrak g}^*$ is a
  moment map for the $G$-action, then $T^*_Z(P,\chi) = \mu\inv(\chi)/G$.

\subsubsection{$\Gm$ Case}
Consider the special case of the above construction when $G=\Gm$, and so $P$ is the $\Gm$-bundle associated to a line bundle
$L$ on $Z$.  Let $\lambda\in {\mathfrak g}_m^* = k$.
There is a more direct way to define the Lie algebroid $\D^1(\lambda)$ corresponding, via \cite{BeiBer}, to the TDO
$\D(\lambda)$ (that is, whose universal enveloping algebroid is $\D(\lambda)$).  Namely, in the case $\lambda = 1$, the Lie algebroid
$\D^1(1)  = \D^1(L)$ just consists of first-order differential operators acting on sections of $L$.  More generally, it is easy to compute
from the Hamiltonian reduction procedure the following:
\begin{lemma}\label{concrete TDOs lemma}
Given a character $\lambda\in {\mathfrak g}_m^* = k$, we have
\bd
\D^1(\lambda) = \left(\theo\oplus \D^1(L)\right)/\theo\cdot (-\lambda, 1)
\ed
as Lie algebroids.
\end{lemma}
We will henceforth write $\D_Z(L^{\otimes\lambda})$ to denote the TDO obtained from a $\Gm$-bundle in this way and
$T^*_Z(L^{\otimes\lambda})$ for the corresponding twisted cotangent bundle; or, if $L$ is understood, we will write
$\D_Z(\lambda)$ and $T^*_Z(\lambda)$, respectively.

\subsubsection{Compatibility of Reductions}
We return to the case of a general affine algebraic group $G$.
 Let $P\rightarrow Z$ be a principal $G$-bundle and
 $\psi: G\rightarrow \Gm$ a group homomorphism, which we will assume for convenience to be surjective.  Then there is an
 associated $\Gm$-bundle
 \bd
 L^\times = \Gm\times_{\psi,G} P \cong K\backslash P
 \ed
  (where $K = \on{ker}(\psi)$) corresponding to a line bundle $L$.
 Taking the derivative $\delta =d\psi: {\mathfrak g}\rightarrow k$ and choosing $\chi = \lambda\cdot \delta$ for some $\lambda\in k$, we obtain an invariant element
 $\chi\in {\mathfrak g}^*$.  Then:
 \begin{lemma}\label{TDO compatibility}
 \mbox{}
 \begin{enumerate}
  \item
 The TDO $\D_Z(L^{\otimes\lambda})$ is isomorphic to the quantum Hamiltonian reduction of $\D_P$ at
 $\chi = \lambda\cdot\delta$.
 \item
 Similarly,
 the twisted cotangent bundle $T^*_Z(L^{\otimes \lambda})$ is isomorphic to the symplectic reduction, with respect to the
 canonical $G$-action, of $T^*P$  at $\chi$.
\end{enumerate}
 \end{lemma}

\subsubsection{TDOs on Curves}\label{concrete TDO}
Let us give an alternative concrete description of some TDOs in a special case.  Namely, let $Z=X$ denote a smooth curve and
$b\in X$ a fixed base point.  We will give a description of $\D^1(\theo(b)^{\otimes\lambda})$ in terms of a local coordinate (i.e.
uniformizing parameter) $z$ at $b$.
 A calculation shows that $\D^1(\theo(b))$ is generated as an $\theo_X$-submodule
 of $\D^1_{X\smallsetminus b}$
 by the local sections $z\frac{\partial}{\partial z}$, $1$, and $\frac{\partial}{\partial z} + z\inv$ in a neighborhood of $b$ (and, of course, agrees with $\D^1_{X\smallsetminus b}$ away from $b$).  In particular, $\D^1(\theo(b))$ is a submodule of $T_X\oplus \theo(b)$.
\begin{lemma}\label{id of TDO on X}
Let $A(\lambda)$ denote the $\theo_X$-submodule of $T_X\oplus \theo_X(b)$ generated by
$z\frac{\partial}{\partial z}$, $1$, and $\frac{\partial}{\partial z} + \lambda z^{-1}$ near $b$ (and agreeing with $T_X\oplus \theo_X(b)$ away from $b$).
Then the restriction of the natural bracket on
$T_X\oplus \theo_{X\smallsetminus b}$ makes $A(\lambda)$ a Lie algebroid such that
\bd
\D^1(\theo(b)^{\otimes\lambda})\cong A(\lambda).
\ed
\end{lemma}

\subsection{Azumaya Algebras, Gerbes, and Module Categories}\label{gerbes}
Recall that an {\em Azumaya algebra} on a scheme or stack
$Z$ is a finite flat sheaf of $\theo_Z$-algebras (with $\theo_Z$ central) that, after passing to a flat
cover of $Z$, becomes isomorphic to a sheaf of matrix algebras.

For the basics of Azumaya algebras  we refer to \cite{DP1, BB}.

\subsubsection{Gerbes}
For us, a {\em $\Gm$-gerbe} on $Z$ will mean a torsor $Y$ for the commutative group stack $B\Gm$; note that this is not the most
general possible meaning of the term (see the discussion in Section 2.3 of \cite{DP1}).   If $Y$ is a $\Gm$-gerbe over $Z$, then the quasicoherent derived category
$D^b_{\on{qcoh}}(Y)$ comes equipped with a natural direct sum decomposition,
\bd
D^b_{\on{qcoh}}(Y)  = \bigoplus_{n\in{\mathbb Z}} D^b_{\on{qcoh}}(Y)_n.
\ed
This decomposition is easy to see when $Y$ is the trivial gerbe $B\Gm\times Z$: it is exactly the ``weight decomposition'' induced by
the weight decomposition for quasicoherent sheaves on $B\Gm$ (that is, the weight decomposition for $\Gm$-equivariant sheaves on $Z$,
where $Z$ is equipped with the trivial $\Gm$-action).  In general, the decomposition on the derived category of $Y$ is the unique
one that is compatible with flat base change along $Z$ and that agrees with the weight decomposition on trivial gerbes.

If $A$ is an Azumaya algebra on $Z$, the category of splittings of $A$ forms a $\Gm$-gerbe on $Z$, which, following \cite{BB}, we
will denote by $\cG_A$.  We then have:
\begin{lemma}[see \cite{BB}, Lemma 2.4]\label{weight 1 lemma}
The category $D^b_{\on{qcoh}}(A)$ is canonically equivalent to $D^b_{\on{qcoh}}(\cG_A)_1$.
\end{lemma}

\subsubsection{Tensor Structures}
Let $G\rightarrow \Ht$ be a commutative group stack over a scheme $\Ht$, with product
$m: G\times_\Ht G\rightarrow G$, unit section $i: \Ht \rightarrow G$, and transpose $\sigma: G\times G \rightarrow G\times G$ (so $\sigma(g_1, g_2) = (g_2, g_1)$).  Let $A$ denote an Azumaya algebra on $G$.
As in \cite{BB} or, especially, Definition 5.23 of \cite{OV}, a {\em tensor structure} on $A$ with respect to the product $m$
is an equivalence of Azumaya algebras---that is, a bimodule $B$ providing a Morita equivalence (which we denote by $\delta$)
$m^*A \simeq A\boxtimes A$ on $G\times_\Ht G$---together with an associativity isomorphism
\bd
(m\times 1)^*B \otimes_{(m\times 1)^*A} (B\otimes A) \cong
(1\times m)^*B \otimes_{(1\times m)^*A} (A\otimes B)
\ed
 on $G\times_\Ht G\times_\Ht G$, satisfying the pentagon axiom of \cite[Definition~1.0.1]{DM}: see \cite{OV}, Section 5.5 for more details.  As in \cite[Lemma~5.24]{OV}, an Azumaya algebra equipped with a tensor structure over $G$ comes equipped with a canonical splitting $N_0$ of its restriction $i^*A$ to the unit section.  
We also have an isomorphism
$\tau: \sigma^*(A\boxtimes A)\cong A\boxtimes A$.
 
 A tensor structure defines a tensor product operation on the category of $A$-modules as follows: given $A$-modules $M_1$ and $M_2$, the bimodule $B$ converts $M_1\boxtimes M_2$ into an $m^*A$-module; then applying $m_*$ to the resulting $m^*A$-module $\delta(M_1\boxtimes M_2)$ yields an $A$-module 
 $M_1\circledast M_2 = m_*(\delta(M_1\boxtimes M_2))$.  Then the direct image $i_*N_0$ of the splitting module for $i^*A$ comes equipped with an isomorphism $i_*N_0 \cong i_*N_0\circledast i_*N_0$ and $i_*N_0$ becomes a unit object for $\circledast$ in 
 $A\on{-mod}$ (Lemma 5.24 of \cite{OV}).

As above, the bimodule $B$ defines an equivalence between $m^*A$ and $A\boxtimes A$.  
We can also convert $\sigma^*B$ into a second such equivalence via the identifications
\bd
m^*A \cong \sigma^*m^* A \;\;\; \text{and}\;\;\; \sigma^*(A\boxtimes A) \xrightarrow{\tau} A\boxtimes A.
\ed  
A commutativity constraint for the tensor structure on $A$ is  an isomorphism $\gamma$ from the bimodule $B$ defining the equivalence $\delta$ to the bimodule $\sigma^*B$.  We require that $\gamma$ satisfy $\gamma \circ \gamma= \on{id}$ and also that $\delta$ and $\gamma$ are compatible: that is, that they
satisfy the hexagon axiom of \cite[Diagram~(1.0.2)]{DM}.  See Section 5.5 of \cite{OV} for more details.

A commutative tensor structure on $A$ determines a structure of commutative group stack on the
gerbe $\cG_A$ of splittings by the following construction.  Suppose $U\rightarrow \Ht$ is a scheme over $\Ht$, and
$U\xrightarrow{f_i} \cG_A$, $i=1,2$ are two morphisms over $\Ht$.  Let $\overline{f}_i: U\rightarrow G$ denote the projection to $G$.
We want to define a morphism $f_1\ast f_2: U\rightarrow \cG_A$ that covers the map
$\overline{f}_1\ast\overline{f}_2: U\rightarrow G$ that is defined as the composite:
\bd
\xymatrix{
 U\ar[r]^{\Delta} \ar@/^1.5pc/[rrr]^{\overline{f}_1\ast \overline{f}_2} &  U\times_\Ht U \ar[r]^{\overline{f}_1\times\overline{f}_2} &  G\times_\Ht G \ar[r]^{m} & G.
 }
\ed
By definition,  $f_1\ast f_2$ is determined by a splitting
of $(\overline{f}_1\ast\overline{f}_2)^*A$.  But we are given $f_1$ and $f_2$, i.e. choices of splitting modules
$\cE_i$ of $\overline{f}_i^*A$.  The module $\Delta^*(\cE_1\boxtimes\cE_2)$ determines a splitting of
\bd
\Delta^*(\overline{f}_1\times\overline{f}_2)^*A\boxtimes A \simeq \Delta^*(\overline{f}_1\times\overline{f}_2)^*m^*A
 = (\overline{f}_1\ast\overline{f}_2)^*A,
 \ed
 where the first equivalence is determined by the tensor structure of $A$.  Thus,
$\cE_1\ast\cE_2 = \Delta^*(\cE_1\boxtimes\cE_2)$ determines a morphism $f_1\ast f_2$ as desired. The associativity 2-morphism and compatibilities for
a group stack are similarly determined by the tensor structure of $A$.

\subsection{Azumaya Property of TDOs in Finite Characteristic}\label{azumaya property subsection}
Let $Z$ be a smooth, separated algebraic space over $k$ and $P$ a $\Gm$-bundle over $Z$ with associated line bundle $L$. Let
$Z\xrightarrow{F}Z^{(1)}$ denote the relative Frobenius morphism of $Z$, where $Z^{(1)}$ denotes the Frobenius twist of
$Z$ (see \cite{BMR}, Section 1.1.1, for a nice discussion).
\subsubsection{Center of $\D$}\label{azumaya property subsubsection}
 As in \cite{BMR}, the center of the algebra $\D_Z(\lambda):=\D_Z(L^{\otimes\lambda})$ is large.  Indeed, let $c=\lambda^p-\lambda$, and let
\bd
T^*_{Z^{(1)}}(c) := T^*_{Z^{(1)}}\left((L^{(1)})^{\otimes c}\right)\xrightarrow{p} Z^{(1)}
\ed
denote the $c$-twisted cotangent bundle of $Z^{(1)}$.
  Then there is an algebra homomorphism
 \bd
 F^{-1}\theo_{T^*_{Z^{(1)}}(c)}\rightarrow \D_Z(\lambda)
 \ed
  that maps isomorphically onto the center of
 $\D_Z(\lambda)$.  Moreover, this map is compatible with the filtrations on these two algebras if we put the generators of
$\theo_{T^*_{Z^{(1)}}(c)}$ in degree $p$ rather than $1$---see \cite{BMR} for more details.

 This inclusion makes $\D_Z(\lambda)$ a finite, flat $F^{-1}\theo_{T^*_{Z^{(1)}}(c)}$-algebra of rank $p^{2\on{dim}(Z)}$.  It follows that there
 is a sheaf of algebras, which we will also abusively write $\D_Z(\lambda)$,  on $T^*_{Z^{(1)}}(c)$, that is itself finite and flat of rank $p^{2\on{dim}(Z)}$ and whose direct image to
 $Z^{(1)}$ is isomorphic to the Frobenius direct image $F_*\D(\lambda)$.  Moreover, there is a natural equivalence of
 module categories for the two algebras.  It follows from \cite[Theorem~2.2.3]{BMR} that $\D_Z(\lambda)$ is an Azumaya algebra
 over $T^*_{Z^{(1)}}(c)$.

\subsubsection{Pullbacks}\label{bimodules}
Suppose that $Z$ and $W$ are smooth, separated algebraic spaces
 over $k$ and $f:Z\rightarrow W$ is an arbitrary $k$-morphism.  Let $L$ denote a line bundle on $W$ and also (abusively) its pullback to
 $Z$, and let $T_Z^*(\alpha)$, $T_W^*(\alpha)$ be the twisted cotangent bundles associated to these line bundles and a choice of weight $\alpha\in k$.    The following lemma simply
says that the usual pullback of differential forms can be twisted.
\begin{lemma}\label{pullbacks lemma}
For any
choice of $\alpha\in k$, we get a morphism of twisted cotangent bundles:
\bd
df: T_W^*(\alpha) \times_W Z \rightarrow T^*_Z(\alpha).
\ed
These morphisms are functorial in $f$ in the usual sense.
\end{lemma}

\subsubsection{Operations}
As in characteristic zero, there are two natural $\D$-bimodules that intervene in the study of direct and inverse images of (twisted)
$\D$-modules for a morphism $f:Z\rightarrow W$ of smooth algebraic spaces \cite[Section~IV.5]{Borel}.\footnote{The twisted version does not appear in {\em loc. cit.}, but it is a straightforward generalization in our setting
 and locally it works identically to the untwisted version.}
  The first is $\D_{Z\rightarrow W}(\lambda)$, which is the pullback
$f^*\D_W(\lambda)$ of $\D_W(\lambda)$ as a left $\theo$-module: in other words,
\bd
\D_{Z\rightarrow W}(\lambda) = \theo_Z \otimes_{f^{-1}\theo_W} f^{-1}\D_W(\lambda).
\ed
The second is $\D_{W\leftarrow Z}(\lambda)$, which is obtained by taking $\D_W(\lambda)$ as a {\em right} $\theo_W$-module
and applying $f^{!}$; alternatively, it is given by the following formula:
\begin{equation}\label{right as left}
\D_{W\leftarrow Z}(\lambda) = f^{-1}(\D_W(\lambda)\otimes_{\theo_W} \omega_W^{-1})\otimes_{f^{-1}\theo_W} \omega_Z.
\end{equation}
See \cite[Equation 5.0.2]{Laumon} and surrounding discussion in characteristic zero (most of which applies equally well in characteristic $p$).  Note that the above formula \eqref{right as left} for $\D_{W\leftarrow Z}(\lambda)$ makes use of the isomorphism
\bd
\D_Z(L^{\otimes\lambda})^{\on{op}} \cong \D_Z(L^{\otimes-\lambda}\otimes\omega_Z)
\ed
in defining the module structures (tensoring with both $\omega_W^{-1}$ and $\omega_Z$ has the effect of interchanging the left and
right module structures).
 By construction, $\D_{Z\rightarrow W}(\lambda)$ is a $\big(\D_Z(\lambda), \D_W(\lambda)\big)$-bimodule and
$\D_{W\leftarrow Z}(\lambda)$ is a $\big(\D_W(\lambda), \D_Z(\lambda)\big)$-bimodule.

 As in \cite[Section~3.6]{BB}, these bimodules
may actually be lifted to coherent sheaves on $T^*_{Z^{(1)}}(c) \times_{W^{(1)}} T^*_{W^{(1)}}(c)$ and
 $T^*_{W^{(1)}}(c) \times_{W^{(1)}} T^*_{Z^{(1)}}(c)$, respectively, where $c=\lambda^p-\lambda$ (\cite{BB} treats the first, and the second works similarly because of \eqref{right as left}).
 In fact, as in \cite[Lemma~3.8]{BB}, these sheaves are actually supported on the two graphs of $df^{(1)}$; to distinguish them, we write
 $\on{graph}\big(df^{(1)}\big) \subset  T^*_{Z^{(1)}}(c) \times_{W^{(1)}} T^*_{W^{(1)}}(c)$ and
  $\on{graph}\big(df^{(1)}\big)^{\dagger} \subset  T^*_{W^{(1)}}(c) \times_{W^{(1)}} T^*_{Z^{(1)}}(c)$, respectively, for these
  graphs.

\begin{prop}\label{pb equiv}
Let $f:Z\rightarrow W$ be a morphism of smooth, separated algebraic spaces.  Let $c= \lambda^p-\lambda$.
\begin{enumerate}
\item Let $\Gamma_f = \on{graph}\big(df^{(1)}\big)$.  Write
\bd
\xymatrix{
T^*_{Z^{(1)}}(c)   & \ar[l]_{\hspace{-3em}\pi_Z}   T^*_{Z^{(1)}}(c) \times_{W^{(1)}} T^*_{W^{(1)}}(c) \ar[r]^{\hspace{3em}\pi_W} &
T^*_{W^{(1)}}(c)
}
\ed
for the two projections.  Then the bimodule  $\D_{Z\rightarrow W}(\lambda)$ gives an equivalence between the Azumaya algebras
$\pi_Z^* \D_Z(\lambda)|_{\Gamma_f}$ and $\pi_W^* \D_W(\lambda)|_{\Gamma_f}$.

\item Similarly, let $\Gamma_f^\dagger = \on{graph}\big(df^{(1)}\big)^\dagger$.  Write
\bd
\xymatrix{
T^*_{W^{(1)}}(c)   & \ar[l]_{\hspace{-3em}\pi_W}   T^*_{W^{(1)}}(c) \times_{W^{(1)}} T^*_{Z^{(1)}}(c) \ar[r]^{\hspace{3em}\pi_Z} &
T^*_{Z^{(1)}}(c)
}
\ed
for the two projections.  Then the bimodule  $\D_{W\rightarrow Z}(\lambda)$ gives an equivalence between the Azumaya algebras
$\pi_W^* \D_W(\lambda)|_{\Gamma_f^\dagger}$ and $\pi_Z^* \D_Z(\lambda)|_{\Gamma_f^\dagger}$.
\end{enumerate}
\end{prop}
\begin{remark}
It is instructive to consider the case when the twisting is trivial and $f$ is a smooth morphism.  Then the differential
$df^{(1)}$ is a closed immersion consisting of pullbacks of cotangent vectors.  Then the proposition says, essentially, that the pullback of
a $\D$-module from downstairs is the same as a $\D$-module upstairs such that vertical tangent vectors act trivially (note that we
have only pulled back along the Frobenius twist of $f$, so any additional functions are also killed by vertical tangent vectors).
\end{remark}

Let us consider one additional special case of Proposition \ref{pb equiv} that plays a central role in this paper.  Suppose
\bd
W_1\xleftarrow{q_1} Z \xrightarrow{q_2} W_2
\ed
is a diagram in which the two morphisms are smooth.  Suppose we equip $W_1, W_2$ and $Z$ with compatible (in the obvious sense)
line bundles.   Then we get immersions of twisted cotangent bundles
\bd
T^*_{W_1^{(1)}}(c)\times_{W_1^{(1)}} Z^{(1)} \hookrightarrow T^*_{Z^{(1)}}(c) \hookleftarrow T^*_{W_2^{(1)}}(c)\times_{W_2^{(1)}} Z^{(1)}.
\ed
As a result, the graphs that appear in Proposition \ref{pb equiv} may be understood as immersed subschemes of $T^*_{Z^{(1)}}(c)$ and
the composite graph $\Gamma_{q_1}^\dagger \circ \Gamma_{q_2}$, the support of the functor
$(q_1)_*q_2^*: D(\D_{W_2}(\lambda))\rightarrow D(\D_{W_1}(\lambda))$, is a subscheme of $T^*_{Z^{(1)}}(c)$.
We observe:
\begin{lemma}\label{support of Hecke}
\bd
\Gamma_{q_1}^\dagger \circ \Gamma_{q_2} =
T^*_{W_1^{(1)}}(c)\times_{W_1^{(1)}} Z^{(1)}\cap
T^*_{W_2^{(1)}}(c)\times_{W_2^{(1)}} Z^{(1)}.
\ed
\end{lemma}

\subsection{Differential Operators on Stacks in Finite Characteristic}\label{stacks subsection}
So far, we have restricted our discussion to differential operators on smooth, separated algebraic spaces.  We next explain some features
of differential operators on stacks.
\subsubsection{Azumaya Property on Stacks}
 Our discussion of the Azumaya property on algebraic spaces generalizes to smooth algebraic stacks.  This is explained clearly in Section 3.13 of \cite{BB}.  The crucial point is:

\begin{lemma}[cf. \cite{BB}, Lemma 3.14]\label{D on stacks}
Let $Z$ be a smooth, irreducible algebraic stack.  Let $L$ be a line bundle on $Z$ and $\lambda\in k$.
 Suppose $\on{dim}(T^*Z) = 2 \on{dim} Z$ and that $T^*_{Z^{(1)}}(c)$ has an open substack
$T^*_{Z^{(1)}}(c)^0$ which is a smooth Deligne-Mumford stack.  Then there exists a natural coherent sheaf of algebras $\D_Z(\lambda)$ on
$T^*_{Z^{(1)}}(c)$ whose restriction to $T^*_{Z^{(1)}}(c)^0$ is an Azumaya algebra.  The algebra $\D_Z(\lambda)$
 satisfies the functorial properties of
Proposition \ref{pb equiv} above and Corollary \ref{D via pullback}, Lemma \ref{pullback and pullback} below and agrees with the
usual definition on algebraic spaces.
\end{lemma}
The construction of \cite{BB} uses Proposition \ref{pb equiv} as a definition: indeed, given a smooth atlas $U\rightarrow Z$, one can
use Proposition \ref{pb equiv} to define the pullback of $\D_Z(\lambda)$ to $T^*_{U^{(1)}}(c)$, and the functoriality properties
show that these definitions are compatible with smooth base change in such a way that the so-defined algebra actually descends to
$T^*_{Z^{(1)}}(c)$.
\begin{remark}
A stack satisfying the condition  $\on{dim}(T^*Z) = 2 \on{dim} Z$ is called ``good'' in \cite[Section~1.1.1]{BD Hitchin}.  The condition that $T^*_{Z^{(1)}}(c)$ has an open substack
$T^*_{Z^{(1)}}(c)^0$ which is a smooth Deligne-Mumford stack is implied by (and indeed, weaker than) the condition of being a ``very good'' stack \cite{BD Hitchin}.
\end{remark}

\subsubsection{Differential Operators on a Gerbe}\label{stack vs space}
There is one other feature of differential operators on stacks that will be important to us.  Namely, suppose ${\mathcal X}$ and $X$ are stacks and ${\mathcal X}\rightarrow X$ makes ${\mathcal X}$ a $\Gm$-gerbe over $X$.  Then there is
essentially no difference between differential operators on $X$ and ${\mathcal X}$: the pullback of differential operators on $X$ gives those on ${\mathcal X}$, and similarly for twisted differential operators $\D_{\mathcal X}(L^{\otimes \lambda})$ and 
$\D_X(L^{\otimes\lambda})$ when $L$ is a line bundle on $X$ (that is, the twist comes from $X$).  This will be
important to us later, when we consider two moduli objects for mirabolic bundles, $\PB_n(X)$ and $\Pb_n(X)$.  The former
is a $\Gm$-gerbe over the latter, and it is convenient to be able to study differential operators on both stacks in
a similar way.\footnote{Note, however, that ${\mathcal X}$ will not satisfy the hypotheses of Lemma \ref{D on stacks}.}

\subsection{Twisted Cotangent Bundles and Canonical Sections}\label{canonical sections}
Let $W$ be a smooth, irreducible algebraic space.
Let $L$ denote a line bundle on $W$; for any space $S\rightarrow W$, we will abusively denote by $L$ the pullback of $L$ to
$S$.  We also write $T_S^*(\alpha) := T_S^*(L^{\otimes \alpha})$.

\subsubsection{Canonical Sections}
Fix $a\in k$.  Let $p_W^{(1)}: {T_{W}^*(a)}^{(1)}\rightarrow W^{(1)}$ denote the Frobenius twist of the canonical projection.  By Lemma \ref{pullbacks lemma}, for any $c$ there is a canonical pullback
morphism
\begin{equation}\label{T^* pb}
dp_W^{(1)}: T_{W^{(1)}}^*(c)\times_{W^{(1)}} (T^*_W(a))^{(1)}\rightarrow T^*_{(T^*_W(a))^{(1)}}(c).
\end{equation}
Setting $c=a^p$, using the equality $T^*_W(a)^{(1)} = T^*_{W^{(1)}}(a^p)$
and composing \eqref{T^* pb} with the diagonal morphism
\begin{equation}\label{delta}
\delta^{(1)}: T^*_{W^{(1)}}(a^p)\rightarrow T^*_{W^{(1)}}(a^p)\times_{W^{(1)}} T^*_{W^{(1)}}(a^p),
\end{equation}
we get a section $\theta_W^{(1)} = dp_W^{(1)}\circ \delta^{(1)}$ of the $a^p$-twisted cotangent bundle
$T^*_{T^*_W(a)^{(1)}}(a^p)$
 of $T^*_W(a)^{(1)}$, which we refer to
as the {\em canonical section}.

We obtain the following from Proposition \ref{pb equiv}:
\begin{corollary}\label{D via pullback}
Fix $\lambda\in k$.  Let $a = \lambda- \lambda^{1/p}$ and $c=\lambda^p-\lambda$ (so that $a^p = c$).  
\mbox{}
\begin{enumerate}
\item Suppose that $Z\xrightarrow{f} W$ is a smooth morphism and that $\delta_1: T^*_W(a)\rightarrow Z$ is a morphism so that 
$f\circ\delta_1 = p_W$.  Then, lettting $\delta  = 1\times \delta_1: T^*_W(a)\rightarrow T^*_W(a)\times_W Z$ and 
$\theta^{(1)} = df^{(1)}\circ \delta^{(1)}$, we get
a canonical equivalence of Azumaya 
algebras on $T^*_{W^{(1)}}(c)$:
\bd
\D_W(\lambda) \simeq \big(\theta^{(1)}\big)^*\D_{Z}(\lambda).
\ed
\item
In particular, taking $Z = T_W^*(a)$, there is a canonical equivalence of Azumaya algebras on $T^*_{W^{(1)}}(c)$:
\bd
\D_W(\lambda) \simeq \big(\theta_W^{(1)}\big)^* \D_{T^*_{W}(a)}(\lambda).
\ed
\end{enumerate}
\end{corollary}
\subsubsection{Equivalence Under Pullback}
Let $Z\xrightarrow{f} W$ be any morphism of smooth, separated algebraic spaces, $L$ a line bundle on $W$; we observe the notational convention
about pullbacks of $L$ as above.
Let $\lambda\in k$ and $c=\lambda^p-\lambda$.  
Let $\sigma: W^{(1)}\rightarrow T^*_{W^{(1)}}(c)$ denote any section.
\begin{defn}\label{pullback def}
The {\em pullback} $(f^{(1)})^*\sigma : Z^{(1)}\rightarrow T^*_{Z^{(1)}}(c)$ is the composite
\bd
Z^{(1)}\xrightarrow{(\sigma\circ f^{(1)}, 1_{Z^{(1)}})} T_{W^{(1)}}^*(c)\times_{W^{(1)}} Z^{(1)} \xrightarrow{df^{(1)}} T^*_{Z^{(1)}}(c).
\ed
\end{defn}
\begin{lemma}\label{pullback and pullback}
There is a canonical equivalence of Azumaya algebras on $Z^{(1)}$:
\bd
\big((f^{(1)})^*\sigma\big)^*\D_Z(\lambda) \simeq \big(\sigma\circ f^{(1)}\big)^*\D_W(\lambda).
\ed
\end{lemma}
\begin{lemma}\label{checking via pullback}
Let $\alpha \in k$.
Suppose that $f:Z\rightarrow W$ is smooth and dominant.  Let
$\sigma_1,\sigma_2: W\rightarrow T^*_W(\alpha)$ be two sections.  Then $\sigma_1=\sigma_2$ if and only if
$f^*\sigma_1= f^*\sigma_2$.
\end{lemma}

\subsubsection{Compatibility of Canonical Sections}
Now, let $f:Z\rightarrow W$ denote a smooth, dominant morphism of smooth algebraic spaces and let $\alpha\in k$.  
We will let $L$ denote a
line bundle on $W$ and observe the notational convention about pullbacks of $L$ as above.  There is a weak form of compatibility
 for the canonical sections of the twisted cotangent bundles
over $T_W^*(\alpha)$ and $T_Z^*(\alpha)$.

More precisely, as above $f$ determines a natural immersion
\bd
df: T_W^*(\alpha)\times_W Z \hookrightarrow T_Z^*(\alpha).
\ed
Let $\pi_W: T_W^*(\alpha)\times_W Z\rightarrow T_W^*(\alpha)$ denote the projection.

\begin{prop}\label{two canonical sections agree}
The sections $\pi_W^*\theta_W$  and $df^*\theta_Z$ of the bundle
\bd
T^*_{T^*_W(\alpha)\times_W Z}(\alpha)\rightarrow
T^*_W(\alpha)\times_W Z
\ed
 agree.  Moreover, $df^*\theta_Z$  naturally lands in
$T^*_{T^*_W(\alpha)}(\alpha)\times_W Z$.
\end{prop}
The proposition follows from comparing the definitions; alternatively, it can be checked by a calculation in local coordinates.

\subsection{Compactification of TCBs and Orders}\label{compactified TCBs and orders}
Suppose $Z$ is a smooth, separated algebraic space, $L$ is a line bundle on $Z$, and $\lambda\in k$ is a weight.  We write
 $\D^\ell_Z(\lambda)$ for the $L^{\otimes\lambda}$-twisted PD differential operators on $Z$ of order less than or equal to $\ell$.
Let
\bd
\cR = \cR(\D(\lambda)) = \bigoplus_{\ell\geq 0} \D^{\ell}(\lambda)\cdot t^\ell \subset \D(\lambda)\otimes_k k[t],
\ed
the {\em Rees algebra} of $\D(\lambda)$.   When we wish to emphasize that this algebra lives on $Z$, we will write
$\cR_Z$.  This is a quasicoherent sheaf of graded algebras on $Z$, with an element $t$ such that
$\cR/t\cR \cong \on{gr}\D(\lambda) = \on{Sym}^\bullet(T_Z)$.

\subsubsection{$\on{Qgr}$ and Veronese Equivalence}
Let $\cR^{(p)}$ denote the $p$-Veronese subring of $\cR$: this means the subalgebra
$\displaystyle\cR^{(p)} = \bigoplus_{\ell\geq 0} \cR_{p\ell}$.

\begin{remark}
There is a potential cause for confusion because of the similarity
to the notation for Frobenius twist---however, in context this will always be clear.
\end{remark}
It is common in noncommutative geometry to think of $\cR$ as the sheaf of homogeneous coordinate rings of a projective bundle
over $Z$; in light of this, we are interested in the derived category $D^b(\on{Qgr}\,\cR)$.  Here $\on{Qgr}\,\cR$ is the quotient of the category of graded $\cR$-modules (that are quasicoherent as $\theo_Z$-modules)
by its Serre subcategory of locally bounded modules; we think of it as the category of quasicoherent sheaves on the noncommutative
projective bundle $\bproj\,\cR \rightarrow Z$---see \cite{perverse}
for more discussion and details.
The subcategory of $\on{Qgr}\,\cR$ consisting of finitely generated modules modulo bounded
modules is denoted $\qgr\,\cR$.

By \cite[Theorem~4.4]{Verevkin}, whenever $s>0$ is an integer, one has equivalences of categories
\begin{equation}\label{Verevkin equiv}
\xymatrix{
\on{Qgr}\,\cR \ar@<.3ex>[r]^{\hspace{-.3em}\on{Ver}} &  \on{Qgr}\,\cR^{(s)}\ar@<.3ex>[l]^{\on{\hspace{-.3em}Ind}},
}
\hspace{3em}
\xymatrix{
\on{qgr}\,\cR \ar@<.3ex>[r]^{\hspace{-.3em}\on{Ver}} &  \on{qgr}\,\cR^{(s)}\ar@<.3ex>[l]^{\on{\hspace{-.3em}Ind}},
}\end{equation}
where the functors are given by taking the $s$-Veronese submodule of a graded $\cR$-module:
\bd
M = \oplus_{\ell} M_\ell \mapsto \on{Ver}(M) = M^{(s)} = \oplus_{\ell} M_{s\ell},
\ed
 and inducing a graded
$\cR^{(s)}$-module to a graded $\cR$-module (i.e. $\on{Ind}(N) = \cR \otimes_{\cR^{(s)}} N$) respectively.  The operation of passing to the $s$-Veronese is exactly the noncommutative
analog of passing from a projectively embedded variety to the same variety projectively embedded by composing with the $s$-Veronese
map on projective space: in particular, this explains why the equivalence above is natural to expect.   We will  use these equivalences
in the case $s=p$.

\subsubsection{Compactification}
We emphasize that, for the discussion that follows, we require that $\on{char}(k) = p>0$.  

Recall that  $\cZ(\D(\lambda))$ is isomorphic to the direct image $(p_{Z^{(1)}})_*\theo_{T^*_{Z^{(1)}}(c)}$ along the projection
$p_{Z^{(1)}}: T^*_{Z^{(1)}}(c)\rightarrow Z^{(1)}$ where
$c=\lambda^p-\lambda$.  The sheaf $(p_{Z^{(1)}})_*\theo_{T^*_{Z^{(1)}}(c)}$  comes equipped 
with a filtration by ``order of pole at infinity in the fibers" of $T^*_{Z^{(1)}}(c)$.
The Rees algebra $\cR\big((p_{Z^{(1)}})_*\theo_{T^*_{Z^{(1)}}(c)}\big)$
associated to this filtration is a graded algebra
whose Proj gives a fiberwise compactification of $T_{Z^{(1)}}^*(c)$ to a projective bundle over $Z^{(1)}$.  In light of this,
we will write
 \bd
 \overline{T^*_{Z^{(1)}}(c)} = \bproj\, \cR\big((p_{Z^{(1)}})_*\theo_{T^*_{Z^{(1)}}(c)}\big),
 \ed
 and call it {\em the compactified twisted cotangent bundle} of $Z^{(1)}$.

As explained in Section \ref{azumaya property subsubsection}, the natural inclusion of 
$\cZ(\D_Z(\lambda))\cong (p_{Z^{(1)}})_*\theo_{T^*_{Z^{(1)}}(c)}$ into $\D_Z(\lambda)$ is compatible with
the filtrations if we put generators of $\theo_{T^*_{Z^{(1)}}(c)}$ in degree $p$.  Thus, it is natural to abuse notation and let
$\cR(\cZ(\D(\lambda)))$ denote the $p{\mathbb Z}$-graded algebra whose $p\ell$th graded piece is the intersection
$\cZ(\D(\lambda))\cap \D^{p\ell}(\lambda)$.
Summarizing: have an injective homomorphism
of $p{\mathbb Z}$-graded algebras (in the above sense)
$\cR(\cZ(\D(\lambda)))\longrightarrow \cR^{(p)},$
where $\cZ(\D(\lambda))$ denotes the center of $\D(\lambda)$.

\subsubsection{Order Property}
The Azumaya property of differential operators in characteristic $p$ has an analog for the compactified cotangent bundle as well.
Namely, the sheaf of algebras  $\cR^{(p)}$ is naturally a graded module over the graded ring $\cR\big(\cZ(\D(\lambda))\big)$.
It follows that, taking the associated sheaf on  $\overline{T^*_{Z^{(1)}}(c)}$, the graded algebra $\cR^{(p)}$
determines a sheaf
 $R = R_Z(\lambda)$ of $\theo$-central algebras on  the compactified twisted cotangent bundle
 $\overline{T^*_{Z^{(1)}}(c)}$.  Moreover, it is immediate from the description of the associated graded of $\D_Z(\lambda)$ and
 compatibility of the filtrations on $\D(\lambda)$ and its center that $R$ is a finite flat algebra on $\overline{T^*_{Z^{(1)}}(c)}$.
 Furthermore,  the restriction of $R$ to the uncompactified twisted cotangent bundle
 $T^*_{Z^{(1)}}(c)$ is $\D(\lambda)$, and in
 particular Azumaya.  It follows that $R$ is an {\em order} on $\overline{T^*_{Z^{(1)}}(c)}$ \cite{McConnell-Robson, Reiner}: it is
 a torsion-free coherent sheaf with an algebra structure whose restriction to the generic point of $\overline{T^*_{Z^{(1)}}(c)}$ is an
 Azumaya algebra.
 The order $R$ is ramified over the divisor at infinity
 $D_\infty^{(1)} = \bproj(\on{Sym}^\bullet T_{Z^{(1)}})$.

\subsubsection{Veronese Equivalence and Orders}
By the above discussion, we get equivalences of module categories:
\begin{equation}\label{R vs R}
\on{Qgr}\,\cR \xrightarrow{\sim} \on{Qgr}\,\cR^{(p)}\xrightarrow{\sim} R\on{-mod},
\end{equation}
where $R\on{-mod}$ means the category of
$\theo$-quasicoherent left $R$-modules; we also have similar equivalences for the categories of
finitely generated modules (that are quasicoherent over $\theo$) and for the categories of right modules.     Inverting $t^p$ in both $\cR$ and its center, the above equivalence localizes to the equivalence of
Section \ref{azumaya property subsubsection} between
$\D(\lambda)$-modules and modules over the corresponding Azumaya algebra.

These equivalences are compatible with direct image to $Z^{(1)}$ (where, for objects of $\on{Qgr}\,\cR$, this means taking direct image
to $Z$ and then direct image by Frobenius to $Z^{(1)}$).
\begin{remark}\label{shift remark}
On any category of graded modules there is a natural functor of ``shifting the grading by $1$.'' 
Because of the intervention of the Veronese functor in the equivalence \eqref{R vs R}, however, 
the shift-of-grading functor on $\on{Qgr}\,\cR$ does not
coincide with the twist by $\theo(D^{(1)}_\infty)$ on $R\on{-mod}$ (which is the natural shift-of-grading functor on $R\on{-mod}$).
\end{remark}

\section{Mirabolic Bundles and the Hitchin System}\label{parabolic bundles}
The geometry of moduli stacks of mirabolic bundles $\PB_n(X)$ and their twisted cotangent bundles $T^*_{\PB}(\alpha)$ plays a central role for us.
Indeed, using the features of twisted PD differential operators from the previous section, the study of twisted $\D_{\PB}$-modules, which is our
primary concern in this paper, is very close to the study of quasicoherent sheaves on $T^*_{\PB}(\alpha)$.  In this section we explain the features of the geometry of the stacks $T^*_{\PB}(\alpha)$ that we will need in the sequel.

\subsection{Basics of Mirabolic Bundles}
We begin with basic definitions and properties of parabolic bundles.
\subsubsection{Parabolic Bundles}
Fix a positive integer $n$.  We will consider {\em (quasi-)parabolic vector bundles} of rank $n$ on $X$.  In general, these consist of a vector bundle
$\cE$ of rank $n$ and a filtration
\bd
\cE = \cE_0 \subset \cE_1 \subset \dots \subset \cE_k \subset \cE(b).
\ed
We will be concerned only with the simplest case when $k=1$ and $\cE_1/\cE_0$ is a $1$-dimensional vector space: in other words, choosing $\cE_1$ amounts to choosing a line in $\cE(b)/\cE$ (or equivalently, in the fiber of $\cE$ at $b$).  We will refer to such
data as a {\em mirabolic bundle}.\footnote{``Parabolic bundle'' often means the structure we have defined together with a choice of some weights---we will not choose weights here.}
\begin{defn}
The moduli stack
$\PB_n(X)$ of mirabolic bundles parametrizes flat families of pairs $\cE\subset \cE_1$ of vector bundles on $X$ for which
$\cE_1/\cE$ is a line bundle over $\{b\}\times S$.
\end{defn}

\subsubsection{Moduli of Mirabolic Bundles}
The moduli stack $\PB_n(X)$ of  mirabolic bundles is, by the above description, a ${\mathbf P}^{n-1}$-bundle over the moduli stack $\Bun_n(X)$
of rank $n$ vector bundles on $X$; more precisely, it is the projective bundle of (lines in) the rank
$n$ vector bundle over $\Bun_n(X)$ whose
fiber over $[\cE]$ is the fiber $\cE_b$.  In particular, the moduli stack of mirabolic bundles comes equipped with a relatively ample line
bundle $\theo_{\PB}(1)$.

Let $\Pb_n(X)$ denote the moduli stack that parametrizes pairs $(\cE,v)$ consisting of a rank $n$ vector bundle $\cE$ and a nonzero
vector $v\in \cE_b(b) = \cE(b)/\cE$.  By forgetting from the vector $v$ to the line it spans, we get a morphism $\Pb_n(X)\rightarrow \PB_n(X)$ that
makes $\Pb_n(X)$ into a principal $\Gm$-bundle over $\PB_n(X)$.

\subsubsection{$\PB$ and $\Pb$}
The moduli
stack $\PB_n(X)$ comes equipped with a natural action of the commutative group stack $B\Gm$: given a scheme $S$, maps
$S\rightarrow \PB_n(X)$ and $S\rightarrow B\Gm$ correspond to choices of a family $\cE\subset\cE_1$ of mirabolic bundles on
$X\times S$ and a line bundle $L$ on $S$, respectively.  Then $L\cdot (\cE\subset \cE_1) = (L\otimes\cE\subset L\otimes \cE_1)$ determines
the action.
  The quotient
$\PB_n(X)/B\Gm$ has a dense open set that is an algebraic space: since $X$ has genus $g\geq 1$, the generic mirabolic bundle in
each degree is simple (i.e. has only scalar endomorphisms).
\begin{lemma}\label{pbs and pbs}
The composite morphism
\begin{equation}\label{PB composite}
\Pb_n(X)\rightarrow \PB_n(X)\rightarrow \PB_n(X)/B\Gm
\end{equation}
is an isomorphism of stacks.
\end{lemma}

\subsubsection{Universal Bundle}
By the definition of $\Pb_n(X)$, there is a universal object $(\cU, v)$ on $X\times \Pb_n(X)$; here $\cU$ is a vector bundle on
$X\times\Pb_n(X)$ and $v\in \cU|_{\{b\}\times\Pb}$ is a nonvanishing section.
Lemma \ref{pbs and pbs} shows that there is also a morphism
\bd
\PB_n(X)\xrightarrow{q} \Pb_n(X).
\ed
 Moreover, by the proof of Lemma \ref{pbs and pbs}, the pullback $(1_X\times q)^*(\cU,v)$ has the following description.  Given $S\rightarrow \PB_n(X)$ a scheme, we
 get a mirabolic bundle $\cE\subset \cE_1$ on $X\times S$.  Let $L = (\cE_1/\cE)^*$, a line bundle on $\{b\}\times S \cong S$.  Pulling
 this line bundle back to $X\times S$, we may tensor to obtain $L\otimes\cE$; this bundle comes equipped with a canonical choice of
 nonvanishing section $v\in (\cE_1/\cE)^*\otimes \cE_1/\cE = k$ in
 \bd
 L\otimes\cE_1/L\otimes\cE \subset L\otimes\cE(b)/L\otimes\cE,
 \ed
 and the construction of the lemma tells us that the pair $(L\otimes\cE, v)$ is canonically identified with the pullback
$(1_X\times q)^*(\cU,v)$.

\subsubsection{Determinant Bundle}
Let $\Det$ denote the line bundle on $\Pb_n(X)$ defined by $\on{det}(\cU_b)$, the top exterior power of the fiber of the universal
bundle $\cU$ over $b\in X$.  Our description above of the pullback $(1_X\times q)^*(\cU,v)$ then gives the following.  Over a point
$\cE\subset\cE_1$ of $\PB_n(X)$, the fiber of $q^*\Det$ is given by
\bd
(q^*\Det)_{(\cE\subset\cE_1)} = \on{det}\big((\cE_1/\cE)^*\otimes\cE_b\big) =  \big((\cE_1/\cE)^*\big)^{\otimes n}\otimes\on{det}(\cE_b).
\ed
\begin{defn}\label{det bundle}
Let $\Det$ denote the line bundle on $\PB_n(X)$ whose fiber over $\cE\subset \cE_1$ is 
$\big((\cE_1/\cE)^*\big)^{\otimes n}\otimes\on{det}(\cE_b)$.
\end{defn}
\begin{remark}\label{det and det}
By the above discussion, we have a canonical isomorphism $q^*\Det = \Det$ on $\PB_n(X)$, so our choice of notation is consistent.
\end{remark}

\subsection{Twisted Cotangent Bundles of $\PB$}\label{det twisting}
Let $\cE$ be a vector bundle on the curve $X$.  Fix an element $\alpha\in k\smallsetminus {\mathbb F}_p$.

\subsubsection{Twisted Higgs Fields}
Suppose $\cE$ comes equipped with a parabolic structure.  An endomorphism $M: \cE_b\rightarrow \cE_b$ is said to be {\em compatible with, respectively nilpotent with respect to, the parabolic structure} if it takes
$\cE_\ell/\cE_0$ into $\cE_\ell/\cE_0$, respectively $\cE_{\ell-1}/\cE_0$ for all $\ell$.
It is known that the cotangent bundle $T^*_{\PB_n(X)}$ is the moduli stack of mirabolic Higgs bundles $(\cE\subset\cE_1,\theta)$: such data consist of a mirabolic bundle $\cE\subset \cE_1$ together with a meromorphic Higgs field $\theta: \cE\rightarrow \cE\otimes \Omega^1(b)$ such that residue of $\theta$ at $b$ is nilpotent with respect to the mirabolic structure (see \cite{Mk} for details of this identification).  
We are interested in a particular twisting of the cotangent bundle $T^*_{\PB_n(X)}$, the {\em determinant twisting}, which we will describe
below.

Write $A(\alpha)^{cl}$ for the left $\theo$-module $\D^1_X(\theo(b)^{\otimes\alpha})$ thought of as a Lie algebroid equipped with the commutative Lie bracket and the symmetric
$\theo$-module structure---we consider $A(\alpha)^{cl}$ to be the ``classical limit'' of
$\D^1_X(\theo(b)^{\otimes\alpha})$.  
An {\em $\alpha$-twisted meromorphic Higgs field on $\cE$ with simple pole at $b$} is a homomorphism
$\omega: A(\alpha)^{cl}\otimes \cE \rightarrow \cE(b)$ that restricts to the natural inclusion $\theo\otimes \cE\hookrightarrow \cE(b)$ via the inclusion $\theo\hookrightarrow A(\alpha)^{cl}$.
We may define a residue of a twisted Higgs field by 
choosing a local splitting of the projection $A(\alpha)^{cl} \twoheadrightarrow T_X$ near $b$; we thus obtain from $\omega$  
a homomorphism $\tilde{\omega}: T_X\otimes \cE\rightarrow \cE(b)$ near $b$.  The residue at $b$ of $\tilde{\omega}$ does not depend
on the choice of local splitting, and we will denote it by $\on{Res}_b(\omega)$.

By Lemma \ref{id of TDO on X}, we have inclusions
\begin{equation}\label{inclusions}
T_X(-b)\oplus \theo \hookrightarrow A(\alpha)^{cl} \hookrightarrow T_X\oplus\theo(b).
\end{equation}
Given a meromorphic Higgs field $\theta: \cE\rightarrow \cE\otimes\Omega^1(b)$, we may tensor with $T_X$ and add the identity
map $\cE(b)\rightarrow \cE(b)$ to obtain a homomorphism $(T_X\oplus\theo(b))\otimes \cE \rightarrow \cE(b)$.  Restricting to
$A(\alpha)^{cl}$ via \eqref{inclusions}, we obtain a twisted meromorphic Higgs field $\on{tw}(\theta): A(\alpha)^{cl}\otimes \cE\rightarrow \cE(b)$.  Conversely, given a twisted meromorphic Higgs field $\omega: A(\alpha)^{cl}\otimes\cE \rightarrow \cE(b)$,
we may restrict to $T_X(-b)\subset A(\alpha)^{cl}$ via \eqref{inclusions} and tensor with $\Omega^1(b)$ to obtain a homomorphism
$\on{untw}(\omega): \cE\rightarrow \cE\otimes\Omega^1(2b)$ which in fact maps to $\cE\otimes\Omega^1(b)$.  
\begin{lemma}[\cite{flows}]\label{mero and twisted Higgs}
Fix a  bundle $\cE$.  
 The maps $\theta\mapsto \on{tw}(\theta)$, $\omega \mapsto \on{untw}(\omega)$ define bijections between:
 \begin{enumerate}
 \item Meromorphic
 Higgs fields $\theta$ with simple pole at $b$ and residue $\on{res}_b(\theta) = -\alpha\cdot I + M$.
 \item $\alpha$-twisted meromorphic Higgs fields $\omega$ with simple pole at $b$ and residue $\on{res}_b(\omega) = M$.
 \end{enumerate}
 \end{lemma}
 \noindent
 This can be checked by a local calculation using Lemma \ref{id of TDO on X}.

\subsubsection{Twisted Cotangent Bundle of $\PB$}
Let $\widetilde{\PB}_n(X)$ denote the moduli stack of triples $(\cE,\phi: \cE_b\xrightarrow{\sim}k^n, i)$ where $i\in k^n\smallsetminus\{0\}$ is a nonzero vector and
$\phi$ is a trivialization of the fiber of $\cE$ over the basepoint $b$.  This maps to $\PB_n(X)$ by forgetting $\phi$ and replacing
$i$ by the line through $\phi^{-1}(i)$: more precisely,
$\cE_1(-b)/\cE(-b)\subset \cE/\cE(-b)$ is this line.   Moreover, $\wt{\PB}_n(X)$ is a principal $GL_n\times\Gm$-bundle over $\PB_n(X)$.
Indeed, let $(g,t)\in GL_n\times\Gm$ act by $(g,t)\cdot (\cE,\phi, i) = (\cE,g\circ\phi, g\cdot t i)$.  This gives a simply
transitive action on the fibers of the projection $\wt{\PB}_n(X)\rightarrow \PB_n(X)$.

Alternatively, we have $\widetilde{\PB}_n(X) \cong \widetilde{\Bun}_n(X)\times (k^n\smallsetminus\{0\})$, the stack of
triples of a bundle $\cE$ with a trivialization of its fiber over $b$ and a nonzero vector in $k^n$.  From this point of view, it is easy
to identify $T^*_{\widetilde{\PB}_n(X)}$: indeed, $T^*_{\widetilde{\Bun}_n(X)}$ consists of bundles equipped with a trivialization of the
fiber over $b$ and a meromorphic Higgs field with a first-order pole at $b$ (and regular elsewhere) \cite{Do, Mk}.

The
moment map for the action of $\GL_n$ on $T^*_{\widetilde{\Bun}_n(X)}$ takes a triple $(\cE,\phi,\theta)$ of a bundle $\cE$ with trivialization $\phi$ of the fiber at $b$ and meromorphic Higgs field $\theta$ to the residue $\on{Res}_b(\theta)$ of the
Higgs field at $b$.  Consequently, the moment map $\mu$ for the $GL_n\times \Gm$-action on $T^*_{\widetilde{\PB}_n(X)}$ is as follows.
A pair consisting of $(\cE,\phi,\theta)$ and an element
\bd
(i, j)\in T^*(k^n\smallsetminus\{0\})\cong (k^n\smallsetminus\{0\})\times (k^n)^*
\ed
gives the moment value
\bd
\mu(\cE,\phi,\theta, i, j) = (\on{Res}_b(\theta) + ij, ji)\in \mathfrak{gl}_n^* \times k = \mathfrak{gl}_n\times k.
\ed
We are going to reduce at a scalar multiple $-\alpha\cdot d\psi$ of
$d\psi = (I, n)\in {\mathfrak gl}_n\times k$ (recall that we are assuming $p>n$ so that $n\neq 0$ in $k$).  The condition 
$ji = -\alpha \cdot n$ implies that $ij$ is a (rank one semsimple) matrix with
$-\alpha \cdot n$ as its unique nonzero eigenvalue and corresponding eigenvector $i$; in particular, the condition
$\on{Res}_b(\theta) +ij = -\alpha\cdot I$ then means that $\on{Res}_b(\theta)$ lies in $-\alpha$ times the  ``Calogero-Moser coadjoint orbit,'' i.e. the orbit of the
matrix $I+CM$ where 
\bd
CM=
\left(\begin{array}{cccc}
-n & 0 &\cdots& 0\\ 0&0&\cdots&0\\
\vdots&\vdots&\ddots&\vdots\\ 0&0&\cdots&0
\end{array}\right) .
\ed

 The concrete description of the moment map above, together with Lemma \ref{mero and twisted Higgs}, gives:
 \begin{prop}\label{twisted Higgs and reduction}
 The twisted cotangent bundle $T^*_{\widetilde{\PB}_n(X)}/\!\!/_{-\alpha\cdot (I, n)} GL_n\times\Gm$  is isomorphic to the moduli stack of data $(\cE\subset\cE_1, \omega)$ where
 $\cE\subset\cE_1$ is a mirabolic bundle of rank $n$ and
 $\omega: A(\alpha)^{cl}\otimes \cE\rightarrow \cE_1$ is an $\alpha$-twisted meromorphic Higgs field whose residue at $b$ is compatible with the mirabolic structure and lies in the conjugacy class of $-\alpha\cdot CM$.
 \end{prop}

 The invariant element $d\psi = (I, n)\in{\mathfrak gl}_n\times k$ is the derivative of the character
 \bd
\psi: GL_n\times \Gm\rightarrow \Gm, \hspace{2em} \psi(g, t) = \on{det}(g)\cdot t^n.
\ed
 In particular,
 by Lemma \ref{TDO compatibility}, the reduction
 $T^*_{\widetilde{\PB}_n(X)}/\!\!/_{-\alpha\cdot (I, n)} GL_n\times \Gm$  is equal to the twisted cotangent bundle
of $\PB_n(X)$ associated to the
$\Gm$-bundle
$\Gm\times_{\psi, \GL_n\times\Gm} \widetilde{\PB}_n(X)$
over $\PB_n(X)$
and the weight $\alpha\in k$.  To describe this line bundle more explicitly, note that the fiber of the $GL_n\times\Gm$-bundle
$\widetilde{\PB}_n(X)\rightarrow \PB_n(X)$ over $\cE\subset\cE_1$ is
\bd
\on{Isom}(\cE_b, k^n)\times \big((\cE_1/\cE)\setminus\{0\}\big) \subset \Hom(\cE_b, k^n)\times \cE_1/\cE.
\ed
 The $\Gm$-torsor associated
to this $GL_n$-torsor via the character $\psi$ then has fiber $\big(\wedge^n\cE_b^*\otimes \big((\cE_1/\cE)^*\big)^{\otimes n}\big)\setminus \{0\}$.    Summarizing, we conclude:

\begin{corollary}\label{det twisting cor}
Let $\Det$ denote the line bundle of Definition \ref{det bundle}.
Then the moduli stack of $\alpha$-twisted pairs $(\cE\subset\cE_1, \omega)$ as in Proposition \ref{twisted Higgs and reduction}
is isomorphic to the twisted cotangent bundle $T^*_{\PB}(\Det^{\otimes\alpha})$ of $\PB_n(X)$.
\end{corollary}
\begin{notation}
We will write $T^*_\PB(\alpha) = T^*_{\PB}(\Det^{\otimes\alpha})$.
\end{notation}

There is an analog of Lemma \ref{pbs and pbs} also for $T^*_{\PB}(\alpha)$ and $T^*_{\Pb}(\alpha)$.  Indeed, $B\Gm$ acts on
$T^*_\PB(\alpha)$ compatibly with the action on $\PB$ itself, and we have:
\begin{lemma}\label{pbs and pbs for cot}
We have: $T^*_\PB(\alpha)/B\Gm = T^*_{\Pb}(\alpha)$.  Moreover, the $\Gm$-gerbe
$T^*_\PB(\alpha)\rightarrow T^*_{\Pb}(\alpha)$ has a section defined in the same way as the section of Lemma \ref{pbs and pbs}.
\end{lemma}

\subsection{Relation to Cherednik Algebras}\label{cherednik algebras}
As we have stated in the introduction, when $X$ is a curve of genus $1$ (smooth or not) there is a close relationship between
the Cherednik algebra of type $A_n$ associated to $X$ and the sheaf of twisted differential operators $\D_\PB(\Det^{\otimes\lambda})$.

Suppose $X$ is an elliptic curve or an integral curve of genus $1$ with a node or cusp; we refer to these latter two cases as
``trigonometric'' and ``rational'' respectively.  
\begin{defn}\label{cherednik ss}
Let $\Pb_n(X,0)^{ss}$ denote the moduli stack of pairs $(\cE, v)$ 
on $X$ whose underlying bundle $\cE$ is semistable of rank $n$ and degree $0$; if $X$ is singular, we impose the additional condition
that the pullback $n^*\cE$ is trivial, where $n: {\mathbf P}^1\rightarrow X$ is the normalization  of $X$.
\end{defn}  
This moduli stack is identified, via a Fourier-Mukai transform, with
the stack of length $n$ torsion sheaves $Q$ on $X$ equipped with a nonzero section
$\theo\rightarrow Q$, with the further proviso that if $X$ is singular then $Q$ should be supported on the smooth locus of $X$.
In the trigonometric and rational cases, this stack is thus identified with the quotients $(GL_n\times (k^n\smallsetminus \{0\}))/GL_n$
and $({\mathfrak gl}_n\times (k^n\smallsetminus \{0\}))/GL_n$ respectively.

Whether $X$ is smooth or singular, we have the space $\widetilde{\PB}_n(X,0)^{ss}$ as in Section \ref{det twisting}.  Under
Fourier-Mukai transform, this space is identified with the space
 parametrizing triples $(Q, \theo\rightarrow Q, \phi)$ where $Q$ and $\theo\rightarrow Q$ are as
above and $\phi$ is a choice of basis of $H^0(Q)$; this space is the open subset
$X^0_n = \on{rep}_X^n \times (k^n\smallsetminus \{0\})$ of the space denoted
$X_n = \on{rep}_X^n \times k^n$ in \cite{FG}.  This latter space is a scheme: indeed, $\on{rep}_X^n$ is easily seen
to be an open subset of the Quot-scheme on $X$ parametrizing quotients $\theo_X^n\rightarrow Q$ of length $n$, which also shows
immediately that it is smooth.  As we saw in Section \ref{det twisting}, $\widetilde{\PB}_n(X,0)^{ss}$ is a
principal $GL_n\times \Gm$-bundle over $\PB_n(X,0)^{ss}$; it is also a principal $GL_n$-bundle over
$\Pb_n(X,0)^{ss}$.  The quotient $X_n^0/\Gm$ is denoted by ${\mathfrak X}_n$ in
\cite{FG}.

We then have the following relationship between twisted differential operators on $\PB_n(X)$ and the spherical Cherednik algebra
of $X$:

\begin{thm}[Translation of \cite{FG}, Theorem 3.3.3]\label{FG theorem}
Let
\bd
\PB_n(X,0)^{ss}\xrightarrow{\pi} \on{Bun}_n(X,0)^{ss}\cong (X^{sm})^{(n)}
\ed
 denote the projection to the moduli space of semistable, degree $0$
vector bundles on $X$ (whose Fourier-Mukai transforms are torsion sheaves supported on the smooth locus of $X$, if $X$ is singular).
Then $\pi_*\D_{\PB_n(X,0)^{ss}}(\kappa)$ is the spherical Cherednik algebra associated to $X$ and the symmetric group $S_n$
with parameter $\kappa$.
\end{thm}
\begin{proof}
It is proven in Theorem 3.3.3 of \cite{FG} that the quantum Hamiltonian reduction of $\D_{\widetilde{\PB}}$, twisted by $\kappa$ times
 the
determinant character  as in Section
\ref{basics of TDOs}, gives the spherical Cherednik algebra of $X$ associated to the symmetric group $S_n$ and the value $\kappa$:
more precisely, this quantum Hamiltonian reduction sheafifies over the moduli {\em space} of semistable degree $0$ vector bundles
over $X$, i.e. the $n$th symmetric product of the (smooth locus of) $X$, and this sheaf is the spherical Cherednik algebra.
In view of Lemma \ref{TDO compatibility} and our description of the determinant line bundle $\Det$ in Section  \ref{det twisting},
it follows that the direct image to $(X^{sm})^{(n)}$ of $\D_{\Pb_n(X,0)^{ss}}(\Det^{\otimes\kappa})$ is identified with
the spherical Cherednik algebra.  The same then follows for the direct image of $\D_{\PB_n(X,0)^{ss}}(\Det^{\otimes\kappa})$ by
Section \ref{stack vs space} and Remark \ref{det and det}.
\qed\end{proof}

\subsection{Hitchin Systems for Mirabolic Bundles}\label{Hitchin systems section}
Let $\alpha\in k\smallsetminus {\mathbb F}_p$.  The $\alpha$-twisted cotangent bundle $T^*_\PB(\alpha)$ is a symplectic stack and admits an integrable system structure of ``Hitchin type."

More precisely, let $p_X: \overline{T^*_X(\alpha)}\longrightarrow X$ denote the compactified $\alpha$-twisted cotangent bundle of $X$.
This is the ruled surface over $X$ associated to the rank $2$ vector bundle $A(\alpha)^{cl}$ defined above, i.e.
associated to the vector bundle $D^1_X(\theo(b)^{\otimes\alpha})$ (as a left $\theo_X$-module):
$\overline{T^*_X(\alpha)} = \bproj\big(\on{Sym}^\bullet(A(\alpha)^{cl})\big)$.
  Since $A(\alpha)^{cl}$ comes equipped with a quotient map $A(\alpha)^{cl}\twoheadrightarrow T_X$ (it is a
Picard Lie algebroid on $X$), it comes equipped with a canonical section, the {\em infinity section}
$X_\infty \subset \overline{T^*_X(\alpha)}$,  corresponding to the quotient
$\on{Sym}^\bullet(A(\alpha)^{cl})\twoheadrightarrow \on{Sym}^\bullet(T_X)$ of sheaves of graded algebras on $X$.

A point of $T^*_\PB(\alpha)$ corresponding to a
meromorphic Higgs pair $(\cE, \theta)$
 determines a lift of $\cE$ to a sheaf on $\overline{T^*_X(\alpha)}$ whose restriction
to the section at infinity $X_\infty\subset \overline{T^*_X(\alpha)}$ is isomorphic to $\theo_b$.  Conversely, if $\cF$ is a sheaf of pure dimension $1$ on $\overline{T^*_X(\alpha)}$
that is finite over $X$ of rank $n$ and for which $\cF|_{X_\infty}\cong \theo_b$, then $\cF$
determines a unique point $(\cE,\theta)$ of $T^*_\PB(\alpha)$ for which $\cE = (p_X)_*\cF$; see \cite{perverse} for a summary.
 We will call such sheaves $\cF$ {\em spectral sheaves} or {\em $\alpha$-twisted spectral sheaves}.
\begin{lemma}\label{linear series}
If $\cF$ is a spectral sheaf that is a rank $1$ torsion-free sheaf over a reduced support curve $\Sigma$, then the
curve $\Sigma$ lies in the linear system
\bd
|n(X_\infty + p_X^*K_X) + F_b|
\ed
 where $F_b$ denotes the fiber over $b$.
\end{lemma}
\begin{proof}
As in
\cite[Section~3]{perverse}, up to a twist $\cF$ is the cokernel of a map on $\overline{T_X^*(\alpha)}$ given in terms of Higgs data by
\bd
\theo(-X_\infty)\otimes p_X^*(T_X\otimes\cE)\rightarrow p_X^*\cE_1.
\ed
Taking determinants and using that $\det(\cE_1)\otimes\det(\cE)^* \cong \theo(b)$, we find that the determinant line bundle of
$\cF$ is $\theo(nX_\infty)\otimes p_X^* K_X^{\otimes n}(b)$, as claimed.
\qed\end{proof}
There is a {\em Hitchin map} $T^*_\PB(\alpha)\xrightarrow{h} \Ht$ where $\Ht$ is the {\em Hitchin base}
for $T^*_\PB(\alpha)$ (cf. \cite{Yokogawa, Mk}). Indeed, as in the proof of the lemma, the (``Koszul'') presentation of $\cF$ gives a section of the line bundle
$\det(\cF) =\theo(nX_\infty)\otimes p_X^* K_X^{\otimes n}(b)$; this gives a map to the linear series, the image of which lies in
the set of curves that do not contain $X_\infty$ as a component.  The subspace of the linear system consisting of such
curves forms an affine space
\bd
\Ht \subset |n(X_\infty + p_X^*K_X) + F_b|,
\ed
 which is the Hitchin base.  We have:
 \bd
 \on{dim}(T^*_\PB(\alpha)) = 2n^2(g-1) + 2n-1, \hspace{2em} \on{dim}(T^*_{\Pb}(\alpha)) = 2n^2(g-1) + 2n,
 \ed
 \bd
 \on{dim}(\Ht) = n^2(g-1) +n.
 \ed
 \begin{defn}

Let $\Ht\gen$ denote the open subset of $\Ht$ that parametrizes smooth curves $\Sigma$ such that, in a neighborhood of
$b\in \Sigma\cap X_\infty$, the map $\Sigma\rightarrow X$ is \'etale; we will prove the existence of such curves in
Section \ref{existence of generic}.  Let
\bd
T^*_\PB(\alpha)\gen = h^{-1}(\Ht\gen);
\ed
 we will refer to this as the {\em generic locus} of $T^*_\PB(\alpha)$.  This generic locus parametrizes Higgs data that correspond to line
 bundles on smooth spectral curves that are \'etale over $X$ near $b$.
 We let
$\Sigma/\Ht$ denote the universal spectral curve; we will refer to its open subset $\Sigma/\Ht\gen$ as the {\em generic
spectral curve}.
\end{defn}

  The morphism $\Sigma\rightarrow \Ht\gen$ is smooth (proof: the universal family of a linear series is always flat;
the conclusion then follows from \cite[Theorem~III.10.2]{Hartshorne}).

It is known that the Hitchin map $h$ is a Lagrangian map
for the (canonical) symplectic structure on $T^*_\PB(\alpha)$.  Furthermore, the map $T^*_\PB(\alpha)\gen\rightarrow \Ht\gen$ is
smooth and is isomorphic over $\Ht\gen$ to the relative Picard stack $\Pic(\Sigma/\Ht\gen)\rightarrow \Ht\gen$.  In particular,
$T^*_\PB(\alpha)\gen$ comes equipped with a structure of smooth commutative group stack over $\Ht\gen$.  After any base change
$\wt{\Sigma} = \Sigma\times_{\Ht\gen} \wt{\Ht}\gen$ and choice of a section of $\wt{\Sigma}/\wt{\Ht}\gen$,
we get an isomorphism of group stacks (that depends on the choice of section):
\bd
\Pic(\wt{\Sigma}/\wt{\Ht}\gen) \cong B\Gm\times \Jac(\wt{\Sigma}/\wt{\Ht}\gen)\times {\mathbb Z},
\ed
 where $\Jac(\wt{\Sigma}/\wt{\Ht}\gen)$ is the
Jacobian variety of $\wt{\Sigma}$ (see \cite{BB}, Section 2.4, Example 4).

\begin{lemma}
Suppose $\cF$ is a line bundle on a generic spectral curve $\Sigma$.  Then $\cE = (p_X)_*\cF$ is a vector bundle on $X$ of degree
\begin{equation}\label{degree of image}
\deg((p_X)_*\cF) = 1-n + (g(X)-1)\cdot (n-n^2) + \deg_\Sigma(\cF).
\end{equation}
\end{lemma}

\subsection{Existence of Generic Spectral Curves}\label{existence of generic}
We next give a new proof of the existence of generic spectral curves in this setting: earlier proofs in the literature seem to work only
in characteristic zero.

\begin{prop}\label{nonemptiness}
Suppose that $p>n$ and let $\alpha\in k\smallsetminus {\mathbb F}_p$.  Consider the stack of $\alpha$-twisted spectral
sheaves.  Then
there exist points $s\in\Ht$ for which the corresponding spectral curve $\Sigma_s$ is smooth and is \'etale over $X$ near $b$.
In other words, $\Ht\gen$ is
nonempty. 

 In particular, if $\lambda \in k\smallsetminus {\mathbb F}_p$ and $c=\lambda^p-\lambda$, there are $c$-twisted spectral sheaves on $\overline{T^*_{X^{(1)}}(c)}$ whose spectral curve is smooth and is \'etale over $X$ near $b$.  
\end{prop}
\begin{remark}
This is the crucial point at which we use $\lambda\notin {\mathbb F}_p$: indeed, in genus $1$ (which seems most interesting for applications) there are no smooth spectral curves of the type we are considering when
$c= \lambda^p-\lambda = 0$.
\end{remark}
\begin{proof}[Proof of Proposition \ref{nonemptiness}]
 There is
a unique-up-to-scalars nonzero class in $\on{Ext}^1_X(T_X,\theo_X)$; hence the surface
$\overline{T^*_X(\alpha)}$ does not depend (up to isomorphism) on $\alpha$ provided that $\alpha\neq 0$.

To prove the proposition, we use the characteristic $p$ Bertini Theorem as it appears in \cite[Theorem~6.3]{Jou}: if the morphism
defined by a basepoint-free linear system is unramified, then almost every member of the linear system is smooth.  The linear system
\bd
D = n(X_\infty + p_X^*K_X) + F_b
\ed
is {\em not} basepoint-free on all of $\overline{T_X^*(\alpha)}$, however.  It has a section that vanishes along $X_\infty\cup F_b \cup p_X^*E$
for any effective canonical divisor $E$ on $X$, so in particular it has no basepoints on $T^*_X(\alpha)\setminus F_b$.  Moreover, if
$C\in |D|$ is any curve that does not contain $X_\infty$ as a component, then the intersection number $C\cdot X_\infty$ is $1$, and
we may conclude that for such a $C$, $C\cap X_\infty  = \{b\}$.  Finally, we will see below that $|D|$ has no basepoints on
$F_b$ other than $b \in X_\infty\cap F_b$; it follows  that the base locus is a subset of $X_\infty\cap F_b$, and we will see that
even scheme-theoretically $\on{Bs}|D| = \{b\}$.  However, since a general $C\in |D|$ has intersection number $1$ with $X_\infty$,
a general $C$ is smooth at $\{b\}$, so we may remove the base locus from consideration---in fact, we may restrict attention to the open
surface $T_X^*(\alpha)$.

To prove that the generic $C\in |D|$ is smooth, then,
it remains to show that the linear system is unramified over $T_X^*(\alpha)$, i.e. for every $x\in T_X^*(\alpha)$
and every length 2 closed subscheme $S$ of $T_X^*(\alpha)$ supported at the single point $x$, the restriction map
\begin{equation}\label{restr for unramified}
H^0(\theo(D))\rightarrow H^0(\theo(D)|_S)
\end{equation}
 is surjective.  In fact, when $X$ has genus $g=1$, we will show something weaker,
namely that the restriction map is surjective except when $\overline{x}$ lies in one of finitely many fibers of $p: T_X^*(\alpha)\rightarrow X$.
Our argument will show that the generic curve $C\in |D|$ is smooth in a neighborhood of each of those fibers, so the weaker statement
will suffice.

\vspace{.3em}

\noindent
{\em First Case. $X$ has genus $g\geq 2$.}\;   Given $0\leq \ell \leq n$, write
\bd
V_{n-\ell} = (p_X)_*\theo(D-\ell X_\infty) = (p_X)_*\theo\left((n-\ell)X_\infty + (p_X)^*K_X^{\otimes n}(b)\right).
\ed
We will also write $V = V_n  = (p_X)_*\theo(D)$. 
We have $V_{n-\ell}/V_{n-\ell-1} \cong K_X^{\otimes\ell}(b)$.  

Suppose that $S \subset \overline{T_X^*(\alpha)}$ is a length 2 closed subscheme supported at a point $x$.  Let $\overline{x}\in X$
denote the image of $x$ under the projection $p_X$.  Let $2\overline{x}$ denote the closed subscheme of $X$ corresponding to the
divisor with the same notation, and let $F_{2\overline{x}}$ denote the scheme-theoretic fiber of the projection $p$ over this closed subscheme.
\begin{claim}\label{vanishing claim}
$H^1(V_{n-2}(-2\overline{x})) = 0$.
\end{claim}
To prove the claim, we note that $V_{n-2}(-2\overline{x})$ has a filtration with subquotients
$K^{\otimes\ell}(b-2\overline{x})$, $\ell\geq 2$.  So it suffices to show
$H^1\left(K^{\otimes\ell}(b-2\overline{x})\right) = 0$ for $\ell\geq 2$.  This follows using Riemann-Roch (the degree of
this bundle is too large to have nonzero $H^1$) when $g\geq 2$ and $\ell \geq 2$.

It follows from the claim that we have an exact sequence
\begin{equation}\label{surjectivity}
0\rightarrow H^0(V_{n-2}(-2\overline{x})) \rightarrow H^0(V_{n-2}) \rightarrow V_{n-2}\otimes \theo_{2\overline{x}} \rightarrow 0.
\end{equation}
 Note that $H^0(V_{n-\ell}\otimes\theo_{2\overline{x}}) = H^0(\theo(D-\ell X_\infty)|_{F_{2\overline{x}}}$.

Now, consider the commutative diagram
\bd
\xymatrix{H^0(\theo(D-2X_\infty)) \ar[r]\ar[d] & H^0(\theo(D-2X_\infty)|_{F_{2\overline{x}}}) \ar[r]\ar[d] & H^0(\theo(D-2X_\infty)|_S)
\ar[d]^{=}\\
H^0(\theo(D))\ar[r] & H^0(\theo(D)|_{F_{2\overline{x}}}) \ar[r] & H^0(\theo(D)|_S).}
\ed
The right-hand vertical arrow is an isomorphism since $S\cap X_\infty = \emptyset$.  The top-left horizontal arrow is surjective
by the exactness of \eqref{surjectivity}.
If $n\geq 3$, the top right horizontal arrow is surjective because $\theo(D-2X_\infty)|_{F_{2\overline{x}}}$ is very ample.  It follows
that in this case, the composite $H^0(\theo(D-2X_\infty))\rightarrow H^0(\theo(D)|_S)$ is surjective, which implies that
\eqref{restr for unramified} is surjective as well.

If $n=2$, the same argument works provided $S$ is not a closed subscheme of
the fiber $F_{\overline{x}}$, because $\theo(D-2X_\infty)|_{F_{2\overline{x}}}$ is still globally generated.   If, however, $n=2$ and
$S$ is a closed subscheme of the fiber $F_{\overline{x}}$ (in other words, $S$ corresponds to a ``vertical tangent direction at
$x$'') we will argue slightly differently.  In that case, a similar argument to that for Claim \ref{vanishing claim} shows that
$H^1(V_{n-1}(-\overline{x})) = 0$ for $\overline{x}\neq b$, and hence that $H^0(\theo(D-X_\infty))\rightarrow H^0(\theo(D-X_\infty)|_{F_{\overline{x}}})$ is surjective for $\overline{x}\neq b$.  Then we get a commutative diagram
\bd
\xymatrix{H^0(\theo(D-X_\infty)) \ar[r]\ar[d] & H^0(\theo(D-X_\infty)|_{F_{\overline{x}}}) \ar[r]\ar[d] & H^0(\theo(D-X_\infty)|_S)
\ar[d]^{=}\\
H^0(\theo(D))\ar[r] & H^0(\theo(D)|_{F_{\overline{x}}}) \ar[r] & H^0(\theo(D)|_S),}
\ed
and, arguing as in the previous paragraph,
we conclude that the linear system $|D|$
is unramified at points $x$ not lying in the fiber $F_b$.

Finally, it suffices to treat the fiber $F_b$.  Recall that we are assuming that $n=2$, so $V= V_2$.
We claim that the natural map
\begin{equation}\label{map of H1s}
H^1(V(-b)) = H^1(V_2(-b)) \rightarrow H^1((V_2/V_1)(-b))
\end{equation}
 is an isomorphism; it is surjective since the next term in the long exact cohomology sequence is an $H^2$.    Note first that
$V_2/V_1 \cong \theo(b)$, so $h^1((V_2/V_1)(-b)) = g$.  Moreover, $h^1(V_2(-b))=g$: indeed, the Serre dual cohomology group
is $H^0(K\otimes \D^2(\theo(b)^{\otimes \alpha}))$.  This vector bundle has a filtration with subquotients $K$, $\theo$, and $T_X$, hence we get an exact
sequence
\bd
0\rightarrow H^0(K)\rightarrow  H^0(K\otimes \D^2(\theo(b)^{\otimes \alpha}))\rightarrow H^0(\theo).
\ed
But if the right-hand map were nonzero, we could split the quotient map $K\otimes \D^1(\theo(b)^{\otimes \alpha})\rightarrow \theo$,
which is impossible if $\alpha\neq 0$.  Thus, we may conclude that $h^0(K\otimes \D^2(\theo(b)^{\otimes \alpha})) = g$, i.e. that
$h^1(V_2(-b))=g$.  Thus, the surjective map \eqref{map of H1s} has domain and target of the same dimension, so it is an isomorphism.
Now, consider the commutative diagram:
\bd
\xymatrix{H^0(V_1)\ar[r]\ar[d] & H^0((V_1)_b)\ar[r]\ar[d] & H^1(V_1(-b))\ar[d] \\
H^0(V) \ar[r] & H^0(V_b)\ar[r] & H^1(V(-b)).}
\ed
The right-hand vertical arrow is followed  in the long exact cohomology sequence by the isomorphism \eqref{map of H1s}, so the
right-hand vertical map is zero.  It follows that the image of $H^0((V_1)_b)\rightarrow H^0(V_b)$ is contained in the image of
$H^0(V)\rightarrow H^0(V_b)$.  But the image of $H^0((V_1)_b) \cong H^0(\theo(D-X_\infty)|_{F_b})$ consists exactly of sections of
$\theo(D)|_{F_b} \cong \theo_{F_b}(2)$ that vanish at $b\in F_b\cap X_\infty$.  In other words, for any point $s \in F_b$, there is
a curve $C\in |D|$ such that the scheme-theoretic intersection $C\cap F_b$ is exactly the divisor $s+b$.  It follows that the generic
curve $C$ is smooth in a neighborhood of $F_b$ and, furthermore,  is \'etale over $X$ near $b$.

\vspace{.3em}

\noindent
{\em Second Case.  $X$ has genus $g=1$.}\;  The first part of the proof in genus $1$ is similar to the last part of the proof for
higher genus.  We begin by considering the map
\begin{equation}\label{map in genus 1}
H^0(\theo(D)) = H^0(V)\rightarrow H^0(V_{\overline{x}}) = H^0(\theo(D)|_{F_{\overline{x}}}).
\end{equation}
\begin{claim}\label{in the non b fiber}
If $\overline{x} \neq b$, then \eqref{map in genus 1} is surjective.
\end{claim}
Indeed, since $V_k/V_{k-1} \cong \theo(b)$ for all $k$ ($K_X$ is trivial),
$H^1((V_k/V_{k-1})(-\overline{x})) = 0$.  An inductive argument then shows $H^1(V(-\overline{x}))=0$, and the conclusion follows
from the long exact cohomology sequence.

\begin{claim}\label{in the b fiber}
If $\overline{x} = b$, then
$\on{Im}\left(H^0(\theo(D))\rightarrow H^0(\theo(D)|_{F_b})\right) = H^0(\theo(D-X_\infty)|_{F_b})$.
\end{claim}
\begin{proof}[Proof of Claim \ref{in the b fiber}]
We have $V_n/V_{n-1}\cong \theo(b)$, so the map
$H^0(V_n/V_{n-1})\rightarrow H^0((V_n/V_{n-1})_b)$ is zero.  Consequently,
$H^0(V)\rightarrow V_b/(V_{n-1})_b$ is zero, and the image of $H^0(V)$ in
$H^0(V_b)$ lies inside the fiber $(V_{n-1})_b$.

Now, $V(-b)$ is the unique (up to isomorphism) indecomposable bundle of rank $n+1$ with a filtration with subquotients isomorphic
to $\theo$; this bundle is known to have $h^1(V(-b)) = 1$.  Hence $H^0(V)\rightarrow V_b$ has at most $1$-dimensional cokernel.
From the previous paragraph, we may conclude that
$\on{Im}(H^0(V)\rightarrow V_b) = (V_{n-1})_b$, as desired.  This proves the claim.
\qed\end{proof}
Now, from Claim \ref{in the non b fiber}, it follows that, for any $\overline{x}\neq b$ and effective degree $n$ divisor $E$ on the fiber
$F_x\cong {\mathbf P}^1$, there exists a curve $C\in |D|$ such that $C\cap F_{\overline{x}} = E$ scheme-theoretically.
Furthermore, from Claim \ref{in the b fiber}, it follows that, for any effective degree $n$ divisor $E$ on $F_b\cong {\mathbf P}^1$ that
contains $b =  F_B\cap X_\infty$ with multiplicity at least $1$, there exists $C\in |D|$ such that $C\cap F_b = E$ scheme-theoretically.
Combining these conclusions, we have first that
 $\on{Bs}|D| = \{b\}$ scheme-theoretically, and second:
 \begin{fact}\label{generic curve and fiber}
For any $x\in X$, the generic curve $C\in |D|$ is \'etale over $X$ in a neighborhood of the fiber $F_x$.
\end{fact}
We are now ready to complete the proof in the genus $1$ case.  As we mentioned before, we will prove that the linear series
$|D|$ is unramified except perhaps in finitely many fibers $F_x$, the fibers over the 2-torsion points of $X$.
We may then conclude, using Bertini's Theorem,
 that the generic $C\in |D|$ is smooth
except perhaps at points that lie in those fibers.  But applying Fact \ref{generic curve and fiber} to those fibers, the generic curve
$C\in |D|$ is also smooth in those fibers, hence is smooth everywhere.
Finally, the generic $C\in |D|$ is \'etale over $X$ near $F_b$ by
Claim \ref{in the b fiber}.

To see that $|D|$ is unramified except possibly in the fibers over 2-torsion points of $X$,
we argue as follows.  First, note that it follows from
Fact \ref{generic curve and fiber} that for any $x\in T_X^*(\alpha)$ and length two subscheme $S$ of $T_X^*(\alpha)$ supported at $x$
and lying scheme-theoretically in the fiber $F_{p(x)}$, the map $H^0(\theo(D))\rightarrow H^0(\theo(D)|_S)$ is surjective.
So it suffices to consider $S$ that are not supported scheme-theoretically in a fiber, i.e. $S$ for which the map $S\rightarrow X$
induces an isomorphism on tangent spaces, $T_xS \rightarrow T_{p(x)}X$.

Note also that it suffices to consider the case $n=2$: indeed, we have a commutative diagram
\bd
\xymatrix{
H^0(\theo(2X_\infty+F_b))\ar[r]\ar[d] & H^0(\theo(2X_\infty+F_b)|_S)\ar[d]^{=} \\
H^0(\theo(nX_\infty+F_b))\ar[r] & H^0(\theo(nX_\infty+F_b)|_S)}
\ed
with the right-hand vertical arrow an isomorphism (since $S$ is supported away from $X_\infty$); so it suffices to prove surjectivity of the top arrow.

Let $A = p_*\theo(X_\infty)$.
We have a 1-parameter family of maps $\theo(-b)\xrightarrow{\phi_\mu} A$ making
\bd
\xymatrix{\theo(-b)\ar[r]\ar[dr] & A\ar[d]\\
 & \theo}
 \ed
 commute; these maps correspond to sections $X\xrightarrow{s_\mu} \overline{T_X^*(\alpha)}$ of the ruled surface such that
 $X_\infty \cap s_\mu(X)= \{b\}$; indeed, for each such section $s_\mu$, $s_\mu(X)$ is linearly equivalent to
 $X_\infty+F_b$.
Choosing one such section, $s_0$, we get an isomorphism $\theo^2 \xrightarrow{\simeq} A|_{X\smallsetminus\{b\}}$ over
 $X\smallsetminus\{b\}$ given by the pair of sections $1$ and $s_0$ (here we are using the natural inclusion $\theo\hookrightarrow A$
 to obtain the section $1$).

We are now ready to describe the restriction map $H^0(\theo(D))\rightarrow H^0(\theo(D)|_S)$ a bit more concretely.
A choice of a subscheme $S$ lying over a point $\overline{x}\in X\smallsetminus\{b\}$ corresponds to a choice of:
\begin{enumerate}
\item A point $a\cdot 1+ s_0(\overline{x}) \in F_{\overline{x}}$ (i.e. with homogeneous coordinates $[a:1]$) for some choice of constant $a$.  Note that the point $[1:0]$ lies on $X_\infty$ and so we may ignore it.
\item A ``first-order variation of $a$ near $\overline{x}$'': that is, if $z$ is a uniformizer in $\theo_{X,\overline{x}}$, a choice of an element
$a(z) \in \theo_{X,\overline{x}}/{\mathfrak m}_{\overline{x}}^2$ with $a(z) = a + a_1z$ such that $S$ is given by
$a(z)\cdot 1 + s_0(z)$.
\end{enumerate}
Now, the sections $1$ and $s_0$ of $A$ give global sections of $S^2(A)(b) = p_*\theo(2X_\infty + F_b)$.  Using the above
description of $S$ and using the constant section $1$ to trivialize $\theo(D)|_S$, their restrictions to
$S$ are identified with the functions $1$ and $-a(z)$ respectively.  These generate $\theo(D)|_S$ as a vector space
unless $a(z)$ is constant
modulo ${\mathfrak m}_{\overline{x}}^2$, i.e. unless $S$ is identified with the first-order neighborhood of a point in one of the
sections $s_\mu(X)$.

So, finally, we may suppose that $S$ is the first-order neighborhood of a point in one of the sections $s_\mu(X)$, thought of
as a length 2 closed subscheme of $T_X^*(\alpha)$.  Write $C=s_\mu(X)$. Since $C\sim X_\infty + F_b$, identifying $C$ with $X$ via
the isomorphism $s_\mu$ gives
$\theo(D)|_C = \theo(2X_\infty\cdot F_b + X_\infty\cdot F_b) = \theo(3b)$.  Moreover, we have
$\theo(D-C) = \theo(X_\infty)$.  Thus $H^0(\theo(D-C))\cong k \cong H^1(\theo(D-C))$.  The map
$H^0(\theo(D))\rightarrow H^0(\theo(D)|_C)$ thus has (at most) 1-dimensional cokernel, and since $b$ is a basepoint of
$|D|$ we find that
\bd
\on{Im}\left[ H^0(\theo(D))\rightarrow H^0(\theo(D)|_C) = H^0(X,\theo(3b))\right]
\ed
is identified with $H^0(X,\theo(2b))$.  This surjects onto $H^0(\theo(2b)\otimes \theo/\theo(-2\overline{x}))$ except when
$\overline{x}$ is a 2-torsion point of $X$.  So, $H^0(\theo(D))\rightarrow H^0(\theo(D)|_S)$ is surjective except possibly when $S$
is supported in a fiber $F_{\overline{x}}$ over a 2-torsion point $\overline{x}$ of $X$.

As explained above, this completes the proof
of Proposition \ref{nonemptiness}.
\qed\end{proof}

\subsection{A Section of the Hitchin System}\label{unit section section}
We again fix $\alpha \in k\smallsetminus {\mathbb F}_q$.  
\begin{defn}\label{unit section}
Let $u:\Ht\rightarrow T^*_{\PB}(\alpha)$ denote the section whose value on $y\in \Ht$ is the line bundle
$\theo_{\Sigma_y}$ on the spectral curve $\Sigma_y$---we will call this the {\em unit section}.
More generally, let
$u_m: \Ht\rightarrow T^*_\PB(\alpha)$ denote the section whose value on $y\in\Ht$ is
$\theo_{\Sigma_y}\otimes p_X^*\theo_X(mb)$, where $p_X:\Sigma\rightarrow X$ is the projection.
We then have $u_0 = u$.
\end{defn}
\begin{prop}\label{unit section is trivial}
The image of $u(\Ht^\circ)$ and, more generally,
$u_m(\Ht^\circ)$ under the projection $T^*_{\PB}(\alpha)\rightarrow \PB_n$ consists of a single point of $\PB_n(X)$.
\end{prop}
\begin{proof}
We will write $S=\overline{T^*_X(\alpha)}$ and $A:= A(\alpha)^{cl}$.

Note that the case of arbitrary $m$ follows immediately from the case $m=0$: indeed,
\bd
(p_X)_*\theo_\Sigma(\ell X_\infty)\otimes p_X^*\theo_X(mb)\cong \big((p_X)_*\theo_\Sigma(\ell X_\infty)\big)\otimes\theo_X(mb)
\ed
by the projection formula, so the image of $u_m(\Ht^\circ)$ in $\PB_n(X)$ is obtained from the image of $u(\Ht^\circ)$ by
twisting by $\theo_X(mb)$.

\begin{lemma}
The relative canonical sheaf $K_{S/X}$ satisfies $K_{S/X} \cong \theo(-2X_\infty)\otimes p_X^*T_X$.
\end{lemma}

\begin{lemma}
Let $\Sigma$ be a generic spectral curve of degree $n$ over $X$.
We have:
\begin{equation}\label{direct images}
\begin{split}
{\mathbf R}^1(p_X)_*\theo(-\Sigma) & = (S^{n-2}A)^*\otimes T_X^{n-1}(-b)\\
{\mathbf R}^1(p_X)_*\theo(X_\infty-\Sigma) & = (S^{n-3}A)^*\otimes T_X^{n-1}(-b).
\end{split}
\end{equation}
Moreover,
\begin{equation}\label{Exts}
\begin{split}
\Ext^1_X({\mathbf R}^1(p_X)_*\theo(-\Sigma), \theo) & = 0,\\
\Ext^1_X({\mathbf R}^1(p_X)_*\theo(X_\infty-\Sigma), \theo) & = 0.
\end{split}
\end{equation}
\end{lemma}
\begin{proof}[Proof of Lemma]
By Lemma \ref{linear series}, $\theo(\Sigma) = \theo(nX_\infty)\otimes p_X^*K_X^n(b)$.  By duality we get
${\mathbf R}^1(p_X)_*\theo(-\Sigma) \cong [(p_X)_*\theo(\Sigma)\otimes K_{S/X}]^*$, which, by the previous lemma
 and the formula for $\theo(\Sigma)$, equals
\bd
\left[(p_X)_*\theo((n-2)X_\infty)\otimes K_X^{n-1}(b)\right]^* = (S^{n-2}A)^* \otimes T_X^{n-1}(-b).
\ed
A similar argument computes ${\mathbf R}^1(p_X)_*\theo(X_\infty-\Sigma)$.  This establishes \eqref{direct images}.

To get the Ext vanishing, note that $S^\ell(A)$ has a filtration with subquotients of the form $T_X^m$ for $0\leq m\leq \ell$.
Now
\bd
H^1(T_X^m\otimes K_X^{n-1}(b)) \cong H^0(K_X^{m-n+2}(-b))^* = 0
\ed
for $m\leq n-2$.  An inductive argument then shows that
\bd
H^1(S^{n-2}A\otimes K_X^{n-1}(b)) = 0 = H^1(S^{n-3}A\otimes K_X^{n-1}(b)),
\ed
which, by \eqref{direct images}, yields \eqref{Exts}.
\qed\end{proof}
Returning to the proof of Proposition \ref{unit section is trivial},
we now prove the claim for $m=0$.
We will push forward the diagram
\begin{equation}\label{big diagram}
\xymatrix{ & 0\ar[d] & 0 \ar[d] & 0 \ar[d] & \\
0\ar[r] & \theo(-\Sigma)\ar[d]\ar[r] & \theo_S \ar[d]\ar[r] & \theo_\Sigma \ar[d]\ar[r] & 0\\
0\ar[r] & \theo(X_\infty-\Sigma)\ar[d]\ar[r] & \theo(X_\infty) \ar[d]\ar[r] & \theo_\Sigma(X_\infty) \ar[d]\ar[r] & 0\\
0\ar[r] & \theo_{X_\infty}(X_\infty-\Sigma)\ar[d]\ar[r] & \theo_{X_\infty}(X_\infty) \ar[d]\ar[r] & \theo_\Sigma(X_\infty)/\theo_\Sigma \ar[d]\ar[r] & 0\\
 & 0 & 0 & 0 &}
 \end{equation}
We get a diagram with exact rows and columns:
\begin{equation}\label{big diagram 2}
\xymatrix{
 & 0\ar[d] & 0\ar[d] & T_X(-b)\ar@{^{(}->}[d] & \\
 0 \ar[r] & \theo_X \ar[d]\ar[r] & \cE_0\ar[r]^{\hspace{-2em}\pi}\ar[d] & {\mathbf R}^1(p_X)_*\theo(-\Sigma)\ar[d]\ar[r] & 0\\
 0\ar[r] & A\ar[d] \ar[r] & \cE_1\ar[d]\ar[r] & {\mathbf R}^1(p_X)_*\theo(X_\infty-\Sigma)\ar[d]\ar[r] & 0\\
 & T_X \ar[d]\ar[r] & \theo_b \ar[d] & 0 & \\
 & 0 & 0 & & }
 \end{equation}
 where the identification of $T_X(-b)$ as the kernel of the map
 ${\mathbf R}^1(p_X)_*\theo(-\Sigma)\rightarrow {\mathbf R}^1(p_X)_*\theo(X_\infty-\Sigma)$ is straightforward using
 \eqref{direct images}.  By \eqref{Exts}, the map $\pi$ is split surjective, so in particular
 \begin{equation}\label{first part}
 \cE_0 \cong \theo_X \oplus (S^{n-2}A)^*\otimes T_X^{n-1}(-b).
 \end{equation}

 Let $A' = \pi\inv(T_X(-b))\subset \cE_0$.  Replacing $\theo_X$ by $A'$ and adjusting \eqref{big diagram 2} as necessary, we get a diagram with exact rows and columns:
 \bd
 \xymatrix{
  & 0 \ar[d] & 0 \ar[d] &  & \\
  0\ar[r] & A'\ar[d]\ar[r] & \cE_0\ar[d]\ar[r] & {\mathbf R}^1(p_X)_*\theo(-\Sigma)/T_X(-b) \ar[r]\ar@{=}[d] & 0\\
  0\ar[r] & A \ar[d]\ar[r] & \cE_1 \ar[d]\ar[r] & {\mathbf R}^1(p_X)_*\theo(X_\infty-\Sigma) \ar[r] & 0\\
  & \theo_b \ar@{=}[r] \ar[d] & \theo_b \ar[d] &   & \\
   & 0 & 0 & & }
 \ed
It follows that $\cE_1$ is the pushout along the inclusion $A'\xrightarrow{f} A$ of $\cE_0$.  So, to complete the proof, we need only
show that this map is unique up to an automorphism $\phi$ of $\cE_0$.

By construction, we may split $A'$ as $A'=\theo_X\oplus T_X(-b)$ in such a way that the composite maps 
$\theo_X\rightarrow A'\xrightarrow{f} A$ and $T_X(-b) \rightarrow A' \xrightarrow{f} A\rightarrow A/\theo = T_X$ are the canonical ones.
  Choosing such a splitting, we may
write $f = \on{can}+h$ where $\on{can}$ denotes the canonical map from $\theo\oplus T_X(-b)$ to $A$ and
$h\in \Hom(T_X(-b), \theo)$.  Now, apply $\Hom(-, \theo)$ to the exact sequence
\bd
0\rightarrow T_X(-b)\rightarrow {\mathbf R}^1(p_X)_*\theo(-\Sigma)\rightarrow   {\mathbf R}^1(p_X)_*\theo(X_\infty-\Sigma)
\rightarrow 0.
\ed
We get an exact sequence
\bd
\Hom({\mathbf R}^1(p_X)_*\theo(-\Sigma),\theo)\rightarrow \Hom(T_X(-b), \theo)\rightarrow
\Ext^1({\mathbf R}^1(p_X)_*\theo(X_\infty-\Sigma),\theo).
\ed
By \eqref{Exts}, the right-hand term is zero, so $h$ may be lifted to a map
$\wt{h}\in  \Hom({\mathbf R}^1(p_X)_*\theo(-\Sigma),\theo)$.  Define
$\phi = 1_{\cE_0} - \wt{h}: \cE_0\rightarrow \cE_0$.  We find that
\bd
(f\circ \phi)|_{A'} = (\on{can}+h)(1- \wt{h})|_{A'}
= (\on{can}+h)(1-h) = \on{can} + h - \on{can}\circ h = \on{can}.
\ed
Combined with \eqref{first part}, this proves that up to isomorphism the triple $(\cE_0,\cE_1, \cE_0\hookrightarrow \cE_1)$
 does not depend on $\Sigma$.
\qed\end{proof}
Any choice of a trivialization of the determinant bundle $\Det$ at the corresponding point of $\PB_n(X)$ determines a trivialization of the pullback
$(u_m|_{\Ht\nt})^*\Det$, hence of any determinant-twisted cotangent bundle when pulled back to $\Ht\nt$.  We will call such a
trivialization a {\em canonical trivialization} of a determinant twisting over $\Ht\nt$.
\begin{corollary}\label{pullback becomes zero}
Let $\theta_{\PB}$ denote the canonical section of $T^*_{T^*_\PB(\alpha)}(\alpha)$ over $T^*_\PB(\alpha)$.  Then, under a canonical trivialization of
a determinant twisting, $(u_m|_{\Ht\nt})^*\theta_{\PB}$ is identified with the zero section.
\end{corollary}
\begin{proof}
Under any choice of local trivialization of a twisting, the canonical section is identified with the usual canonical 1-form.  The corollary
then follows from the definition of the canonical 1-form on a cotangent bundle.
\qed\end{proof}

\section{Twisted Local Systems}\label{twisted local systems}
In this section, we introduce the Fourier-dual geometry to twisted $\D_{\PB}$-modules, the moduli of twisted mirabolic local systems.  We also give
a description of the geometry of this stack in terms of the geometry of a twisted cotangent bundle of $\PB_n(X)$.
\subsection{Rees Modules and Local Systems}
Fix a weight $\lambda\in k$.  As in Section \ref{compactified TCBs and orders}, write
$\D^\ell(\lambda)$ for $\D^\ell_X(\theo(b)^{\otimes\lambda})$,
$\D_X(\lambda) = \D_X(\theo(b)^{\otimes\lambda})$, and
$\cR$ for
the Rees algebra of $\D(\lambda)$.

\begin{defn}
  A {\em $\lambda$-twisted mirabolic (de Rham) local system} $(\cE_\bullet, \nabla) = (\cE\subset \cE_1, \nabla)$ on $X$ consists of a mirabolic bundle
$\cE\subset \cE_1$ on $X$ together with a $\lambda$-twisted meromorphic connection
$\nabla: \D^1_X(\lambda)\otimes \cE \rightarrow \cE_1$ whose residue is compatible with the mirabolic structure.
\end{defn}
There is also a description analogous to Lemma \ref{mero and twisted Higgs} and Proposition \ref{twisted Higgs and reduction} identifying 
such twisted connections with ordinary connections with simple pole and prescribed residue at $b$.
We let $\PL_n^\lambda(X)$ denote the moduli stack of $\lambda$-twisted mirabolic local systems of rank $n$ on $X$.

We will use three (noncommutative) projective surfaces in our study of twisted mirabolic local systems.  Two of these we have already
encountered in Section \ref{compactified TCBs and orders}: one of them is the category $\on{Qgr}\,\cR$ associated to the Rees algebra $\cR$; we think of this as a compactification
of the noncommutative twisted cotangent bundle ``$\on{Spec}\, \D_X(\lambda)$.'' Another is the compactification
$\overline{T^*_{X^{(1)}}(c)}$ of the ordinary twisted cotangent bundle of $X^{(1)}$.
This is the projective bundle associated to the graded ring $\cR(\cZ(\D_X(\lambda)))$, the Rees algebra of the center of
$\D_X(\lambda)$; see Section \ref{compactified TCBs and orders} for a general discussion.

The third surface is intermediate between these two: it is the projective bundle associated to the Rees algebra
$\cR(\mathsf{Z})$ of $\mathsf{Z}= \mathsf{Z}_{\theo_X}(\D_X(\lambda))$, the centralizer of
$\theo_X$ in $\D_X(\lambda)$.  This is a sheaf of commutative algebras, and the associated projective surface is exactly
$X\times_{X^{(1)}} \overline{T^*_{X^{(1)}}(c)}$.  Note that, as for the center of $\cR$, the Rees algebra of the
centralizer $\cR(\mathsf{Z})$  is naturally $p$-graded,
and it is the Veronese algebra $\cR^{(p)}$ that is naturally a finite and flat graded module over $\cR(\mathsf{Z})$.

We get a commutative diagram of direct image functors between the
module categories:
\begin{equation}\label{direct images diagram}
\xymatrix{
\on{Qgr}\,\cR^{(p)} \ar[r]\ar[d] & \on{Qcoh}(X\times_{X^{(1)}}\overline{T^*_{X^{(1)}}(c)}) \ar[d]\ar[r] & \on{Qcoh}(\overline{T^*_{X^{(1)}}(c)})\ar[d]\\
\on{Qcoh}(X)\ar[r] & \on{Qcoh}(X) \ar[r] & \on{Qcoh}(X^{(1)}).
}
\end{equation}

We also establish some notation for the ``sections at infinity'' of the projective bundles
\begin{equation}\label{projective bundles}
X\times_{X^{(1)}} \overline{T^*_{X^{(1)}}(c)}\rightarrow X\;\;  \text{and}\;\;
\overline{T^*_{X^{(1)}}(c)}\rightarrow X^{(1)}.
\end{equation}
\begin{notation}\label{sections notation}
We let $\wt{X}_\infty\cong X$, $X_\infty^{(1)}\cong X^{(1)}$ respectively denote the canonical sections at infinity of the
projective bundles \eqref{projective bundles}.
\end{notation}

The stack $\PL_n^\lambda(X)$ can be described in terms of objects of $\qgr\,\cR$.
\begin{prop}\label{loc closed substack}
The stack $\PL_n^\lambda(X)$ of $\lambda$-twisted mirabolic local systems can be identified with a locally
closed substack of the moduli stack of objects of $\qgr\,\cR$.
\end{prop}
\begin{proof}
This is a special case of Theorem 4.2 of \cite{perverse}.  We will only briefly sketch the procedure by which the identification is made.

Given a twisted mirabolic local system $(\cE_\bullet,\nabla)$, the procedure of \cite[Theorem~4.2]{perverse} defines an object of
$\qgr\,\cR$:
\begin{equation}\label{def in qgr}
\cF = \cF(\cE_\bullet, \nabla) = \on{coker}\left[\pi\cR(-1)\otimes T_X\otimes \cE\xrightarrow{d} \pi\cR\otimes \cE_1\right],
\end{equation}
where $d$ is the map defined in \cite[Section~3.1]{perverse}.  It is clear from this definition and \cite[Section~4.2]{perverse}
that taking the Koszul data (terminology as in \cite{perverse}) corresponding to $\cF$, we obtain
\bd
p_*\cF(-1) = \cE, \,\,\,\,\, p_*\cF = \cE_1,
\ed
and the ``action map''
\bd
\nabla: \D^1(\lambda)\otimes \cE = \cR_1\otimes p_*\cF(-1) \rightarrow p_*\cF = \cE_1.
\ed
Here $\cF(-1)$ denotes the endofunctor of $\qgr\,\cR$ that shifts the grading of a graded module by $-1$.  It follows from the calculations
in \cite{perverse} that the action map is exactly the meromorphic connection we started with.
\qed\end{proof}

\subsection{Generic Local Systems}\label{generic local systems}
We now explain an open subset of the moduli stack of local systems which parametrizes generic local systems.

As in Section \ref{compactified TCBs and orders}, let $\cR^{(p)}$ denote the $p$-Veronese subring of $\cR$.
The ring $\cR^{(p)}$ sheafifies over $\overline{T^*_{X^{(1)}}(c)}$, and gives an order $R$ on $\overline{T^*_{X^{(1)}}(c)}$
whose restriction to the twisted cotangent bundle $T^*_{X^{(1)}}(c)$ is an Azumaya algebra, which is exactly the lift of
$\D_X(\lambda)$ to the twisted cotangent bundle described in Section \ref{azumaya property subsection}; see Section
\ref{compactified TCBs and orders} for a general discussion of the sheafification over the compactification $\overline{T^*_{X^{(1)}}(c)}$.

We have equivalences of
categories:
\begin{equation}\label{equivs on curves}
\on{Qgr}\,\cR \simeq \on{Qgr}\, \cR^{(p)} \simeq R\on{-mod}.
\end{equation}
\begin{defn}
Given an object $\cF$ of $\on{Qgr}\,\cR$, we will let $\cF^{(p)}$ denote the corresponding object of $R\on{-mod}$.  Given a twisted
mirabolic local system $(\cE_\bullet,\nabla)$, we call the object
$\cF= \cF(\cE_\bullet,\nabla)$ of $\qgr\,\cR$ defined above the {\em spectral sheaf} corresponding to $(\cE_\bullet,\nabla)$.
\end{defn}
 It is easy
to see that, as an $\theo_{\overline{T^*_{X^{(1)}}(c)}}$-module, the sheaf $\cF^{(p)}$ associated to a twisted mirabolic
local system is of pure dimension $1$.
We call the support $\Sigma$ of $\cF^{(p)}$ in $\overline{T^*_{X^{(1)}}(c)}$ the {\em spectral curve} of $\cF^{(p)}$ or of $(\cE_\bullet, \nabla)$.

By \eqref{def in qgr}, we have an exact sequence
\bd
0\rightarrow \pi\cR(-1)\otimes T_X\otimes \cE\rightarrow \pi\cR\otimes \cE_1\rightarrow \cF\rightarrow 0
\ed
in $\qgr\,\cR$ associated to a twisted mirabolic local system $(\cE_\bullet, \nabla)$.  A calculation using this sequence then shows that
the ``restriction to the curve $X_\infty$ at infinity" functor (see \cite{perverse} for details) takes $\cF$ to a sheaf
$\cF|_{X_\infty} \cong \cE_1/\cE$ supported at $x\in X_\infty  = X$.
In particular, the spectral curve $\Sigma$ of $\cF$ intersects $X_\infty^{(1)}$ only at the point $b\in X^{(1)}$.
\begin{defn}
 We will call $(\cE_\bullet,\nabla)$ a
{\em generic} ($\lambda$-twisted mirabolic) local system if the corresponding spectral curve $\Sigma\subset \overline{T^*_{X^{(1)}}(c)}$ is a generic spectral curve for the Hitchin system (Section \ref{Hitchin systems section}).
  We will let $\PL_n^\lambda(X)\gen$ denote the moduli stack of generic $\lambda$-twisted mirabolic local systems on $X$.
\end{defn}
\begin{remark}
The stack $\PL_n^\lambda(X)\gen$ is {\em not} the same as the stack
$\PLoc_n^\lambda(X, \overline{\ell})\gen$ that appears in the equivalence of Theorem \ref{thm 1}.  This will be explained in Section \ref{torsor section} below.
\end{remark}
Taking a generic local system to its spectral curve defines a map
\bd
\PL_n^\lambda(X)\gen\rightarrow (\Ht\gen)^{(1)},
\ed
where $\Ht^{(1)}$ is the twisted Hitchin base of Section \ref{Hitchin systems section}.

Let $\Sigma\subset \overline{T^*_{X^{(1)}}(c)}$ be a generic spectral curve.  A vector bundle $\cF^{(p)}$ on $\Sigma$ equipped with an
$R$-module structure is of {\em minimal rank} if it has rank $p$ as an $\theo_\Sigma$-module---note that, since $R$
is an order of rank $p^2$, this is the smallest possible rank of a nonzero $R$-module of pure dimension $1$.

\begin{lemma}\label{parabolic is minimal rk}
Under the equivalences of Proposition \ref{loc closed substack} and \eqref{equivs on curves}, generic twisted mirabolic
local systems correspond to
$R$-modules of minimal rank on generic spectral curves.
\end{lemma}

\subsection{Comparison of Twists}\label{twisting subsubsection}
In light of Lemma \ref{parabolic is minimal rk}, it is reasonable to ask how to describe the ``shift by $1$" functor
on $\on{qgr}\,\cR$ in geometric terms in the equivalent category $R\on{-mod}$.  We will turn to this question next.

Suppose that $\Sigma\subset \overline{T^*_{X^{(1)}}(c)}$ is a smooth curve---or, more generally, a family of smooth curves
over a base $T$---that is finite over $X^{(1)}$ of degree $n$ and has simple intersection $\Sigma\cap X_\infty^{(1)} = \{b\}$; we use
the same notation in the case of a family over $T$.  Suppose that the projection $\Sigma\rightarrow X^{(1)}$ is \'etale in a neighborhood
of $b$: in other words, $\Sigma$ is a generic spectral curve.  Let $\Sigma_{et}\subset \Sigma$ denote the open subset consisting of points of $\Sigma$ at which the projection
$\Sigma\rightarrow X^{(1)}$ is \'etale.  Let $\wt{\Sigma}_{et} = X\times_{X^{(1)}}\Sigma_{et}$; the projection
$\wt{\Sigma}_{et}\rightarrow X$ is \'etale of
degree $n$.

Alternatively, we may write $\wt{\Sigma}_{et}$ as
$(X\times_{X^{(1)}}\overline{T^*_{X^{(1)}}(c)})\times_{\overline{T^*_{X^{(1)}}(c)}} \Sigma_{et}$.  The scheme-theoretic
intersection $\wt{X}_\infty\cap \wt{\Sigma}_{et}$ is then given by:
\begin{multline*}
\wt{X}_\infty \cap \wt{\Sigma}_{et} =
(X\times_{X^{(1)}} X_\infty^{(1)})\times_{X\times_{X^{(1)}} \overline{T^*_{X^{(1)}}(c)}} (X\times_{X^{(1)}}
\overline{T^*_{X^{(1)}}(c)}) \times_{\overline{T^*_{X^{(1)}}(c)}} \Sigma_{et} \\
 = X\times_{X^{(1)}}(X_\infty^{(1)}\times_{\overline{T^*_{X^{(1)}}(c)}}\Sigma_{et}) = X\times_{X^{(1)}} (X_\infty^{(1)}\cap \Sigma_{et}).
 \end{multline*}
Since the map $X\rightarrow X^{(1)}$ is totally ramified of degree $p$, $\wt{\Sigma}_{et}\cap X_\infty^{(1)}$ is nonreduced: in fact, it is
of
degree $p$ over the corresponding point $x = F^{-1}(b)$ in $X$.

On the other hand, because, by hypothesis, $\wt{\Sigma}_{et}\rightarrow X$ is \'etale, one gets a quasifinite map
$\{x\}\times_X \wt{\Sigma}_{et}\rightarrow \{x\}$ (here, again, if $\Sigma$ lives over a base scheme $T$ then $\{x\}$ really means
$T\times \{x\}$, etc.).  The scheme $\{x\}\times_X \wt{\Sigma}_{et}$ has a component, which we will denote by $\wt{b}$, that maps
isomorphically to $x$ and lies over $b\in\Sigma_{et}$: this is the ``point in the fiber of $\wt{\Sigma}_{et}\rightarrow X$ over $x$
that lies on $\wt{X}_\infty$.''  In other words, $\wt{b}$ provides a kind of $p$th root of the scheme-theoretic intersection
$\wt{X}_\infty\cap \wt{\Sigma}_{et}$.

By the above discussion, $\wt{b}$ is an effective Cartier divisor on $X\times_{X^{(1)}}\Sigma$ (in fact, supported on $\wt{\Sigma}_{et}$),
so it makes sense to twist a quasicoherent sheaf $\cF$ on $X\times_{X^{(1)}}\Sigma$ by $\theo(\ell\wt{b})$ ($\ell\in{\mathbb Z}$).

Suppose $\cF^{(p)}$ is a vector bundle on a generic spectral curve $\Sigma\subset\overline{T^*_{X^{(1)}}(c)}$.
As is implicit in the discussion above, a choice of $R$-module structure on $\cF^{(p)}$ determines a lift of $\cF^{(p)}$ to a sheaf on
$X\times_{X^{(1)}} \Sigma\subset X\times_{X^1}\overline{T^*_{X^{(1)}}(c)}$.  One way to see this is to use the equivalence
\eqref{equivs on curves} to lift $\cF^{(p)}$ to an object $\cF$ of $\on{qgr}\,\cR$ and then push forward to
$X\times_{X^1}\overline{T^*_{X^{(1)}}(c)}$.  A more direct way to explain this is just to observe that an $R$-module is just the
sheafification of a $\cR^{(p)}$-module, and thus in particular has an action of the Rees algebra of the centralizer $\cR(\mathsf{Z})$, i.e.
a lift to $X\times_{X^{(1)}}\overline{T^*_{X^{(1)}}(c)} = \bproj(\cR(\mathsf{Z}))$.  It then follows from our earlier discussion that it makes sense to twist $\cF^{(p)}$ by the line bundle $\theo(\wt{b})$ on $X\times_{X^{(1)}} \Sigma$.
\begin{prop}\label{shift is twist by btilde}
Suppose that $\cF^{(p)}$ is a vector bundle on a generic spectral curve $\Sigma$, equipped with a structure of an $R$-module of
minimal rank.  Let $\cF$ denote the corresponding object of $\on{qgr}\,\cR$.  Then $\cF(1)$ is identified with $\cF^{(p)}(\wt{b})$,
compatibly with the natural maps from $\cF$ and $\cF^{(p)}$, respectively.
\end{prop}
\begin{proof}
Consider $\cF$ as a sheaf on $X\times_{X^{(1)}} \Sigma$ (via the left-hand arrow in the top row of \eqref{direct images diagram}).  As in the proof of Lemma \ref{parabolic is minimal rk}, $\cF(1)/\cF$ has length $1$ and is supported at $\wt{b}$.  Note, however, that
$\cF|_{\wt{\Sigma}_{et}}$ is a line bundle---indeed, it is $\theo_{\wt{\Sigma}_{et}}$-coherent and torsion-free, and its direct image
to $\Sigma_{et}$ has minimal rank, i.e. rank $p$, so $\cF|_{\wt{\Sigma}_{et}}$ has rank $1$.  The same argument shows that
$\cF(1)|_{\wt{\Sigma}_{et}}$ is a line bundle.  Consequently, since the cokernel of the inclusion
$\cF\rightarrow \cF(1)$ is supported at the point $\wt{b}\in \wt{\Sigma}_{et}$, the conclusion follows.
\qed\end{proof}

We will also need the following property of line bundles on $\wt{\Sigma}_{et}$ in the sequel.
\begin{lemma}\label{line bdl ext}
Let $\Sigma= \Sigma_T$ be a family of generic spectral curves over a smooth $k$-scheme $T$.  Let $\wt{b}$ be the effective Cartier divisor
on $\wt{\Sigma}_{et}$ defined above.  Let $L$ be a line bundle on $\wt{\Sigma}_{et}\setminus \{\wt{b}\}$.  Then:
\begin{enumerate}
\item There exists a line bundle $\overline{L}$ on $\wt{\Sigma}_{et}$ equipped with an isomorphism
\bd
\overline{L}|_{\wt{\Sigma}_{et}\setminus\{\wt{b}\}} \cong L.
\ed
\item For any two such extensions $\overline{L}_1$ and $\overline{L}_2$, there exists an $\ell\in{\mathbb Z}$ for which
$\overline{L}_1(\ell\wt{b}) = \overline{L}_2$.
\end{enumerate}
\end{lemma}

\subsection{An Extension Property}
In this section, we prove an extension property for $R$-modules on generic spectral curves, which says that a flat family of
 $R$-modules of minimal
rank on a family
$\Sigma\setminus\{b\}$ of generic spectral curves can be extended to a flat family of $R$-modules on the full curve $\Sigma$.  We also
explain how these extensions are related.

\begin{prop}\label{extension lemma}
Let $T\rightarrow(\Ht^\circ)^{(1)}$ be a smooth scheme over $(\Ht^\circ)^{(1)}$.  Write $\Sigma = \Sigma_T\subset \overline{T^*_{X^{(1)}}(c)}$ for the corresponding family of generic spectral curves.  Let $\cF^{(p)}$ be a vector bundle on $\Sigma_T\setminus\{b\}$ equipped with a
structure of finitely generated  $R$-module that is of minimal rank.  Then:
\begin{enumerate}
\item $\cF^{(p)}$ extends to a vector bundle over $\Sigma_T$ equipped with a structure of $R$-module of minimal rank.
\item If $\overline{\cF}^{(p)}_1$ and $\overline{\cF}^{(p)}_2$ are two such extensions, then, after a shift
$\big(\overline{\cF}_1(s)\big)^{(p)}$ of the corresponding object in $\qgr\,\cR$, there is a unique isomorphism
\bd
\overline{\cF}_1(s)^{(p)}\cong \overline{\cF}_2^{(p)}
\ed
compatible with the inclusions into $\cF^{(p)}$.
\end{enumerate}
\end{prop}
\begin{proof}
Write $\Sigma = \Sigma_T$.
As in the previous section, let $\Sigma_{et}$ denote the open subset of $\Sigma$ consisting of points near which the map to
$T\times X^{(1)}$ is \'etale, and let $\wt{\Sigma}_{et} = X\times_{X^{(1)}} \Sigma_{et}$.  This is a family of curves \'etale and
quasifinite over $T\times X$ of generic degree $n$; the map to $\Sigma_{et}$ is finite, flat,  and totally ramified.

Let  $\wt{b}\subset \wt{\Sigma}_{et}$ denote the Cartier divisor defined in the previous section.  Then the image of
$\wt{\Sigma}_{et} \setminus\{\wt{b}\}$ in $\Sigma_{et}$ is $\Sigma_{et}\setminus\{b\}$.  Consequently, the $R$-module structure
on $\cF^{(p)}$ determines, as in Section \ref{twisting subsubsection},  a lift $\wt{\cF}$
of $\cF^{(p)}$ to $\wt{\Sigma}_{et}\setminus\{\wt{b}\}$.
Since $\cF^{(p)}$ is a vector bundle of rank $p$ on $\Sigma\setminus\{b\}$, the lift $\wt{\cF}$ is a line bundle on the degree $p$ cover
$\wt{\Sigma}_{et}\setminus\{\wt{b}\}$.  By Lemma \ref{line bdl ext}, $\wt{\cF}$ extends to a line bundle on
$\wt{\Sigma}_{et}$, and moreover any two such extensions differ by a twist by $\theo(\ell\wt{b})$ for some $\ell$.  Choose one such
extension, which we will call $\overline{\cF}^{(p)}$.

Combining Proposition \ref{shift is twist by btilde} and
Lemma \ref{line bdl ext}, to complete the proof of existence and uniqueness,
it suffices to prove that the subsheaf $\overline{\cF}^{(p)}\subset \cF^{(p)}$ is preserved by the
$R$-action on $\cF^{(p)}$: in other words, to prove that the map
$R\otimes \overline{\cF}^{(p)}\rightarrow \cF^{(p)}/\overline{\cF}^{(p)}$ is zero.  Since the left-hand side is finitely
generated, its image will land in $\overline{\cF}^{(p)}(\ell\wt{b})/\overline{\cF}^{(p)} \subset \cF^{(p)}/\overline{\cF}^{(p)}$
for some sufficiently large $\ell$, and so it suffices to prove that the multiplication map
\begin{equation}\label{mult}
A\otimes \overline{\cF}^{(p)}\rightarrow \overline{\cF}^{(p)}(\ell\wt{b})/\overline{\cF}^{(p)}
\end{equation}
 is zero.  Since this module is finite and flat over $S$ and $S$ is reduced, it suffices to check that \eqref{mult} is zero
 fiberwise, i.e.  on the restriction to the fiber over
 every $s\in S$.

 Thus, we may assume that $T = \on{Spec}(K)$.  We will prove that $\eqref{mult}$ is the zero map, i.e. that $\overline{\cF}^{(p)}$ is
 an $R$-submodule of $\cF^{(p)}$.
 A standard argument shows that there is an $R$-submodule $\cF'$ with  $\overline{\cF}^{(p)}\subseteq\overline{\cF}'\subset \cF^{(p)}$:
 since $R$ is a finite algebra, we may take the image of $R\otimes\overline{\cF}^{(p)}$ in $\cF^{(p)}$.\footnote{Indeed, we could use
 this method to prove (1) pointwise on $T$---we do some work to prove the proposition mainly because it is not {\em a priori} obvious
 that using this method over a general base $T$ would result in a flat family of modules.}  By construction, the restriction of the
 inclusion $\overline{\cF}^{(p)}\subseteq \overline{\cF}'$ to $\wt{\Sigma}_{et}$ is an inclusion of line bundles on $\wt{\Sigma}_{et}$ that
 is an isomorphism over $\wt{\Sigma}_{et}\setminus\{\wt{b}\}$.  Thus, by Lemma \ref{line bdl ext}(2), there is an
 $\ell\in{\mathbb Z}$ for which $\overline{\cF}^{(p)} = \overline{\cF}'(-\ell\wt{b})$.   But, by Proposition
 \ref{shift is twist by btilde}, $\overline{\cF}'(-\ell\wt{b})$ is also an $R$-submodule of $\cF^{(p)}$.  This completes the proof.
 \qed\end{proof}

\subsection{Torsor Structure on Moduli of Local Systems}\label{torsor section}
Let $(\cE_\bullet,\nabla)$ be a generic local system and  $\cF$ the corresponding spectral sheaf with spectral curve
$\Sigma\subset\overline{T^*_{X^{(1)}}(c)}$.

For each line bundle $L$ on $\Sigma$, we
may form a new generic local system associated to the spectral sheaf $L\otimes \cF^{(p)}$---note that
$L\otimes\cF^{(p)}$ is naturally an $R(\lambda)$-module
since $\theo_{\overline{T^*_{X^{(1)}}(c)}}$ is a central subalgebra in $R(\lambda)$.  This procedure defines an action of
$\Pic(\Sigma^{(1)}/(\Ht\gen)^{(1)})$ on $\PL_n^\lambda(X)\gen$ over $(\Ht\gen)^{(1)}$  ``by twist.''  In addition, we may define an action of
${\mathbb Z}$ on $\PL_n^\lambda(X)\gen$ over $(\Ht\gen)^{(1)}$ by shifting the grading of  the $\cR$-module $\cF$.
As in Remark \ref{shift remark}, it is easy to check that, under the equivalence \eqref{R vs R}, the action of shifting  by
$p$ is identified with the action of $\theo_{\overline{T^*_{X^{(1)}}(c)}}(X_\infty^{(1)})$ (restricted to the spectral curve) by twist on spectral
sheaves; it follows that the quotient group
\bd
\sP:= \left(\Pic(\Sigma^{(1)}/(\Ht\gen)^{(1)})\times{\mathbb Z}\right)/p{\mathbb Z}
\ed
acts on $\PL_n^\lambda(X)\gen$ over $(\Ht\gen)^{(1)}$.
\begin{prop}\label{torsor structure}
The spectral curve map $\PL_n^\lambda(X)\gen\rightarrow (\Ht\gen)^{(1)}$ has a natural structure of
$\sP$-torsor over $(\Ht\gen)^{(1)}$.
\end{prop}
\begin{proof}
It is straightforward to see that $\sP$ acts freely; hence it suffices to check that $\sP$ acts transitively.

Let $(\cE_\bullet,\nabla)$ and $(\cE'_\bullet, \nabla')$ be generic local systems with spectral curve $\Sigma$, and $\cF, \cG$ the
corresponding spectral sheaves. Define
\bd
H : = \Hom_{R(\lambda)}(\cF^{(p)},\cG^{(p)}) \subset \Hom_{\theo_\Sigma}(\cF^{(p)},\cG^{(p)}).
\ed
Since $\Hom_{\theo_\Sigma}(\cF^{(p)},\cG^{(p)})$ is a vector bundle on $\Sigma$, $H$ is torsion-free on $\Sigma$.  Moreover, $R(\lambda)$ is
generically Azumaya and $\cF, \cG$ are modules of minimal rank, so $H$ has rank $1$, i.e. is a line bundle on $\Sigma$.  Replacing
$\cF^{(p)}$ by $H\otimes\cF^{(p)}$, we get an injective map $\cF^{(p)}\rightarrow \cG^{(p)}$ of $R(\lambda)$-modules that is an isomorphism over
$\Sigma\smallsetminus\{b\}$.

The conclusion now follows from Proposition \ref{extension lemma}(2).
\qed\end{proof}

The stack $\PL_n^\lambda(X)$ has components labelled by integers $\ell$ which describe the degree of a vector bundle $\cE$ underlying
a generic mirabolic local system.
\begin{defn}
For a choice $\overline{\ell}\in {\mathbb Z}/p{\mathbb Z}$, we let $\PLoc_n^\lambda(X, \overline{\ell})\gen$ denote the union of
components of $\PL_n^\lambda(X)\gen$ labelled by integers congruent to $\ell$ mod $p$.
\end{defn}
The following is then an immediate consequence of Proposition \ref{torsor structure}.
\begin{corollary}\label{ploc as torsor}
For each $\overline{\ell}$, $\PLoc_n^\lambda(X, \overline{\ell})\gen$ is a
$\Pic(\Sigma^{(1)}/(\Ht\gen)^{(1)})$-torsor over $(\Ht\gen)^{(1)}$.
\end{corollary}

\section{Hecke Correspondences and Tensor Structure}\label{hecke section}
This section introduces mirabolic Hecke correspondences relating the different components of $\PB_n(X)$.  The geometry of these correspondences is
used in an essential way to describe a tensor structure on the Azumaya algebra $\D_{\PB}(\lambda)$, and hence a group structure on its gerbe of
splittings, that plays a central role in Theorem \ref{thm 1}.

\subsection{$\cH$ and Twistings}
Let $\cH = \cH_1$ denote the mirabolic {\em Hecke correspondence}, defined as follows.  The stack $\cH$ parametrizes triples
$(\cE_\bullet,\cF_\bullet, i)$ consisting of mirabolic bundles $\cE_\bullet, \cF_\bullet$ and an inclusion
$i:\cE_\bullet\hookrightarrow\cF_\bullet$ with the following property: the quotients $\cF_0/\cE_0$ and $\cF_1/\cE_1$ are torsion
sheaves of length $1$ on $X$, and the natural map $\cE_1/\cE_0\rightarrow \cF_1/\cF_0$ is an isomorphism.  One should think that
points of $\cH$ are ``modifications of $\cE$ that do not change the mirabolic structure.''
\begin{remark}
The modifications that {\em do} change the mirabolic structure also change the twisting line bundle $\Det$ in a way that is incompatible
with our methods below.
\end{remark}

The stack $\cH$ comes equipped with natural maps $\cH\xrightarrow{q_1} X\times \PB_n$ and $\cH\xrightarrow{q_2} \PB_n$ given
by $q_1(\cE_\bullet,\cF_\bullet,i) = (\on{supp}(\cF_0/\cE_0), \cE_\bullet)$ and $q_2(\cE_\bullet, \cF_\bullet, i) = \cF_\bullet$.  The
maps $q_1$ and $q_2$ are smooth surjective morphisms.
\begin{lemma}\label{twistings are the same}
Let $\Det$ denote the line bundle on $\PB_n(X)$ of Definition \ref{det bundle}.
  Then there is a canonical isomorphism
\begin{equation}\label{twistings}
q_2^*({\Det}) \cong q_1^*(\theo_X(b)\boxtimes {\Det}).
\end{equation}
\end{lemma}
\subsection{More General Hecke Correspondences and Twistings}\label{Hecke sub 2}
We will now define a kind of ``$r$-point generic Hecke correspondence."  Namely, let $\cH_r^\circ$ denote the moduli stack for
triples $(\cE_\bullet, \cF_\bullet, i)$ consisting of mirabolic bundles $\cE_\bullet, \cF_\bullet$ and an inclusion
$i:\cE_\bullet\hookrightarrow\cF_\bullet$ with the following property: the quotients $\cF_0/\cE_0$ and $\cF_1/\cE_1$ are torsion
sheaves of length $r$ on $X$ with support consisting of $r$ distinct points, and the natural map
$\cE_1/\cE_0\rightarrow \cF_1/\cF_0$ is an isomorphism.  The points of $\cH_r^\circ$ correspond to ``modifications of $\cE$ at
$r$ distinct points that do not change the mirabolic structure."

\begin{remark}
We restrict attention to ``generic" Hecke correspondences because these suffice for our applications below and,
by restricting, we can ignore scheme-theoretic issues below.
\end{remark}

As above, $\cH_r^\circ$ comes equipped with natural maps
\bd
\cH_r^\circ \xrightarrow{q_1} (S^rX\smallsetminus \Delta) \times\PB_n\;\; \text{and}\;\; \cH_r^\circ\xrightarrow{q_2} \PB_n.
\ed
Here $\Delta\subset S^rX$ denotes the ``big diagonal.''
  These maps are smooth and dominant.

  Finally, for use in Section \ref{final Hecke section}, we also define:
  \begin{defn}\label{Hecke stacks}
 Choose $r$ such that $1\leq r \leq n$.  Let $\on{Hecke}_r$ denote the moduli stack parametrizing quadruples
 $(\cE_\bullet, \cF_\bullet, i, x)$ consisting of mirabolic bundles $\cE_\bullet, \cF_\bullet$, an inclusion
 $i:\cE_\bullet\hookrightarrow\cF_\bullet$ of sheaves, and a point $x\in X$ with the property that the quotients
 $\cF_0/\cE_0$ and $\cF_1/\cE_1$ are torsion sheaves of length $r$ supported scheme-theoretically at $x$ and the natural
 map $\cE_1/\cE_0\rightarrow \cF_1/\cF_0$ is an isomorphism. Let
 \bd
 \PB_n \times X \xleftarrow{q_1} \on{Hecke_r}\xrightarrow{q_2} \PB_n
 \ed
 denote the projections taking $q_1(\cE_\bullet, \cF_\bullet, i, x) = (\cE_\bullet, x)$ and
 $q_2(\cE_\bullet, \cF_\bullet, i, x) = \cF_\bullet$.
  Note that $\on{Hecke_1} = \cH_1$.
   \end{defn}

  Let $\Det$ denote the line bundle on $\PB_n(X)$ as above.  Let $\theo_X(b)^{S_r}$ denote the line bundle on $S^rX$
   obtained as the invariant direct image $\pi^{S_r}_*\theo_X(b)^{\boxtimes r}$, where  $\pi: X^r\rightarrow S^rX$ denotes the projection.   We have the following analog of Lemma \ref{twistings are the same}:
\begin{lemma}\label{gen twistings are the same}
There is a canonical isomorphism on $\cH_r^\circ$:
\begin{equation}\label{gen twistings}
q_2^*(\Det)\cong q_1^*(\theo_X(b)^{S_r} \boxtimes \Det).
\end{equation}
\end{lemma}
\noindent
We will write $\Det$ for this line bundle on $\cH_r^\circ$.

Given $\alpha \in k\smallsetminus {\mathbb F}_p$, we will write $T^*_{\cH_r^\circ}(\alpha) := T^*_{\cH_r^\circ}(\Det^{\otimes \alpha})$.
Let $Z_r$ denote the moduli stack of quadruples $(\Sigma, L_1, L_2, i)$ where $\Sigma\subset \overline{T^*_X(\alpha)}$
is a generic spectral curve, $L_1$ and $L_2$ are line bundles on $\Sigma$, and $i: L_1\hookrightarrow L_2$ is an inclusion with
cokernel of length $r$ supported on $\Sigma\smallsetminus\{b\}$ and over $r$ distinct points of $X$.  We will write
\bd
T^*_{S^rX}(\alpha) : = T^*_{S^rX}\big(\theo(b)^{S_r})^{\otimes \alpha}\big).
\ed
We have ``forgetful" maps
\begin{equation}\label{forgetfuls}
\xymatrix{T^*_{S^rX}(\alpha)\times T^*_{\PB}(\alpha) & \ar[l]_{\hspace{4em}s_1}   Z_r \ar[r]^{\hspace{-1em}s_2} & T^*_\PB(\alpha)}
\end{equation}
that take $(\Sigma,L_1,L_2,i)$ to $(\on{supp}(L_2/L_1), L_1)$ and $L_2$ respectively.  Since $Z_r$ also maps forgetfully to
$\cH_r^\circ$, we thus obtain a commutative diagram (using \eqref{gen twistings} to get $j_1$ and $j_2$):
\begin{equation}\label{Hecke diagram}
\xymatrix{
Z_r \ar[d]^{i_1 = s_1\times p_{\cH_r}} \ar[r]^{i_2 = s_2\times p_{\cH_r}} & T^*_{\PB}(\alpha)\times_\PB \cH_r^\circ \ar[d]_{j_2}\\
\left(T^*_{S^rX}(\alpha)\times T^*_{\PB}(\alpha)\right)\times_{S^rX\times\PB}\cH_r^\circ \ar[r]^{\hspace{4em}j_1} &
T^*_{\cH_r^\circ}(\alpha).}
\end{equation}
These maps are all immersions.
\begin{remark}[Support of Hecke Operators]\label{hecke operators remark}
By Lemma \ref{support of Hecke}, the fiber product in Diagram \ref{Hecke diagram} exactly tells us the support of the Hecke operator
$(q_1)_*q_2^*$ determined by the diagram
\bd
(S^r X\smallsetminus \Delta)\times \PB \xleftarrow{q_1} \cH_r^\circ \xrightarrow{q_2} \PB.
\ed
So, describing $Z_r$ gives us a concrete way to compute the action of the ``generic $r$-point Hecke operator'' on $\D$-modules.
This is the essential point in the proof of Proposition \ref{equality of thetas} below.
\end{remark}

Recall that $\Sigma/\Ht$ denotes the universal generic spectral curve, and that $T^*_\PB(\alpha)^\circ$ is identified with the relative
Picard $\Pic(\Sigma/\Ht^\circ)$.  For each $r\geq 1$, let
\begin{equation}\label{AJ label}
AJ_r\times 1_{T^*}: S^r\Sigma\times_{\Ht} T^*_{\PB}(\alpha)^\circ\rightarrow  T^*_\PB(\alpha)^\circ\times_{\Ht}  T^*_\PB(\alpha)^\circ
\end{equation}
denote the Abel-Jacobi map on the first factor.
Let
\bd
m: T^*_{\PB}(\alpha)^\circ\times_{\Ht} T^*_{\PB}(\alpha)^\circ\rightarrow T^*_{\PB}(\alpha)^\circ
\ed
 denote the product on the relative
Picard.
\begin{lemma}\label{Z identification}
There is a natural identification
\bd
Z_r \hookrightarrow S^r\Sigma\times_{\Ht} T^*_\PB(\alpha)
\ed
of $Z_r$ with an open subset of $S^r\Sigma\times_{\Ht} T^*_\PB(\alpha)$ in such a way that the two forgetful maps \eqref{forgetfuls} are identified with the obvious
projection to $T^*_{S^rX}(\alpha)\times T^*_{\PB}(\alpha)$ and the composite $m\circ (AJ_r\times 1_{T^*})$, respectively.
\end{lemma}
The isomorphism of the lemma is immediate from the description of $Z_r$.

Let
\bd
AJ_\Sigma: \Sigma\smallsetminus\{b\}\rightarrow \on{Pic}(\Sigma) = T^*_\PB(\alpha)^\circ
\ed
denote the Abel-Jacobi map on the complement of $b$.
\begin{lemma}\label{AJ pullback of det}
Let $p_X: \Sigma\rightarrow X$ denote the projection.
We have $AJ_\Sigma^*\Det \cong \pi_X^*\theo_X(b)$.
\end{lemma}

\subsection{Character Property for the Canonical Section}\label{character property}
Let
\bd
m: T^*_{\PB}(\alpha)^\circ\times_{\Ht} T^*_{\PB}(\alpha)^\circ\rightarrow T^*_{\PB}(\alpha)^\circ
\ed
denote the product on the Frobenius twist.

Let $G\rightarrow\Ht$ be a commutative group stack over a scheme $\Ht$ with product $m$ and unit $\iota$.  A {\em character line bundle} on $G$ is a line bundle
$L$ equipped with an isomorphism $m^*L \cong L\boxtimes L$ and an isomorphism $\iota^*L\cong \theo_X$ satisfying ``standard''
associativity and unit commutativity diagrams:
cf. \cite{OV}, Definition~5.11 for discussion.

We begin with the character property for the twisting line bundle:
\begin{lemma}\label{det is equivariant}
By abuse of notation, let $\Det$ denote the pullback to $T^*_\PB(\alpha)$ of the line bundle $\Det$ on $\PB$.
Then $\Det$ is a character line bundle on $T^*_\PB(\alpha)$.
\end{lemma}

Recall that $\theta$ denotes the canonical section of the ``twisted cotangent bundle of the twisted cotangent bundle'' (see Section
\ref{canonical sections}).   The character property of the canonical section $\theta_\PB$ over $T^*_\PB(\alpha)$ is:
\begin{prop}\label{equality of thetas}
We have
\begin{equation}\label{compatibility of thetas}
m^*\theta_\PB = \theta_\PB\boxtimes\theta_\PB.
\end{equation}
\end{prop}
\begin{proof}
Let $a: Z_r\rightarrow T^*_\PB(\alpha)\times_{\Ht} T^*_\PB(\alpha)$ denote the restriction of the Abel-Jacobi map $AJ_r\times 1_T^*$ (see \eqref{AJ label})  to $Z_r$.  For
$r$ sufficiently large, $a$ surjects onto the degree $r$ component $T^*_\PB(\alpha)_r\times_{\Ht} T^*_{\PB}(\alpha)$.  We will first check
\eqref{compatibility of thetas} on $T^*_\PB(\alpha)_r\times_{\Ht} T^*_\PB(\alpha)$ for $r$ large---by Lemma
\ref{checking via pullback}, it suffices to check that
$a^*m^*\theta_\PB = a^* \theta\boxtimes\theta$.

Using notation as in \eqref{Hecke diagram} and abbreviating $T^* = T^*_\PB(\alpha)$, we have a commutative diagram
\bd
\xymatrix{Z_r \ar[r]^{a} \ar[d]^{i_2} & T^*\times_\Ht T^*\ar[d]_{m}\\
T^*\times_\PB \cH_r^\circ \ar[r]^{pr_1} & T^*.}
\ed
Thus $a^*m^*\theta_\PB = i_2^*pr_1^*\theta_\PB$.  Applying Proposition \ref{two canonical sections agree} to $j_2$, we get
\begin{equation}\label{multiplicative calc}
a^*m^*\theta_\PB = i_2^*pr_1^*\theta_\PB = i_2^* j_2^*\theta_{\cH_r^\circ} = i_1^* j_1^*\theta_{\cH_r^\circ}
=s_1^*\theta_{S^rX\times\PB} = \theta_{S^rX\times\PB}|_{S^r\Sigma\times_\Ht T^*},
\end{equation}
where the second-to-last equality also follows from Proposition \ref{two canonical sections agree}.

We now repeat the argument of the previous paragraph using the diagram
\bd
\xymatrix{Z_r\times_\Ht T^* \ar[r]^{a\times 1} \ar[d]^{i_2} & T^*\times_\Ht T^*\times_\Ht T^*\ar[d]_{m\times 1}\\
T^*\times_\PB \cH_r^\circ \times T^* \ar[r]^{pr_1\times 1} & T^* \times T^*.}
\ed
We get
\begin{multline}
(a\times 1)^*(m\times 1)^* (\theta_{\PB\times\PB}) =
(i_2\times 1)^*(pr_1\times 1)^*\theta_{\PB\times\PB})\\
= (i_2\times 1)^*(j_2\times 1)^* \theta_{\cH_r\nt\times\PB}
= (i_1\times 1)^*(j_1\times 1)^*\theta_{\cH_r\nt\times\PB}\\
= (s_1\times 1)^*\theta_{S^r\times \PB\times\PB}.
\end{multline}
We now pull back to $S^r\Sigma\times_\Ht T^*$ along the inclusion
\bd
S^r\Sigma\times_\Ht T^* = S^r\Sigma \times_\Ht u(\Ht\nt)\times_\Ht T^* \hookrightarrow Z_r\times_\Ht T^*.
\ed
We get
\begin{equation}\label{what is a*}
a^*\theta\boxtimes\theta = \theta_{S^rX\times\PB\times\PB}|_{S^r\Sigma\times_\Ht T^*} =  \theta_{S^rX\times\PB}|_{S^r\Sigma\times_\Ht T^*},
\end{equation}
where the second equality follows from Corollary \ref{pullback becomes zero}.
Combining \eqref{what is a*} and \eqref{multiplicative calc} now gives $a^*m^*\theta_\PB = a^*\theta\boxtimes\theta$.
This proves that $m^*\theta_\PB = \theta\boxtimes\theta$ on $T^*_\PB(\alpha)_r\times_{\Ht}T^*_\PB(\alpha)$ for all $r$ sufficiently large.

It remains to prove that, for an arbitrary $r,s$, the multiplication map
\bd
m_{r,s}: \on{Pic}^r(\Sigma)\times_\Ht\on{Pic}^s(\Sigma)\rightarrow \on{Pic}^{r+s}(\Sigma)
\ed
satisfies $m_{r,s}^*\theta = \theta\boxtimes\theta$.  Consider the diagram
\bd
\xymatrix{
\on{Pic}^N\times_\Ht \on{Pic}^r\times_\Ht\on{Pic}^s \ar[rr]^{\hspace{1em}1\times m_{r,s}}
\ar[d]^{m_{N,r}\times 1} & & \on{Pic}^N\times_\Ht\on{Pic}^{r+s} \ar[d]^{m_{N,r+s}} \\
\on{Pic}^{N+r}\times_\Ht \on{Pic}^s \ar[rr]^{m_{N+r,s}} & & \on{Pic}^{N+r+s}.
}
\ed
For $N$ sufficiently large, the conclusion of the previous paragraph gives
\begin{multline*}
\theta\boxtimes\theta\boxtimes\theta = (m_{N,r}\times 1)^*\theta\boxtimes\theta =
(m_{N,r}\times 1)^* m_{N+r,s}^*\theta\\
= (1\times m_{r,s})^*m_{N,r+s}^*\theta = (1\times m_{r,s})^*\theta\boxtimes\theta
= \theta\boxtimes (m_{r,s}^*\theta).
\end{multline*}
Now, we pull back along the map
\bd
u_N\times 1: \Ht^\circ\times_\Ht \on{Pic}^r\times_\Ht \on{Pic}^s
\rightarrow \on{Pic}^N\times_\Ht \on{Pic}^r\times_\Ht \on{Pic}^s,
\ed
where $u_N$ is the twisted unit section (Definition \ref{unit section}).
For $N$ sufficiently large and of an appropriate value (determined by Formula \eqref{degree of image}),
Corollary \ref{pullback becomes zero} applied to $\theta\boxtimes\theta\boxtimes\theta = \theta\boxtimes (m_{r,s}^*\theta)$
gives
\bd
\theta\boxtimes\theta = (u_N\times 1\times 1)^*(\theta\boxtimes\theta\boxtimes\theta)
=  (u_N\times 1\times 1)^*(\theta\boxtimes (m_{r,s}^*\theta)) = m_{r,s}^*\theta.
\ed
This completes the proof in general.
\qed\end{proof}

\subsection{Group Structure for the Gerbe Associated to the TDO on $\PB$}
We now fix $\lambda \in k\smallsetminus {\mathbb F}_p$, $c=\lambda^p - \lambda$, and $a = \lambda - \lambda^{1/p}$.  

Recall the definition of a {\em tensor structure} on an Azumaya algebra over a commutative group stack from Section \ref{gerbes}.
\begin{thm}\label{multiplicative structure}
Consider $T^*_{\PB^{(1)}}(c)^\circ$ with its natural group structure over $\Ht^{(1)}$ as the relative Picard stack of the
generic spectral curve $\Sigma/(\Ht^\circ)^{(1)}$.
Then the Azumaya algebra $\D_{\PB}(\lambda)$ on $T^*_{\PB^{(1)}}(c)^\circ$ has a commutative tensor structure with respect to this
product.
\end{thm}
\noindent
Before we prove Theorem \ref{multiplicative structure}, we note one consequence:
\begin{corollary}\label{local triviality cor}
The $\Gm$-gerbe $\cG_\lambda$ of splittings of $\D_{\PB}(\lambda)$ over $T^*_{\PB^{(1)}}(c)^\circ$, as a group extension of $T^*_{\PB^{(1)}}(c)^\circ$ by $B\Gm$ over $\Ht^{(1)}$, is locally split in the smooth topology of $\Ht ^{(1)}$.
\end{corollary}
\begin{proof}
This is immediate from the discussion preceding Proposition 2.9 on p. 160 of \cite{BB}.
\qed\end{proof}
\noindent
The remainder of this section will be devoted to the proof of Theorem \ref{multiplicative structure}.

We recall the convolution product on $\D(\lambda)$-modules over $G = T^*_\PB(a)^\circ$; this works as follows.  Given two left $\D(\lambda)$-modules
$M_1, M_2$ on $G$, we may form the product $M_1\boxtimes M_2$ on $G\times G$; it is a left $\D_{G\times G}(\lambda\boxtimes\lambda)$-module.
We now want to restrict to $G\times_\Ht G$ and take the (twisted) $\D$-module direct image $m_*$ to obtain a twisted
$\D$-module on $G$ again.
More precisely, it follows from Lemma \ref{det is equivariant}
that the twistings $\lambda\boxtimes\lambda$ and $m^*\lambda$ are canonically isomorphic.
Consider the maps
\bd
\xymatrix{G & \ar[l]_{\hspace{-1em}m} G\times_{\Ht} G \ar[r]^{\Delta} & G\times G}
\ed
relating $G\times G$, the fiber product $G\times_{\Ht}G$, and $G$.  We may form the tensor product of twisted $\D$-modules:
\begin{equation}\label{mult bimodule}
\D_{\on{conv}}(\lambda)  \overset{\on{def}}{=}
\D_{G \leftarrow G\times_\Ht G}(\lambda)  \otimes_{\D_{G\times_\Ht G}(\lambda)} D_{G\times_\Ht G\rightarrow G\times G}(\lambda\boxtimes \lambda).
\end{equation}
See Section \ref{bimodules} for the meaning of the notation.
This is a $(\D_G(\lambda), \D_{G\times G}(\lambda\boxtimes\lambda))$-bimodule on $G\times (G\times G)$.  Given $\D(\lambda)$-modules $M_1$ and
$M_2$, their external product $M_1\boxtimes M_2$ on $G\times G$ is a $\D_{G\times G}(\lambda\boxtimes\lambda)$-module,
and hence we may tensor with the bimodule \eqref{mult bimodule} and apply $m_*$\footnote{Usually one derives both the tensor product and the direct
image, and the result is then a complex of twisted $\D$-modules.  This distinction will not be important for our purposes.}
 to obtain a sheaf on $G$ that is, by construction,
a left $\D(\lambda)$-module, the {\em convolution}
\bd
M_1\ast M_2 = m_*\big(\D_{\on{conv}}(\lambda)\otimes_{\D_{G\times G}(\lambda)}(M_1\boxtimes M_2)\big).
\ed
See \cite[Section~7.6]{BD Hitchin} for a brief discussion of the untwisted case; the twisted
case is equivalent locally to the untwisted one, and all the relevant discussion carries over immediately to our setting.  This defines a
monoidal structure on the (stable  $\infty$-)category of $\D_G(\lambda)$-modules.

The convolution product $M_1\ast M_2$ comes equipped with an associativity constraint which satisfies the pentagon axiom:
see \cite{BD Hitchin} (and \cite{DM} for background on tensor categories).   In fact, the properties of this tensor structure
 actually follow from a corresponding collection of structures on the $\D$-bimodules obtained from \eqref{mult bimodule}.  That is,
given three left $\D(\lambda)$-modules $M_1, M_2, M_3$, the two convolutions $(M_1\ast M_2)\ast M_3$ and $M_1 \ast (M_2\ast M_3)$
are determined by two bimodules, namely
\begin{equation}\label{composite bimodules}
\D_{\on{conv}}(\lambda) \otimes_{\D_{G\times G}(\lambda)} (m\times 1)^* \D_{\on{conv}}(\lambda), \;\;\;
\D_{\on{conv}}(\lambda) \otimes_{\D_{G\times G}(\lambda)} (1\times m)^* \D_{\on{conv}}(\lambda).
\end{equation}
These are $\big(\D_G(\lambda), \D_{G^3}(\lambda^{\boxtimes 3})\big)$-bimodules.  The associativity isomorphism
$(M_1\ast M_2)\ast M_3 \cong M_1\ast (M_2\ast M_3)$ is then given by an isomorphism $I_m$ between the two bimodules in
\eqref{composite bimodules}.  Furthermore, the statement that the pentagon axiom \cite[Diagram~1.0.1]{DM} holds for the convolution
product of $\D(\lambda)$-modules is guaranteed by the corresponding equality on the quadruple product
$G^4$ satisfied by the isomorphism $I_m$ of bimodules.  One similarly obtains a commutativity constraint satisfying the standard
compatibilities.

The next step in the proof of Theorem \ref{multiplicative structure} is to reduce the existence of a tensor structure on the Azumaya
algebra $\D_\PB(\lambda)$ over the Frobenius twist $G^{(1)} = (T^*_{\PB}(a)^\circ)^{(1)} \cong T^*_{\PB^{(1)}}(c)^\circ$
 to the existence of the convolution structure on
twisted differential operators on $G$.  For this, we observe that, from Corollary \ref{D via pullback}, we have that
$\D_\PB(\lambda)$ is equivalent to $(\theta^{(1)})^*\D_{T^*_\PB(a)}(\lambda)$, where $\theta$ is the canonical section of the twisted cotangent
bundle.  Thus, we will want to prove the existence of a tensor structure on $\theta^*\D_{T^*_\PB(a)^\circ}(\lambda)$.

Recall that $m^*\theta^{(1)} = \theta^{(1)}\boxtimes\theta^{(1)}$ as sections of the twisted cotangent bundle
on $T^*_{\PB^{(1)}}(c)^\circ\times_\Ht T^*_{\PB^{(1)}}(c)^\circ$ (Proposition \ref{equality of thetas}).  Given this, we
explain how to conclude that $(\theta^{(1)})^*\D_{T^*_\PB(a)}(\lambda)$ comes equipped with a tensor structure.
This is similar to Lemma 3.16 of \cite{BB}, but we prefer to spell it out in detail.

As we explained in Section \ref{bimodules}, the bimodules $\D_{Z\rightarrow W}(\lambda)$, $\D_{W\leftarrow Z}(\lambda)$ for a
map $f:Z\rightarrow W$ sheafify
over the graph $\Gamma_f$ of $df^{(1)}$ or its ``adjoint'' $\Gamma_f^\dagger$, respectively.  It follows that the tensor product
bimodule
$\D_{\on{conv}}(\lambda)$ in \eqref{mult bimodule} sheafifies over the composite of the correspondences
\begin{multline*}
\Gamma_m^\dagger\circ \Gamma_\Delta  = ((\Gamma_m^\dagger \times (G\times G)^{(1)}\big) \cap (G^{(1)}\times \Gamma_\Delta) \\
\subset T^*_{G^{(1)}}(c)\times_{G^{(1)}} T^*_{(G\times_\Ht G)^{(1)}}(c) \times_{(G\times G)^{(1)}} T^*_{(G\times G)^{(1)}}(c).
\end{multline*}
Moreover, by Proposition \ref{pb equiv}, $\D_{\on{conv}}(\lambda)$ then defines an equivalence between the Azumaya algebras
$\pi_{G}^* \D_G(\lambda)|_{\Gamma_m^\dagger\circ \Gamma_\Delta}$ and
$\pi_{G\times G}^* \D_{G\times G}(\lambda)|_{\Gamma_m^\dagger\circ \Gamma_\Delta}$.  We now pull the Azumaya algebras and
the bimodule back along the map
\bd
\Theta: (G\times_\Ht G)^{(1)} \longrightarrow
T^*_{G^{(1)}}(c)\times_{G^{(1)}} T^*_{(G\times_\Ht G)^{(1)}}(c) \times_{(G\times G)^{(1)}} T^*_{(G\times G)^{(1)}}(c)
\ed
given by $\Theta = \big(\theta\circ m, \theta\boxtimes\theta, (\theta\times\theta) \circ \Delta\big)^{(1)}$.
By Proposition \ref{equality of thetas}, i.e. the equation $m^*\theta^{(1)} = \theta^{(1)}\boxtimes\theta^{(1)}$, it follows that the image of this
map lies in the composite of graphs $\Gamma_m^\dagger\circ \Gamma_\Delta$.
Consequently, the bimodule pulls back to an equivalence between
the Azumaya algebras
$\Theta^*\pi_G^*\D_G(\lambda) = m^*\big((\theta^{(1)})^*\D_G(\lambda)\big)$ and
\begin{multline*}
\Theta^*\pi_{G\times G}^*\D_{G\times G}(\lambda\times\lambda) = (\theta^{(1)})^*\D_G(\lambda)\boxtimes(\theta^{(1)})^*\D_G(\lambda)
= \\
\Delta^*\big(((\theta^{(1)})^*\D_G(\lambda))\boxtimes_{G\times G}((\theta^{(1)})^*\D_G(\lambda))\big)
= ((\theta^{(1)})^*\D_G(\lambda))\boxtimes ((\theta^{(1)})^*\D_G(\lambda)).
\end{multline*}

Analogous calculations prove that the isomorphisms of bimodules \eqref{composite bimodules} pull back to the desired isomorphisms
of bimodules for the Azumaya algebra $(\theta^{(1)})^*\D_G(\lambda)$, and the pentagon condition is immediate from the corresponding
condition for convolution on $G$.  Commutativity is also immediate from the construction since $G$ is a commutative group over
$\Ht$.

Consequently, we have reduced Theorem \ref{multiplicative structure} to Proposition
\ref{equality of thetas}.   This completes the proof of
Theorem \ref{multiplicative structure}.\hfill\qedsymbol

\section{Fourier-Mukai Duality for TDOs}\label{FM duality}
In this section, we first review Fourier-Mukai duality for commutative group stacks.  We then prove the main derived equivalence theorem of the paper.

\subsection{Fourier-Mukai Transform for Commutative Group Stacks}\label{FM general}
A general Fourier-Mukai duality for commutative group stacks has been developed and interpreted in terms of Cartier duality; see
\cite{Laumon2}, Arinkin's appendix to \cite{DP1} (Arinkin attributes the picture explained there to Beilinson) and
Section 2 of \cite{BB}.

Let $\cG$ be a commutative group stack over an irreducible scheme $\Ht$ of finite type over an algebraically closed field $k$.
More precisely, we suppose $\cG$ is a stack locally of finite type over $\Ht$ that is equipped with a structure of commutative group
over $\Ht$.  The {\em Cartier dual commutative group stack} $\cG^\vee$ is, by definition, the stack of group homomorphisms from $\cG$ to $B\Gm$:
\bd
\cG^\vee \overset{\on{def}}{=} \Hom_{\on{gp}}(\cG,B\Gm).
\ed
Equivalently, $\cG^\vee$ is the classifying
stack for
extensions of commutative group stacks
\bd
0\rightarrow \Gm\rightarrow \wt{\cG} \rightarrow \cG\rightarrow 0.
\ed
Alternatively, $\cG^\vee$ may be described as the stack of {\em character line bundles} or {\em geometric characters} on $\cG$
(see Section \ref{character property}).
In nice cases, the Cartier dual group is familiar: for example, the dual of ${\mathbb Z}$ is $B\Gm$ (and the dual of $B\Gm$ is ${\mathbb Z}$).  The dual of an abelian variety $A$ over
$\Ht$ is the dual abelian fibration $A^\vee$ over $\Ht$; these examples and others are discussed in \cite{BB}.

A commutative group stack $\cG$ over $\Ht$ is called {\em very nice} in the terminology of \cite{BB} if, locally in the smooth topology of
$\Ht$, $\cG$ is isomorphic to a finite product of abelian varieties over $\Ht$, finitely generated abelian groups, and copies of $B\Gm$.
In this case, one has the following Fourier-Mukai equivalence for the quasicoherent derived category:
\begin{thm}[See \cite{Laumon2}, \cite{Ar}, or Theorem 2.7 of \cite{BB}]
Let $\cG$ be a very nice commutative group stack and $\cG^\vee$ the dual (very nice) commutative group stack.  Then the Fourier-Mukai
transform induces an equivalence between the quasicoherent derived categories of $\cG$ and $\cG^\vee$.
\end{thm}
Suppose now that $\cG$ is a commutative group stack over $\Ht$ that is an extension of commutative groups:
\bd
0\rightarrow B\Gm\rightarrow \cG\rightarrow G\rightarrow 0
\ed
for a group stack $G$ that locally (on $\Ht$) takes the form ${\mathbb Z}^r\times A\times B\Gm^s$ for an abelian variety $A/\Ht$ and
some nonnegative integers $r$ and $s$.
Then $\cG^\vee$ is itself an extension
\begin{equation}\label{group vs. torsor}
0\rightarrow G^\vee\rightarrow \cG^\vee\xrightarrow{\pi_{\cG}} {\mathbb Z}\rightarrow 0.
\end{equation}
Moreover, such group extensions \eqref{group vs. torsor} correspond exactly to $G^\vee$-torsors over $\Ht$ via the correspondence
\bd
\cG^\vee \leftrightarrow \cG^\vee_1,
\ed
where $\cG^\vee_1 = \pi^{-1}_{\cG}(1)$ is the degree $1$ component of $\cG^\vee$, i.e.
 the inverse image of $1\in {\mathbb Z}$; see
Section 2.8 of \cite{BB} for more details and the properties of such extensions.  The Cartier dual $\cG^\vee$ is again a
very nice commutative group stack.  One then has:
\begin{prop}[Proposition 2.9 of \cite{BB}]\label{FM equiv}
The Fourier-Mukai equivalence between $\cG$ and $\cG^\vee$ restricts to a derived equivalence between
$D^b(\cG^\vee_1)$ and $D^b(\cG)_1$, where the latter is the ``weight one component" of the derived category of the $\Gm$-gerbe
$\cG$ over $G$.
\end{prop}

Let us note also the following description of $\cG^\vee_1$.  Suppose that $\cG \cong G\times B\Gm$: by Corollary \ref{local triviality
cor}, this is true,  locally in the smooth topology of $\Ht$, for our gerbe  $\cG_\lambda$ of splittings of $\D_\PB(\lambda)$.
Then a choice of a group stack homomorphism $\phi: G\times B\Gm\rightarrow B\Gm$ that lies in the component $\cG^\vee_1$ is a choice
of homomorphism that restricts to an isomorphism on $B\Gm$; in particular, it induces a choice of splitting, $\cG = G\times B\Gm$.
In the case that $\cG = \cG_\lambda$ is the gerbe of splittings of the Azumaya algebra $\D_\PB(\lambda)$, such a choice gives a splitting
of $\D_\PB(\lambda)$.  We let $\cE_\phi$ denote the splitting module associated to the group stack homomorphism $\phi$.

\subsection{Main Equivalence Theorem}\label{main equiv section}
We have a commutative group stack
$G = \on{Pic}(\Sigma/\Ht\gen)^{(1)} = T^*_{\PB^{(1)}}(c)^\circ$ over $\Ht^{(1)}$, the relative Picard stack of line bundles on the generic
spectral curve.  By Theorem \ref{multiplicative structure}, the restriction of $\D_\PB(\lambda)$ to $G$ is an
Azumaya algebra $A = \D_{\PB}(\lambda)$ that comes equipped with
a tensor structure over $G$.  Let $\cG_\lambda$ denote its gerbe of splittings.
\begin{defn}
We let $D(\D_\PB(\lambda))\gen$ denote the (quasicoherent) derived category of the Azumaya algebra $A$ on $G$.
\end{defn}

We are now ready to prove:
\begin{thm}\label{main equiv thm}
The stack $\PLoc_n^\lambda(X, \overline{\ell})\gen$ is isomorphic, as a $\on{Pic}(\Sigma/\Ht^\circ)^{(1)}$-torsor over $(\Ht^\circ)^{(1)}$, to the degree
$1$ component $(\cG_\lambda)^\vee_1$ of the Cartier dual $\cG_\lambda^\vee$ of the gerbe $\cG_\lambda$ of splittings of $\D_{\PB}(\lambda)$ over
$(T^*_{\PB}(\lambda)^\circ)^{(1)}$.
\end{thm}
\begin{proof}
Let
$G = \on{Pic}(\Sigma/\Ht\gen)^{(1)}$.  By Theorem \ref{multiplicative structure}, the Azumaya algebra $A = \D_{\PB}(\lambda)$ on $G$
comes equipped with
a tensor structure over $G$.  Let $\cG_\lambda$ denote its gerbe of splittings.  Corollary \ref{local triviality cor} tells us that
$\cG_\lambda$ is split as a commutative group extension of $G$ by $B\Gm$, locally in the smooth topology of $\Ht^{(1)}$.

In the rest of the proof, we will frequently omit notation for Frobenius twists.
To each smooth $\Ht\gen$-scheme $T\rightarrow \Ht\gen$, and each choice of
$\phi\in(\cG_\lambda)^\vee_1$ over $T$, we will associate an object $\cL(\phi)$
of $\PLoc_n^\lambda(X)$ parametrized by $T$.  We will see that the assignment $\phi\mapsto \cL(\phi)$ is $G^\vee$-equivariant (and it
will be evidently functorial).  This gives the desired isomorphism $(\cG_\lambda)^\vee_1 \cong \PLoc_n^\lambda(X)$.

As we discussed following Proposition \ref{FM equiv}, $\phi$ gives a splitting module $\cE_\phi$ of $\D_\PB(\lambda)$.  Our construction
of $\cL(\phi)$ comes in two steps:
\begin{enumerate}
\item $\cE_\phi$ determines a splitting of $\D_X(\lambda)$ over $\Sigma\setminus\{b\}$.
\item This splitting extends $G^\vee$-equivariantly to $\Sigma$.
\end{enumerate}
\noindent
We obtain (1) from a rather long chain of equivalences.  This starts from the equivalence $\D_\PB(\lambda)\simeq (\theta_{\PB}^{(1)})^*\D_{T^*_\PB(a)}(\lambda)$
implied by Corollary \ref{D via pullback}.   Pulling this equivalence back along the Abel-Jacobi map
\bd
AJ_\Sigma: \Sigma\setminus\{b\} \hookrightarrow T^*_\PB(a),
\ed
or more precisely its Frobenius twist (which we also denote by $AJ_\Sigma$), and taking note that the twists are compatible by
Lemma \ref{AJ pullback of det}, we get:
\begin{equation}\label{arriving at sigma}
AJ_\Sigma^*\D_\PB(\lambda) \simeq AJ_\Sigma^*\circ(\theta_{\PB})^*\D_{T^*_\PB(a)}(\lambda) \simeq (AJ_\Sigma^*\theta_\PB)^*\D_{T_\Sigma^*(a)}(\lambda).
\end{equation}
Here the right-hand equivalence of \eqref{arriving at sigma} follows from Lemma \ref{pullback and pullback}.

We now observe:
\begin{lemma}\label{equiv sublemma}
$AJ_\Sigma^*\theta_\PB = i^*\theta_X$, where $i:\Sigma\setminus\{b\}\rightarrow \overline{T_{X^{(1)}}^*(c)}$ is our natural map.
\end{lemma}
\begin{proof}[Proof of Lemma]
This follows from \eqref{multiplicative calc} (in the case $r=1$) using Corollary \ref{pullback becomes zero}.
\qed\end{proof}
We then have
\begin{equation}\label{many equivs}
(AJ_\Sigma^*\theta_\PB)^*\D_{T_\Sigma^*(a)}(\lambda) \simeq (i^*\theta_X)^*\D_{T_\Sigma^*(a)}(\lambda)
\simeq i^*\theta_X^*\D_{T_X^*(a)}(\lambda) \simeq i^*\D_X(\lambda),
\end{equation}
where the first equivalence follows from Lemma \ref{equiv sublemma}, the second follows from Lemma \ref{pullback and pullback}, and
the third follows from Corollary \ref{D via pullback}.
Combining Equivalences \eqref{arriving at sigma} and \eqref{many equivs}, we obtain:
\begin{equation}\label{AJ equiv}
AJ_\Sigma^*\D_\PB(\lambda)\simeq i^*\D_X(\lambda).
\end{equation}
Returning to step (1), it is then immediate from \eqref{AJ equiv} that a choice of splitting module $\cE_\phi$ of $\D_\PB(\lambda)$
over $T\times_\Ht G$ gives a splitting module $\wt{\cL}(\cE_\phi)$ of $i^*\D_X(\lambda)$.

To complete the proof, Lemma \ref{parabolic is minimal rk} guarantees that we need only to extend $\wt{\cL}(\cE_\phi)$ to an
$R_X$-module $\cL(\cE_\phi)$ on $\Sigma$.  The existence is guaranteed by Proposition \ref{extension lemma}, but, because of
the nonuniqueness explained in part (2) of that proposition---namely, the extension is only unique up to a shift in $\qgr\,\cR$---we must
make a coherent choice in order to guarantee the $G^\vee$-equivariance we need.  To do this, recall that, as discussed in Section
\ref{det twisting}---specifically, Lemma \ref{pbs and pbs for cot}---we have a section $G^0 = T^*_{\Pb^{(1)}}(c)\rightarrow T^*_{\PB^{(1)}}(c)=G$ of the projection
$T^*_{\PB^{(1)}}(c)\rightarrow T^*_{\PB^{(1)}}(c)/B\Gm = T^*_{\Pb^{(1)}}(c)$, which splits the group stack $T^*_{\PB^{(1)}}(c)^\circ\rightarrow \Ht$ as a product of
$T^*_{\Pb^{(1)}}(c)^\circ\rightarrow \Ht$ and $B\Gm$.   Moreover, it follows that this splits the gerbe $\cG_\lambda$ of
splittings:
\bd
\cG_\lambda = \cG^0_\lambda \times B\Gm\rightarrow T^*_{\Pb^{(1)}}(c)^\circ\times B\Gm = T^*_{\PB^{(1)}}(c)^\circ.
\ed
It follows that $\cG_\lambda^\vee$ is split as $(\cG^0_\lambda)^\vee \times {\mathbb Z}$, and that this copy of ${\mathbb Z}$ is
identified under the exact sequence \eqref{group vs. torsor} with the natural copy of ${\mathbb Z}$ in
$G^\vee = \on{Pic}(\Sigma/\Ht\gen)$.

To make a coherent choice of $\cL(\cE_\phi)$, then, we first choose a component $\PL_n^\lambda(X)_\ell$; this amounts to choosing
a degree $\ell$ for the vector bundles on $X$ underlying mirabolic local systems.  Then, given a splitting module $\cE_\phi$ coming from
a choice of $\phi\in (\cG^0_\lambda)^\vee_1 \subset (\cG_\lambda)^\vee_1$, we extend the splitting module
$\wt{\cL}(\cE_\phi)$ of $i^*\D_X(\lambda)$ to an $R$-module
$\cL(\cE_\phi)$ on $\Sigma$ that lies in the component $\PL_n^\lambda(X)_\ell$: in other words, thinking of this $R$-module as a sheaf
on $\overline{T^*_X(\lambda)}$, its direct image to $X$ should have degree $\ell$.  Such a choice exists and is unique once we have fixed
$\ell$ by Proposition \ref{extension lemma}.  This defines our functor
\bd
\cL: (\cG^0_\lambda)^\vee_1 \longrightarrow \PLoc_n^\lambda(X,\overline{\ell})_\ell
\ed
(i.e. to the $\ell$th component).
This functor is clearly equivariant for the natural action of $(G^0)^\vee$, by tensoring by line bundles pulled back from $G^0$.  It then
has a unique extension to a $G^\vee$-equivariant functor (where $G^\vee = (G^0)^\vee\times {\mathbb Z}$)
\begin{equation}\label{isomorphism L}
\cL: (\cG_\lambda)^\vee_1\longrightarrow \PLoc_n^\lambda(X, \overline{\ell}).
\end{equation}
This completes the proof.
\qed\end{proof}
We will omit notation for Frobenius twists in the remainder of this section.
Let $\cP$ denote the Poincar\'e sheaf on $\cG^\vee_\lambda\times_\Ht \cG_\lambda$.   By Theorem \ref{main equiv thm}, the stack 
$\PLoc_n^\lambda(X,\overline{\ell})^\circ$ is isomorphic to the degree $1$ component of the Cartier dual $\cG_\lambda^\vee$ to the
gerbe $\cG_\lambda$ of splittings of $\D_\PB(\lambda)$.  On the other hand, by Lemma \ref{weight 1 lemma}, the weight 1 component of the derived category of $\cG_\lambda$ is equivalent to $D(\D_{\PB}(\lambda))^\circ$.
Proposition \ref{FM equiv} then implies that the Fourier-Mukai transform $\Phi^\cP$
restricts to a functor
\bd
\Phi = \Phi^\cP:
 D_{\on{qcoh}}\big(\PLoc_n^\lambda(X,\overline{\ell})^\circ\big)
 \longrightarrow D(\D_{\PB}(\lambda))^\circ.
 \ed
  Slightly abusively, we let
$\cP^\vee$ denote the ``adjoint'' Fourier-Mukai kernel.  Then:
 \begin{corollary}\label{FM equiv cor}
We have mutually quasi-inverse equivalences of derived categories:
\bd
\xymatrix{
D(\D_{\PB}(\lambda))^\circ \ar@<.5ex>[r]^{\hspace{-2em}\Phi^{\cP^\vee}} &   D_{\on{qcoh}}\big(\PLoc_n^\lambda(X,\overline{\ell})^\circ\big)
\ar@<.5ex>[l]^{\hspace{-2em}\Phi^\cP}.}
\ed
\end{corollary}
\begin{proof}
This follows from Theorem \ref{main equiv thm} by Proposition \ref{FM equiv}.
\qed\end{proof}

\subsection{Hecke Operators}\label{final Hecke section}
Recall from Definition \ref{Hecke stacks} the definition of the Hecke correspondences $\on{Hecke}_r$.  We define the Hecke functor
\bd
{\mathbf H}_r: D(\D_{\PB}(\lambda))^\circ\rightarrow D(\D_{\PB}(\lambda)\boxtimes \D_X(\lambda))^\circ
\ed
by $M\mapsto (q_1)_*q_2^* M$ (here the superscript $\circ$ on the right-hand side again means that we restrict to modules supported
on the generic locus of the $\PB$ factor in the product).

 We define tensor-product functors ${\mathbf T}_r$, $1\leq r\leq n$,
\bd
{\mathbf T}_r: D_{\on{qcoh}}\big(\PLoc_n^\lambda(X,\overline{\ell})^\circ\big)\longrightarrow
D_{\on{qcoh}}\big(\theo_{\PLoc_n^\lambda(X,\overline{\ell})^\circ}\boxtimes\D_X(\lambda)\big)
\ed
as follows.  Let $\pi_1: \PLoc_n^\lambda(X,\overline{\ell})^\circ \times X
\rightarrow \PLoc_n^\lambda(X,\overline{\ell})^\circ$ denote projection on the first factor.  Let $\cL_{\on{univ}}$ denote the universal mirabolic
local system on $\PLoc_n^\lambda(X,\overline{\ell})^\circ \times X$: by this, we mean the following.  Thinking of a mirabolic local system in terms
of the corresponding object of $\on{Qgr}\,\cR$, we may localize to a $\D_X(\lambda)$-module---this corresponds to restricting to the open subset 
$\Sigma\smallsetminus\{b\}$ of the spectral curve $\Sigma$ and pushing forward to $X$.  This sheaf, as an $\theo$-module, is what we denote by
$\cL_{\on{univ}}$.  
We then let
\bd
{\mathbf T}_r(M) = \pi_1^*M\otimes \wedge^r\cL_{\on{univ}}.
\ed

We now have the ``Hecke eigenvalue'' or ``spectral decomposition'' property of the functors $\Phi^\cP$, $\Phi^{\cP^\vee}$:
\begin{thm}\label{hecke property}
For $1\leq r\leq n$, we have
$\Phi^{\cP^\vee}\circ {\mathbf H}_r \simeq {\mathbf T}_r \circ \Phi^{\cP^\vee}$.
\end{thm}
\noindent
We explain the proof for $r=1$; more precisely, we explain why
$\Phi^{\cP^\vee}\circ {\mathbf H}_1 \circ \Phi^\cP \simeq {\mathbf T}_1$.  The proof for $1<r\leq n$ can be copied from \cite{BB} with similar details.

Let us first make an observation that is implicit in Section \ref{Hecke sub 2} and, especially, Remark \ref{hecke operators remark}.
Let $M: (\Sigma\smallsetminus\{b\})\times_\Ht G\xrightarrow{M} G$ denote the multiplication map; here $M$ is identified with the
composite $Z_1\xrightarrow{a} G\times_\Ht G \xrightarrow{m} G$ under Lemma \ref{Z identification}.  The observation implicit
in Remark \ref{hecke operators remark} is that the diagram \eqref{Hecke diagram} is an instance of Lemma \ref{support of Hecke}:
that is, $Z_1$ is an open subset of the fiber product of $j_1$ and $j_2$ in Diagram \eqref{Hecke diagram}.  It then follows from Lemma
\ref{support of Hecke} that the Hecke functor ${\mathbf H}_1= (q_1)_*q_2^*$ is given by $M^*$: that is, if ${\mathcal S}$ is a
$\D_\PB(\lambda)$-module living over $G$, then ${\mathbf H}_1({\mathcal S})$ is naturally identified with $M^*{\mathcal S}$, which
is a $\D_{X\times\PB}(\lambda)$-module supported on $(\Sigma\smallsetminus\{b\})\times_\Ht G\subset T^*_{X\times\PB}(\lambda)$
(recall that we are omitting notation for Frobenius twists).

We next want to reformulate slightly the content of Corollary \ref{FM equiv cor} in terms of the construction of the isomorphism
\eqref{isomorphism L}.  Consider the diagram
\begin{equation}\label{big hecke picture}
\xymatrix{
 & \PLoc\times_\Ht  (\Sigma\smallsetminus\{b\}) \times_\Ht G \ar[ld]_{pr_1} \ar[r]^{\hspace{1em}\pi_{23}} \ar[d]^{1\times M}
 &
 (\Sigma\smallsetminus\{b\})\times_\Ht G\ar[d]^{M}\\
 \PLoc & \PLoc\times_\Ht G \ar[l]_{\pi_1} \ar[r]^{\pi_2} & G.}
 \end{equation}
The Poincar\'e sheaf $\cP$ may  be understood as an $\theo\boxtimes\D(\lambda)$-module on $\PLoc\times_\Ht G$.
Pulling back along $1\times M$, we get an $\theo_{\PLoc}\boxtimes \D_{X\times\PB}(\lambda)$-module
$(1\times M)^*\cP$.  It follows from the previous paragraph that
\begin{equation}\label{hecke 1 id}
(1\times M)^*\cP \cong ({\mathbf 1}\times {\mathbf H}_1)(\cP).
\end{equation}

On the other hand, it is immediate from the construction of the isomorphism $\cL$ of
\eqref{isomorphism L} that the restriction of $(1\times M)^*\cP$ to
\bd
\PLoc\times_\Ht (\Sigma\smallsetminus\{b\}) \times_\Ht u(\Ht) \cong \PLoc \times_\Ht (\Sigma\smallsetminus\{b\})
\ed
is exactly the restriction to $\Sigma\smallsetminus\{b\}$ of the universal twisted local system $\cE$ on $\PLoc \times_\Ht\Sigma$.
Moreover, note that the pullback of $\cP$ along $1\times m: \PLoc\times_\Ht G\times_\Ht G\rightarrow \PLoc \times_\Ht G$ is
$\pi_{12}^*\cP\times \pi_{13}^*\cP$ (this is the character property of $\cP$); it follows that
\begin{equation}\label{universal sheaf}
(1\times M)^*\cP\cong \pi_{12}^*\cE\otimes \pi_{13}^*\cP
\end{equation}
 over $\PLoc_n^\lambda(X,\overline{\ell})^\circ \times_\Ht (\Sigma\smallsetminus\{b\})\times_\Ht G$.

We are now ready to explain Theorem \ref{hecke property} for $r=1$.
The functor $\Phi^{\cP}$ is given by ${\mathcal S}\mapsto (\pi_2)_*\big(\cP\otimes\pi_1^*{\mathcal S}\big)$.  By the above discussion,
${\mathbf H}_1\circ \Phi^\cP$ is given by
${\mathcal S}\mapsto M^*(\pi_2)_*(\cP\otimes \pi_1^*{\mathcal S})$.  By flat pullback,
$M^*(\pi_2)_* = (\pi_{23})_*(1\times M)^*$, so
\begin{equation}\label{intermediate step for H1}
{\mathbf H}_1\circ\Phi^\cP = (\pi_{23})_*(1\times M)^*(\cP\otimes \pi_1^*(-))
= (\pi_{23})_*\left[\left((1\times M)^*\cP\right) \otimes pr_1^*(-)\right].
\end{equation}
Substituting \eqref{universal sheaf} into \eqref{intermediate step for H1} gives
\begin{equation}\label{penultimate}
{\mathbf H}_1\circ\Phi^\cP=
(\pi_{23})_*\left[(\pi_{12}^*\cE\otimes \pi_{13}^*\cP) \otimes pr_1^*(-)\right].
\end{equation}
  Now, composing with $\Phi^{\cP^\vee}$ to give
$\Phi^{\cP^\vee}\circ {\mathbf H}_1\circ\Phi^\cP$ has the effect of cancelling the factor of $\pi_{13}^*\cP$
in \eqref{penultimate}, and reduces $\Phi^{\cP^\vee}\circ {\mathbf H}_1\circ\Phi^\cP$ to ${\mathbf T}_1$ as desired.

A similar computation relying on Lemma 5.8 of \cite{BB}, gives the claim for $r>1$.

\bibliographystyle{alpha}

\end{document}